\documentclass{article}

\usepackage{amsmath,amssymb,amsthm}
\usepackage[utf8]{inputenc}
\usepackage[T1]{fontenc}
\usepackage[a4paper]{geometry}
\usepackage{datetime}
\usepackage[export]{adjustbox}
\usepackage[dvipsnames]{xcolor}
\usepackage{fancyhdr}
\usepackage{setspace}
\usepackage{url}
\usepackage{bm}
\usepackage{scrextend}
\usepackage{amsfonts}
\addtokomafont{labelinglabel}{\ttfamily}
\usepackage{float}
\usepackage{tocloft}
\usepackage{caption}
\usepackage{pgfplots} 
\usepackage{chngcntr}
\usepackage{scalerel}
\counterwithin{figure}{section}

\pgfplotsset{compat=newest} 
\pgfplotsset{plot coordinates/math parser=false} 
\newlength\figureheight 
\newlength\figurewidth

\usepackage{authblk}
\title{Stable approximation of Helmholtz solutions\\
in the 3D ball using evanescent plane waves}
\author[1,*]{Nicola Galante}
\author[2]{Andrea Moiola}
\author[3,*]{Emile Parolin}
\affil[1]{Alpines, Inria, Paris, France, \texttt{nicola.galante@inria.fr}}
\affil[2]{University of Pavia, Pavia, Italy, \texttt{andrea.moiola@unipv.it}}
\affil[3]{Alpines, Inria, Paris, France, \texttt{emile.parolin@inria.fr}}
\affil[*]{Laboratoire Jacques-Louis Lions, Sorbonne Université, Paris, France}

\usepackage[sorting=nyt,minbibnames=5,maxbibnames=5,giveninits=true,url=false,isbn=false,eprint=true,doi=false]{biblatex}
\bibliography{biblio}
\usepackage{hyperref}
\hypersetup{
    linktocpage=true,
    bookmarksopen=true,
    colorlinks=true,
    linkcolor=blue,
    citecolor=blue,
    urlcolor=black,
    bookmarksopenlevel=5, 
}

\usepackage[shortlabels]{enumitem}

\newcommand\numberthis{\addtocounter{equation}{1}\tag{\theequation}}

\makeatletter
\newcommand{\rightleftarrowss}[2]{
  \mathrel{\mathop{
    \vcenter{\offinterlineskip\m@th
      \ialign{\hfil##\hfil\cr
        \hphantom{$\scriptstyle\mspace{8mu}{#1}\mspace{8mu}$}\cr
        \rightarrowfill\cr
        \vrule height0pt width 2em\cr
        \leftarrowfill\cr
        \hphantom{$\scriptstyle\mspace{8mu}{#2}\mspace{8mu}$}\cr
        \noalign{\kern-0.3ex}
      }
    }
  }\limits^{#1}_{#2}}
}
\makeatother

\newcommand{\Tdisc}[2]{\mathrm{T}_{#1}^{\scaleto{\textup{#2}\mathstrut}{5pt}}}
\newcommand{\Tcont}[2]{\boldsymbol{\mathrm{T}}_{#1}^{\scaleto{\textup{#2}\mathstrut}{5pt}}}

\usepackage{graphicx,caption,subcaption}

\theoremstyle{plain}
\newtheorem{theorem}{Theorem}[section]
\newtheorem{proposition}[theorem]{Proposition}
\newtheorem{lemma}[theorem]{Lemma}
\newtheorem{corollary}[theorem]{Corollary}
\newtheorem{definition}[theorem]{Definition}
\newtheorem{remark}[theorem]{Remark}
\numberwithin{equation}{section}

\usepackage{bookmark} 
\begin{document}

\maketitle

\begin{abstract}
The goal of this paper is to show that evanescent plane waves are much better at numerically approximating Helmholtz solutions than classical propagative plane waves.
By generalizing the Jacobi--Anger identity to complex-valued directions, we first prove that any solution of the Helmholtz equation on a three-dimensional ball can be written as a continuous superposition of evanescent plane waves in a stable way.
We then propose a practical numerical recipe to select discrete approximation sets of evanescent plane waves, which exhibits considerable improvements over standard propagative plane wave schemes in numerical experiments.
We show that all this is not possible for propagative plane waves: they cannot
stably represent general Helmholtz solutions, and any approximation based on discrete sets of propagative plane waves is doomed to have exponentially large coefficients and thus to be numerically unstable.
This paper is motivated by applications to Trefftz-type Galerkin schemes and extends the recent results in [Parolin, Huybrechs and Moiola, M2AN, 2023] from two to three space dimensions.
\end{abstract}



\clearpage
\vspace*{-1.9cm}
\tableofcontents

\section{Introduction}
\label{s:intro}

The homogeneous Helmholtz equation
\begin{equation}
-\Delta u - \kappa^2 u=0,
\label{Helmholtz equation}
\end{equation}
where $\kappa>0$ is a real parameter called \emph{wavenumber}, finds extensive
application in diverse scientific and engineering fields, including acoustics,
electromagnetics, elasticity, and quantum mechanics.
Linked to the scalar wave equation, it characterizes the spatial dependence of
time-harmonic solutions.

In high-frequency settings where the wavelength $\lambda=2\pi/\kappa$ is much
smaller than the domain scale, approximating Helmholtz solutions is 
complex and computationally expensive, as their oscillatory nature demands
numerous degrees of freedom (DOFs) for accuracy and leads to dispersion errors.
Trefftz methods~\cite{hiptmair-moiola-perugia1} tackle these issues by using particular solutions of the PDE as spanning elements, thereby reducing both the number of DOFs required and dispersion compared to polynomial-based schemes~\cite{Deraemaeker1999,Gittelson2014}.
These properties make Trefftz methods effective for a wide range of real-world problems, such as inverse scattering \cite{colton-kress}, invisibility cloaking \cite{Greenleaf2009}, large-scale electromagnetic simulations \cite{Sirdey2022}, and sound field reconstruction from sparse acoustic measurements \cite{Chardon20214}.
Examples of Trefftz formulations include the Wave-Based Method (WBM) \cite{Desmet1998,Deckers2014}, the Variational Theory of Complex Rays (VTCR) \cite{Riou2011}, the Discontinuous Enrichment Method (DEM) \cite{Massimi2008}, the Trefftz Discontinuous Galerkin (TDG) methods \cite{Gittelson2009,Gittelson2014}, and the Ultra Weak Variational Formulation (UWVF) \cite{Cessenat1998}; these and others are systematically reviewed and compared in \cite{hiptmair-moiola-perugia1}.

Propagative plane waves (PPWs) $\mathbf{x} \mapsto e^{i\kappa
\mathbf{d}\cdot \mathbf{x}}$, where $\mathbf{d} \in \mathbb{R}^n$ with
$\mathbf{d} \cdot \mathbf{d}=1$, form an appealing family of Trefftz basis functions as they offer
efficient implementation due to the possibility for 
closed-form integration on
flat sub-manifolds \cite[sect.~4.1]{hiptmair-moiola-perugia1}.
However, ill-conditioning emerges in linear systems for high-resolution Trefftz spaces, leading to strong numerical instability and stalled convergence
when using floating-point arithmetic \cite{Barucq2021,huybrechs3,Luostari2013}.
As a result, the convergence predicted by the approximation theory \cite[sect.~4.3]{hiptmair-moiola-perugia1} cannot be achieved in practice.

\paragraph{Recent results in 2D}
The study in~\cite{parolin-huybrechs-moiola} makes advances in the analysis of PPW instability.
Using recent progress in frame approximation theory~\cite{huybrechs1,huybrechs2}, this work argues that, in floating-point arithmetic and in presence of ill-conditioning, 
to obtain accurate approximations it is not enough to prove the existence of a discrete function with small approximation error, but a representation with bounded coefficients is needed.
It turns out that large coefficients are unavoidable when considering approximations in the form of linear combinations of PPWs if the Helmholtz solution contains high-frequency Fourier modes~\cite[Th.~4.3]{parolin-huybrechs-moiola}.
This highlights a key trade-off: while PPWs are highly effective in low-accuracy regimes -- where a limited number of directions already yields useful results in practice -- their performance deteriorates as one targets higher precision.
In such settings, numerical instabilities arising from ill-conditioning and large coefficient norms can lead to convergence stagnation, even when regularization techniques are employed. These insights motivate alternative strategies to achieve accurate and stable numerical representations.

The work~\cite{parolin-huybrechs-moiola} then proposes a remedy.
For accurate, bounded-coefficient approximations, the key idea is to enrich the
approximation sets with \emph{evanescent plane waves} (EPWs).
These Helmholtz solutions allow for simple and cost-effective implementations,
maintaining the form
$\mathbf{x} \mapsto e^{i\kappa \mathbf{d}\cdot \mathbf{x}}$ with a
complex-valued direction $\mathbf{d} \in \mathbb{C}^n$ satisfying
$\mathbf{d}\cdot\mathbf{d}=1$.
Such a wave oscillates with period shorter than the Helmholtz wavelength $\lambda$
in the propagation direction $\Re(\mathbf{d})$, and decays exponentially in the orthogonal evanescent direction $\Im(\mathbf{d})$. 
Modal analysis reveals that EPWs effectively approximate high Fourier modes, filling the gap left by PPWs.

This is backed by \cite[Th.~6.7]{parolin-huybrechs-moiola}, which establishes
that any Helmholtz solution on the unit disk can be uniquely expressed as a
continuous superposition of EPWs.
The operator that maps the associated density to the Helmholtz solution is
called \emph{Herglotz transform} \cite[Def.~6.6]{parolin-huybrechs-moiola}
and admits a continuous inverse, so that the density is bounded in a weighted
$L^{2}$ norm, indicating a form of stability at the continuous level.
For applications, the difficulty then lies in identifying effective EPW sets with moderate size that retain both accuracy and stability.
The construction in \cite[sect.~7]{parolin-huybrechs-moiola}, based on
\cite{Cohen_Migliorati,Hampton,Migliorati_Nobile}, proposes a simple recipe
that exhibits a significant improvement over conventional PPW methods in numerical experiments.

EPWs also feature in the Wave Based Method~\cite{Deckers2014} and have shown particular effectiveness for interface problems in Trefftz schemes~\cite{Luostari2013,Massimi2008}.
They have further been employed in~\cite{Chapman2024} to approximate particular 3D Helmholtz solutions, written in cylindrical and spherical coordinates, in selected regions of space -- using locally a single EPW.

\paragraph{Extension to 3D}
This paper presents the challenging extension of \cite{parolin-huybrechs-moiola} to the 3D setting and is mainly based on the Master thesis of the first author~\cite{galante}.
It focuses on spherical domains in order to yield explicit theoretical results via modal analysis.
Up to rescaling the wavenumber $\kappa$, we consider the Helmholtz equation posed on the open unit ball 
$B_1:=\{\mathbf{x} \in \mathbb{R}^3: |\mathbf{x}| <1\}$.

In section \ref{sec:evanescent plane waves}, we define and study 3D EPWs.
A first non-trivial challenge is the parametrization of the complex direction set
$\{\mathbf{d} \in \mathbb{C}^3 : \mathbf{d} \cdot \mathbf{d}=1\}$.
Our approach involves defining a complex-valued reference direction and then consider its rigid-body rotations via Euler angles.
We then prove a new generalized Jacobi--Anger identity for complex-valued directions in Theorem~\ref{Theorem 2.10}, i.e.\ the Fourier expansion of EPWs on the spherical wave basis.
This requires extending the Ferrers functions (appearing in the definition of spherical harmonics) to arguments outside the usual domain $[-1,1]$, and the use of Wigner D-matrices.
We discuss EPW modal analysis revealing that EPWs effectively encompass high
Fourier regimes, unlike the propagative case.

Section \ref{sec:herglotz transform} introduces a notion of ``stable continuous approximation'', which essentially entails approximating Helmholtz solutions by continuous superpositions of the elements of a given Bessel family (indexed by a continuous parameter); stability follows from the boundedness of the associated density. 
Analogously to what was done in 2D, we then prove in Theorem \ref{Theorem 3.9} that the EPW family provides such a stable continuous approximation.
We call ``Herglotz transform'' the isomorphism mapping densities to Helmholtz solutions.
In fact, in the parlance of frame theory, EPWs are shown to form a continuous frame for the Helmholtz solution space.
In contrast, PPWs cannot provide such stable continuous approximations, as proved in Theorem~\ref{Thm:PPW-SCA}.

Section \ref{sec:stable numerical approximation} presents the corresponding notion of ``stable discrete approximation'' with finite 
expansions associated to bounded coefficients.
A sampling-based scheme relying on regularized Singular Value Decomposition and oversampling is then presented.
We prove in Corollary~\ref{Corollary 4.3} that this scheme yields accurate numerical solutions in finite-precision arithmetic, provided the approximation set enjoys the stable discrete approximation property and
suitable sampling points are chosen.
Theorem~\ref{Theorem 4.5} shows that PPWs are unstable: some Helmholtz solutions can be approximated by linear combinations of PPWs only if exponentially large coefficients are present.

Section \ref{sec:numerical recipe} presents a numerical recipe that mirrors \cite[sect.~7]{parolin-huybrechs-moiola}, drawing inspiration from optimal sampling techniques \cite{Cohen_Migliorati}.
In practice, it selects an EPW basis by sampling the parametric domain according to an explicit probability density function.
While such a construction exhibits experimentally the desired properties, a
full proof that it satisfies the stable discrete approximation requirements is
yet to be established.

Section \ref{sec:numerical results} showcases several numerical experiments supporting the choice of using EPWs for approximating Helmholtz solution in
3D\footnote{The MATLAB code used to generate the numerical results of this paper is available at\\
\url{https://github.com/Nicola-Galante/evanescent-plane-wave-approximation}.}.
Our EPW sets significantly outperform standard PPW schemes, and also behave well on different geometries, despite being grounded in unit ball analysis.
Additionally, they appear to maintain near-optimality: the DOF budget required to approximate the first $N$ modes scales linearly with $N$, for a fixed level of accuracy.
These results provide strong evidence of 
the potential of the proposed numerical approach for EPW approximations and Trefftz methods.

Table~\ref{t:notation} summarizes the symbols used throughout the paper.

\begin{table}[htbp]
\begin{tabular}{|l|l|l|}\hline
$\kappa,\lambda$ & wavenumber and wavelength & \S\ref{s:intro}\\
$B_1,\mathbb S^2$ & unit ball and sphere in $\mathbb R^3$& \S\ref{s:intro}, \eqref{propagative direction}\\
$\Theta,Y$ & parameter domains & Def.~\ref{def:EPW}\\
$\boldsymbol{\theta},\psi,\zeta,z,\mathbf y$ & EPW parameters & Def.~\ref{def:EPW}\\
$R_{\boldsymbol\theta,\psi},R_y(\theta),R_z(\theta)$ & rotation matrices & Def.~\ref{def:EPW}\\
$\textup{EW}_{\mathbf{y}}$ & evanescent plane wave &\eqref{evanescent wave}\\
$\mathbf d_\uparrow(z),\mathbf d(\mathbf y)$ & EPW direction vectors & \eqref{complex direction}\\
$\textup{PW}_{\boldsymbol{\theta}},\mathbf d(\boldsymbol\theta)$ & propagative plane wave and direction & \eqref{eq:PPW_definition}, \eqref{propagative direction} 
\\
$\mathcal{I}$ & spherical wave index set & \S\ref{ss:Spherical}\\
$\mathsf{P}_{\ell}^m$ & Ferrers functions, Legendre polynomials &\eqref{legendre polynomials}\\
$Y_\ell^m,\gamma_\ell^m$ &  spherical harmonics &\eqref{spherical harmonic}\\
$J_\ell,j_\ell$ & Bessel and spherical Bessel functions &\S\ref{ss:Spherical}\\
$\tilde b_\ell^m,b_\ell^m,\beta_\ell^m$ & spherical waves and normalization &\eqref{b tilde definizione}\\
$\mathcal B,\|\cdot\|_{\mathcal B},(\cdot,\cdot)_{\mathcal B}$ & Helmholtz solution space & \eqref{def:B}\\
$P_\ell^m$&associated Legendre functions&\eqref{legendre2 polynomials}\\
$D_\ell(\boldsymbol\theta,\psi),d_\ell(\theta)$ & Wigner D- and d-matrices & \eqref{DD matrix}, \eqref{d matrix}\\
$\mathbf D_\ell^m(\boldsymbol\theta,\psi),\mathbf P_\ell(\zeta)$
& Wigner matrix columns, Jacobi--Anger coefficients &\S\ref{ss:ModalAnalysis}\\
$\widehat{\textup{EW}}_{\ell}^m
,b_\ell[\mathbf y],\tilde b_\ell[\mathbf y],\widehat{\textup{EW}}_{\ell}
$& modal expansion coefficients &\S\ref{ss:ModalAnalysis}\\
$(X,\mu),\boldsymbol\Phi_X,\Tcont{\!X}{\,}$ & measure space, Bessel family, synthesis operator & \S\ref{ss:StableContApprox}\\
$C_{\textup{cs}},\eta$ & stable continuous approx.\ bound and tolerance  & Def.~\ref{def:SCA}\\
$\sigma,\nu,w$ & measures and density on $\mathbb S^2$ and  
$Y$ & \eqref{weight}\\
$a_\ell^m,\tilde a_\ell^m,\alpha_\ell,\mathcal A,\|\cdot\|_{\mathcal A}$ & Herglotz densities, normalization and space &Def.~\ref{definition herglotz densities}\\
$\tau_\ell,\tau_\pm$ & Jacobi--Anger coefficients and bounds &\eqref{tau jacobi-anger}, \eqref{uniform bounds tau}\\
$\Tcont{Y}{EW}$ & Herglotz transform & \eqref{Herglotz transform}\\
$\Tcont{\Theta}{PW}$ & PPW continuous synthesis operator 
& \eqref{herglotz}\\
$\Phi_P,\Tdisc{\!P\,}{\,}$ & discrete approximation set, synthesis operator & \eqref{eq:discrete_synthesis_op}\\
$C_{\textup{ds}},\eta,s_{\textup{ds}}$ & stable discrete approx.\ bound, tolerance, exponent & 
Def.~\ref{def:SDA}\\
$\gamma,\mathbf x_s,w_s,S$&  Dirichlet trace, sampling nodes, weights, number &\S\ref{subsec:regularized boundary sampling method}\\
$A,\mathbf b,\boldsymbol\xi$ & sampling matrix, load and solution vectors & \S\ref{subsec:regularized boundary sampling method}\\
$\Sigma,\sigma_p,\sigma_{\max}$ & sampling matrix singular values & \S\ref{subsec:regularized boundary sampling method}\\
$\epsilon$ & regularization parameter & \S\ref{subsec:regularized boundary sampling method}\\
$\Sigma_\epsilon,A_{S,\epsilon},\boldsymbol \xi_{S,\epsilon}, ^\dagger$ & regularized matrices and vector, pseudoinverse 
&\S\ref{subsec:regularized boundary sampling method}\\
$\mathcal E$& relative residual & \eqref{relative residual}\\
$\Tdisc{\!P\,}{PW}$ & PPW discrete synthesis operator & \eqref{plane waves approximation set}\\
$K, K_{\mathbf y}$ & reproducing kernel, sampling functionals & Prop.~\ref{Proposition 5.1}\\
$L,\mathcal A_L,\mathcal B_L, N(L)$ & truncation parameter, truncated spaces, dimension & Def.~\ref{def:ALBL}\\
$\mu_N,\rho_N,\widehat{\rho}_N,\Upsilon_N$ & Christoffel, probability density, cumulative funct.s 
&(\ref{rho density}--\ref{rho zeta density}--\ref{cumulative distribution})\\
$\mathbf z_p, \mathbf y_p,(\widehat\theta_{p,1},\widehat\theta_{p,2})$ & sampling point in hypercube, in $Y$, in $\Theta$ & \S\ref{s:ITS}\\
$P$ & cardinality of the approximation set  &\eqref{isomorphism approximation sets}\\
$\Psi_{L,P},\Phi_{L,P}$ & sampling-functional and EPW approximating sets & \eqref{isomorphism approximation sets}\\
$\widetilde{\Upsilon}_N$ & approximated cumulative function & \eqref{approx zeta cumulative}\\
$Q, \Gamma(\cdot,\cdot)$ & (normalized) upper incomplete Gamma function & \eqref{approx zeta cumulative},\eqref{eq:AlphaApprox}\\
$\widehat u_\ell^m, Q_1$ & random Fourier coefficients, unit cube & \S\ref{subsec:approximation of random-expansion solution}--\ref{ss:cube}
\\\hline
\end{tabular}
\caption{List of the symbols used in the paper.}
\label{t:notation}
\end{table}

\section{Evanescent plane waves}\label{sec:evanescent plane waves}

\begin{figure}
\centering
\includegraphics[trim=45 40 0 0,clip,width=0.7\textwidth]{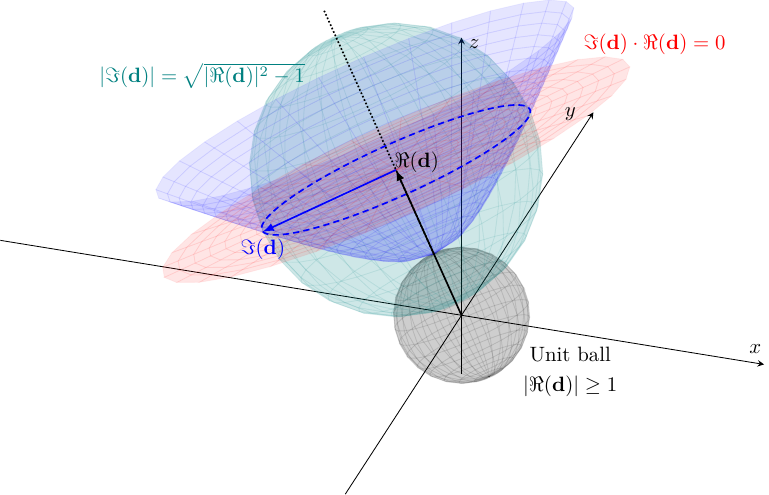}
\caption{Geometrical representation of the constraints in (\ref{complex direction conditions}).
The black dotted line (which originates on the boundary of the unit ball)
and the blue hyperboloid depict respectively the real and imaginary parts of
  elements in the set
$\{\mathbf{d} = R_{\boldsymbol{\theta},\psi}\,
\mathbf{d}_{\uparrow}\left(1+\zeta/2\kappa \right), (\psi,\zeta) \in [0,2\pi) \times [0,+\infty)\}$
for fixed $\boldsymbol{\theta} \in \Theta$, see \eqref{complex direction}.}
\label{figure 2.0}
\end{figure}

We start by introducing and studying evanescent plane waves (EPWs) in 3D. 
These waves satisfy the Helmholtz equation~\eqref{Helmholtz
equation} and generalize the well-known propagative plane waves (PPWs) while
preserving their simple exponential form. 
This section also presents the modal analysis tools that are later used to analyse
approximation properties of both types of plane waves.
In particular, we extend the classical Jacobi--Anger expansion to complex propagation directions, and use this to compute the coefficients of the spherical-wave expansion of any EPW.

\subsection{Definition}\label{ss:EPWdef}

A plane wave $\mathbf{x} \mapsto e^{i\kappa \mathbf{d}\cdot \mathbf{x}}$
satisfies the homogeneous Helmholtz equation (\ref{Helmholtz equation}) if and
only if the direction vector
$\mathbf{d}=(\textup{d}_1,\textup{d}_2,\textup{d}_3) \in \mathbb{C}^3$ fulfills the
constraint $\mathbf{d} \cdot \mathbf{d}=\sum_{i=1}^3\textup{d}_i^2=1$, or
equivalently
\begin{equation}\label{complex direction conditions}
    \left|\Re\left(\mathbf{d}\right)\right|^2-\left|\Im\left(\mathbf{d}\right)\right|^2=1,
    \quad\text{(a)}
    \qquad \qquad \qquad
    \Re\left(\mathbf{d}\right) \cdot \Im\left(\mathbf{d}\right)=0.
    \quad\text{(b)}
\end{equation}
Hence, $\Re\left(\mathbf{d}\right)$ is required only to have a modulus larger
than $1$, and $\Im\left(\mathbf{d}\right)$ must lie on the circle of radius of
$(\left|\Re\left(\mathbf{d}\right)\right|^2-1)^{1/2}$ in the plane orthogonal
to $\Re\left(\mathbf{d}\right)$, see Figure~\ref{figure 2.0}.
We parametrize the set $\{\mathbf{d}\in \mathbb{C}^3:\mathbf{d}\cdot \mathbf{d}=1\}$  by fixing a reference complex
direction vector $\mathbf d_\uparrow$ that meets conditions~\eqref{complex direction conditions}, and
then considering all its possible rigid-body rotations in space. 
For instance, if we let $\Re\left(\mathbf{d}_{\uparrow}\right)$ be aligned with the $z$-axis,
we can pick $\Im\left(\mathbf{d}_{\uparrow}\right)$ aligned with the $x$-axis
so that~(\ref{complex direction conditions}b) is satisfied,
and then~(\ref{complex direction conditions}a) simplifies to
$\Re\left(\textup{d}_{\uparrow,3}\right)^2-\Im\left(\textup{d}_{\uparrow,1}\right)^2=1$. Assuming
$\textup{d}_{\uparrow,1}\geq 0$ and $\textup{d}_{\uparrow,3}\geq 0$, and  
setting $z:=\Re\left(\textup{d}_{\uparrow,3}\right) \geq 1$, we get
$\Im\left(\textup{d}_{\uparrow,1}\right)=(z^2-1)^{1/2}$.
This prompts us to propose the following definition and parametrization of an
\emph{evanescent plane wave}.

\begin{definition}[Evanescent plane wave]\label{def:EPW}
Let $\boldsymbol{\theta}:=(\theta_1,\theta_2) \in \Theta:=[0,\pi] \times [0,2\pi)$, $\psi \in [0,2\pi)$ be the Euler angles and $R_{\boldsymbol{\theta},\psi}:=R_{z}(\theta_2)R_{y}(\theta_1)R_{z}(\psi)$ the associated rotation matrix, where
\begin{equation*}
R_{y}(\theta):=
\begin{bmatrix}
    \cos{(\theta)}       & 0 & \sin{(\theta)}\\
    0       & 1 & 0\\
    -\sin{(\theta)}       & 0 & \cos{(\theta)}
\end{bmatrix},\qquad
R_{z}(\theta):=
\begin{bmatrix}
    \cos{(\theta)}       & -\sin{(\theta)} & 0\\
    \sin{(\theta)}       & \cos{(\theta)} & 0\\
    0       & 0 & 1
\end{bmatrix}.
\end{equation*}
For any $\mathbf{y}:=(\boldsymbol{\theta},\psi,\zeta) \in Y := \Theta \times [0,2\pi) \times [0,+\infty)$, we let
\begin{equation}
\textup{EW}_{\mathbf{y}}(\mathbf{x}):=e^{i\kappa \mathbf{d}(\mathbf{y})\cdot \mathbf{x}}\qquad \forall \mathbf{x} \in \mathbb{R}^3,
\label{evanescent wave}
\end{equation}
where the wave complex direction is given by
\begin{equation}
\mathbf{d}(\mathbf{y}):=R_{\boldsymbol{\theta},\psi}\,\mathbf{d}_{\uparrow}\left(1+\zeta/2\kappa \right) \in \mathbb{C}^3, \qquad \text{and} \qquad 
\mathbf{d}_{\uparrow}(z):=\left(i\sqrt{z^2-1},0,z\right) \in \mathbb{C}^3 \quad \forall z \geq 1.
\label{complex direction}
\end{equation}
\end{definition}

\begin{remark}
We parametrize $\mathbf{d}(\mathbf{y})$ in \textup{(\ref{complex direction})} with three angles $\theta_1,\theta_2,\psi$ and with $\zeta\ge0$, 
while $\mathbf d_\uparrow$ is parametrized by $z\ge1$, related to $\zeta$ by 
$z=1+\zeta/2\kappa$.
This choice, although not immediately apparent, leads to simpler results in the subsequent analysis.
\end{remark}

Assuming $\zeta=0$ in (\ref{evanescent wave}), for any $(\boldsymbol{\theta},\psi) \in \Theta \times [0,2\pi)$, we recover the standard definition of a \emph{propagative plane wave}. For any $\boldsymbol{\theta} \in \Theta$, we let
\begin{equation} \label{eq:PPW_definition} 
\textup{PW}_{\boldsymbol{\theta}}(\mathbf{x}):=
\textup{EW}_{(\boldsymbol{\theta},0,0)}(\mathbf{x})=
e^{i\kappa \mathbf{d}(\boldsymbol{\theta}) \cdot \mathbf{x}} \qquad \forall \mathbf{x} \in \mathbb{R}^3,
\end{equation}
where the wave propagation direction is given by
\begin{equation}
\mathbf{d}(\boldsymbol{\theta}):=\left(\sin \theta_1 \cos \theta_2, \: \sin \theta_1 \sin \theta_2, \: \cos \theta_1\right) \in \mathbb{S}^2 
:=\{\mathbf{x} \in \mathbb{R}^3:|\mathbf{x}| =1\} \subset \mathbb{R}^3.
\label{propagative direction}
\end{equation}
Here, $\mathbf{d}(\boldsymbol{\theta})$ does not depend on 
$\psi$, as $\mathbf{d}_{\uparrow}(1)=(0,0,1)$ is invariant under the rotation $R_z(\psi)$.

\begin{figure}
     \centering
     \begin{subfigure}[b]{0.24\textwidth}
         \centering
         \includegraphics[trim=125 125 5 5,clip,width=3.5cm,height=3.5cm]{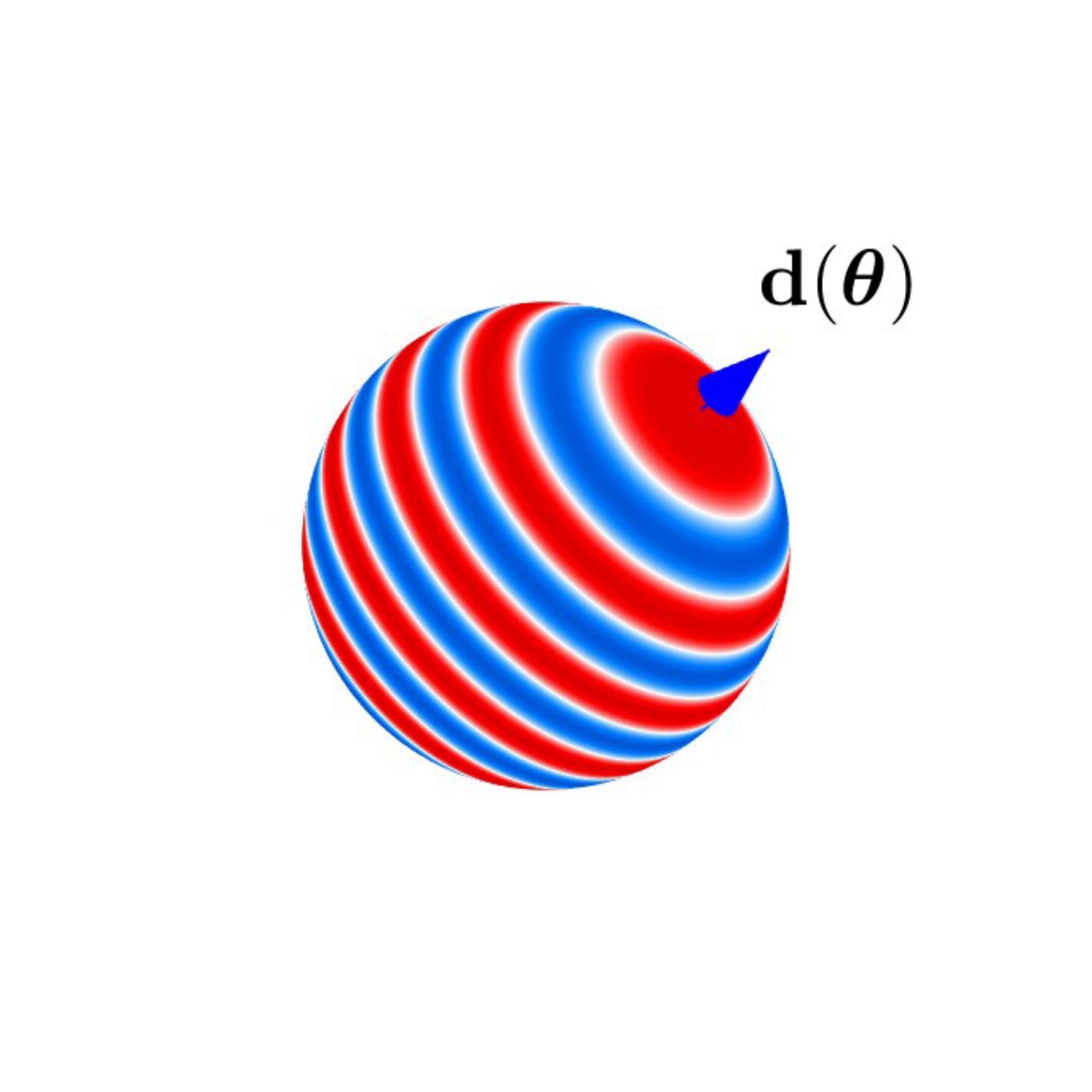}
         \captionsetup{labelformat=empty}   \caption{\!\!\!\!\!\!\!\!\!\!\!\!\!$\psi=\cdot\,$, $\zeta=0$}
     \end{subfigure}
     \begin{subfigure}[b]{0.24\textwidth}
         \centering
         \includegraphics[trim=125 125 5 5,clip,width=3.5cm,height=3.5cm]{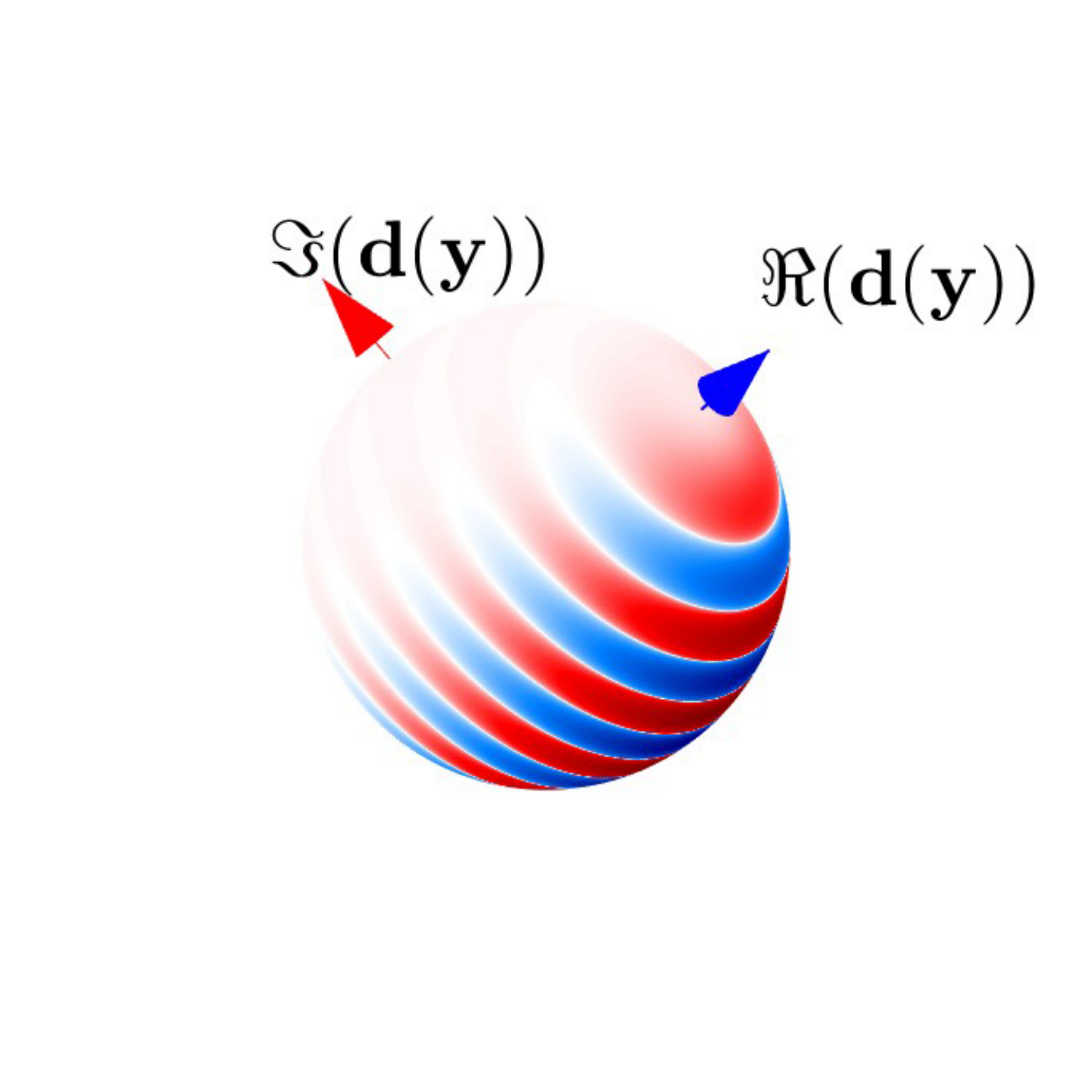}
         \captionsetup{labelformat=empty}   \caption{\!\!\!\!\!\!\!\!\!\!\!\!\!$\psi=\pi$, $\zeta=2$}
     \end{subfigure}
     \begin{subfigure}[b]{0.24\textwidth}
         \centering
         \includegraphics[trim=125 125 5 5,clip,width=3.5cm,height=3.5cm]{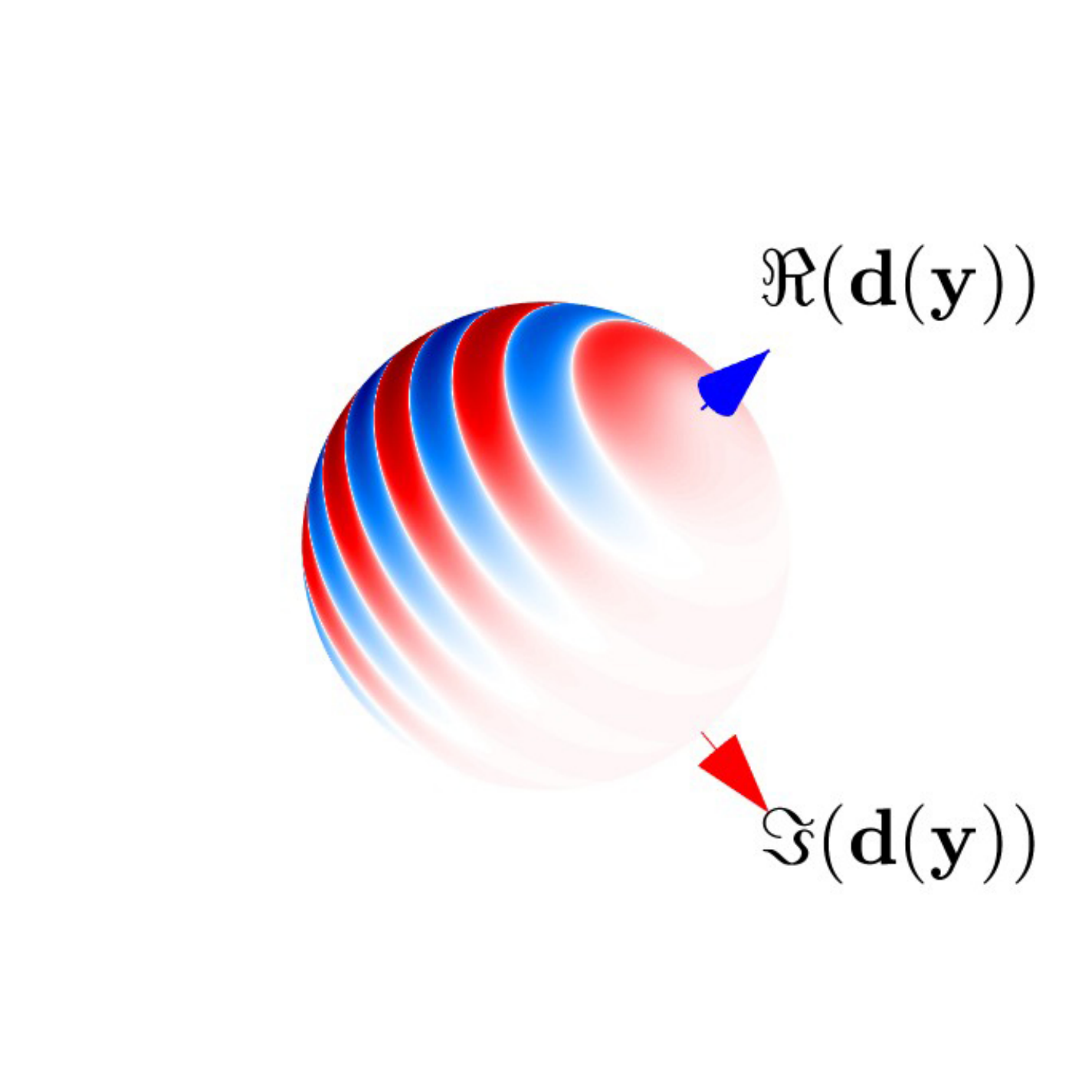}
         \captionsetup{labelformat=empty} \caption{\!\!\!\!\!\!\!\!\!\!\!\!\!$\psi=0$, $\zeta=2$}
     \end{subfigure}
     \begin{subfigure}[b]{0.24\textwidth}
         \centering
         \includegraphics[trim=125 125 5 5,clip,width=3.5cm,height=3.5cm]{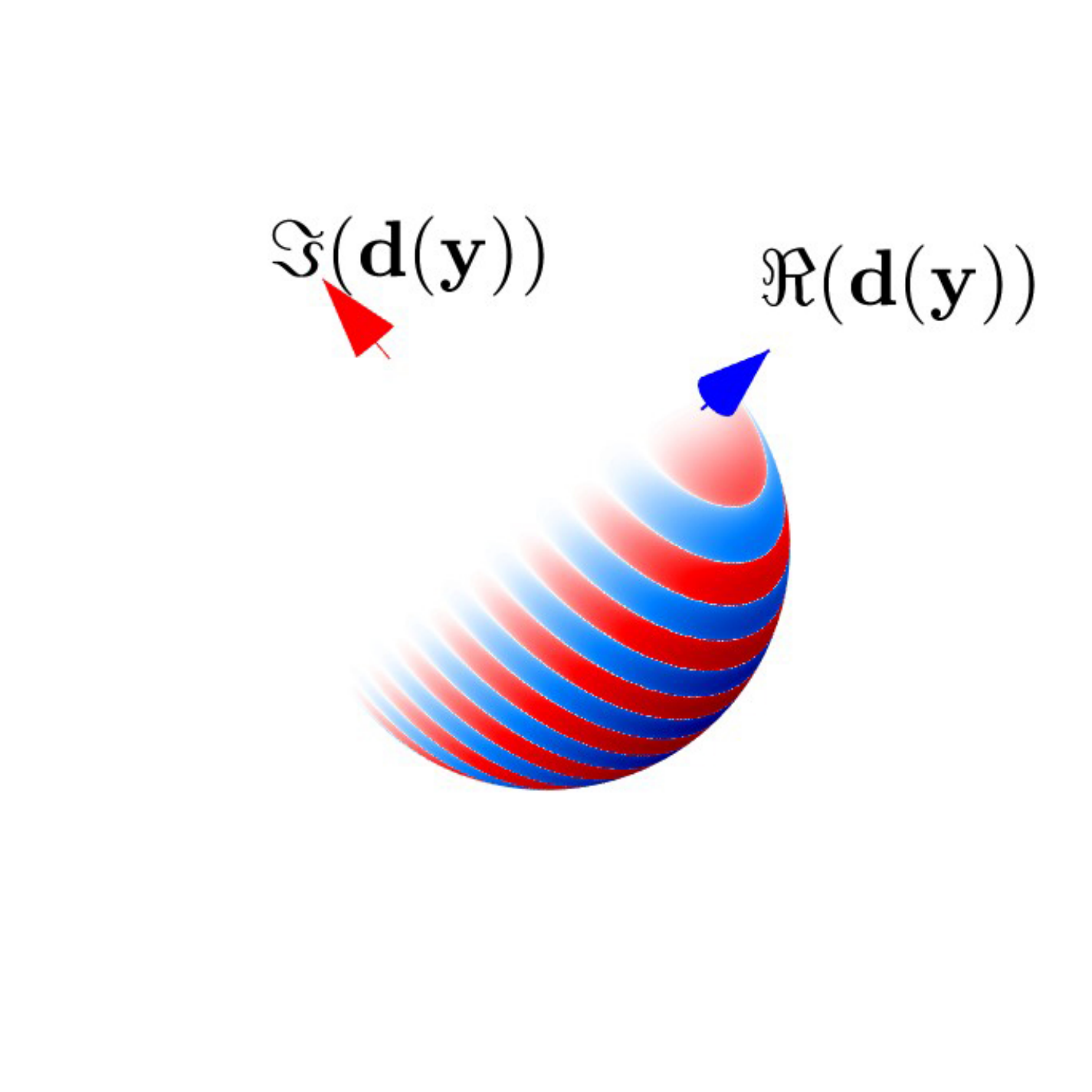}
         \captionsetup{labelformat=empty}  \caption{\!\!\!\!\!\!\!\!\!\!\!$\psi=\pi$, $\zeta=20$}
     \end{subfigure}
\caption{Real part of some EPWs in (\ref{evanescent wave}) restricted to $\partial B_1$.
The blue arrow denotes $\Re\left(\mathbf{d}(\mathbf{y})\right)$ direction, i.e.\ $\mathbf{d}(\boldsymbol{\theta})$ in (\ref{propagative direction}), while the red arrow represents $\Im\left(\mathbf{d}(\mathbf{y})\right)$ direction, i.e.\ $\mathbf{d}^{\bot}(\boldsymbol{\theta},\psi)$, that is the first column of the rotation matrix $R_{\boldsymbol{\theta},\psi}$.
We set $\theta_1=\theta_2=\pi/4$ and wavenumber $\kappa=16$.}
\label{figure 2.1}
\end{figure}

Since the direction vector $\mathbf{d}(\mathbf{y})$ in (\ref{complex direction}) is complex valued, the wave behavior can become unclear.
A more explicit expression of the EPW \eqref{evanescent wave} is
\begin{equation}
\textup{EW}_{\mathbf{y}}(\mathbf{x})=e^{i\kappa\Re(\mathbf{d}(\mathbf{y}))\cdot\mathbf{x}}e^{-\kappa\Im(\mathbf{d}(\mathbf{y}))\cdot\mathbf{x}}=e^{i \left(\frac{\zeta}{2}+\kappa \right) \mathbf{d}(\boldsymbol{\theta}) \cdot \mathbf{x}}e^{- \left( \zeta \left(\frac{\zeta}{4}+\kappa \right) \right)^{1/2} \mathbf{d}^{\bot}(\boldsymbol{\theta},\psi) \cdot \mathbf{x}},
\label{clearer behavior}
\end{equation}
where $\mathbf{d}(\boldsymbol{\theta})$ is defined in (\ref{propagative direction}) and we denote with $\mathbf{d}^{\bot}(\boldsymbol{\theta},\psi)$ the first column of the matrix $R_{\boldsymbol{\theta},\psi}$.
The wave oscillates with apparent wavenumber $\zeta/2+\kappa \geq \kappa$ in the propagation direction $\mathbf{d}(\boldsymbol{\theta})$, parallel to $\Re\left(\mathbf{d}(\mathbf{y})\right)$. Additionally, the wave decays exponentially in the direction $\mathbf{d}^{\bot}(\boldsymbol{\theta},\psi)$, which is orthogonal to $\mathbf{d}(\boldsymbol{\theta})$ and parallel to $\Im\left(\mathbf{d}(\mathbf{y})\right)$.
This justifies naming the new parameters $(\psi,\zeta) \in [0,2\pi) \times [0,+\infty)$, which control the imaginary part of the complex direction $\mathbf{d}(\mathbf{y})$ in (\ref{complex direction}), \emph{evanescence parameters}.
Some EPWs are represented in Figure~\ref{figure 2.1}.

\subsection{Spherical waves}\label{ss:Spherical}

For spherical domains, an explicit orthonormal basis for the Helmholtz solution
space is given by acoustic Fourier modes, the so-called \emph{spherical waves}.
To define them, we briefly review some special functions.

For conciseness, we introduce the index set $\mathcal{I}:=\{(\ell,m) \in \mathbb{Z}^2:  0 \leq |m| \leq \ell\}$. Following \cite[eqs.~(14.7.10) and (14.9.3)]{nist}, the \emph{Ferrers functions} are defined, for all $(\ell,m) \in \mathcal{I}$ and $|x|\leq 1$, as
\begin{equation}
\mathsf{P}_{\ell}^m(x):=\frac{(-1)^m}{2^{\ell}\ell!}(1-x^2)^{\frac{m}{2}}\frac{\textup{d}^{\ell+m}}{\textup{d}x^{\ell+m}}(x^2-1)^{\ell}, \!\!\qquad\!\! \text{so that} \!\!\qquad\!\! \mathsf{P}_{\ell}^{-m}(x)=(-1)^m\frac{(\ell-m)!}{(\ell+m)!}\mathsf{P}_{\ell}^{m}(x).
\label{legendre polynomials}
\end{equation}
In particular, $\mathsf{P}_{\ell}:=\mathsf{P}_{\ell}^{0}$ are simply called \emph{Legendre polynomials of degree $\ell$}.
Following \cite[eq.~(14.30.1)]{nist}, for every $(\theta,\varphi) \in \Theta$ and $(\ell,m) \in \mathcal{I}$, the \emph{spherical harmonics} are defined as
\begin{equation}
Y_{\ell}^m(\theta,\varphi):=\gamma_{\ell}^m e^{im\varphi}\mathsf{P}_{\ell}^m(\cos{\theta}), \qquad \text{where} \qquad \gamma_{\ell}^m:=\left[\frac{2\ell+1}{4\pi} \frac{(\ell-m)!}{(\ell+m)!} \right]^{1/2}
\label{spherical harmonic}
\end{equation}
is a normalization constant, such that
$\|Y_{\ell}^m\|_{L^2(\mathbb{S}^2)}=1$.
With a little abuse of notation, we also write $Y_{\ell}^m(\mathbf{x})$ in
place of $Y_{\ell}^m(\theta,\varphi)$, for $\mathbf{x}=(\sin \theta \cos
\varphi, \sin \theta \sin \varphi,\cos \theta) \in \mathbb{S}^2$.
These functions constitute an orthonormal basis of $L^2(\mathbb{S}^2)$.
The \emph{Condon--Shortley convention} is used, i.e.\ the phase factor
of $(-1)^m$ is included in~\eqref{legendre polynomials} rather than
in $\gamma_{\ell}^m$. Finally, for every $r >0 $, we denote with
$j_{\ell}(r):=\sqrt{\pi/2r}J_{\ell+1/2}(r)$ the \emph{spherical Bessel
functions} \cite[eq.~(10.47.3)]{nist}, where $J_{\ell}(r)$ are the usual
Bessel functions \cite[eq.~(10.2.2)]{nist}.

We are now ready to define the spherical waves. For normalization purposes, let us introduce the following  $\kappa$-dependent Hermitian product and associated norm:
\begin{equation} \label{eq:B_norm}
(u,v)_{\mathcal{B}}:=(u,v)_{L^2(B_1)}+\kappa^{-2}(\nabla u, \nabla v)_{L^2(B_1)^3}, \qquad \|u\|^2_{\mathcal{B}}:=(u,u)_{\mathcal{B}} \qquad \forall u,v \in H^1(B_1).
\end{equation}

\begin{figure}
\begin{subfigure}{.31\textwidth}
\centering
\includegraphics[trim=100 100 100 100,clip,width=3.5cm,height=3.5cm]{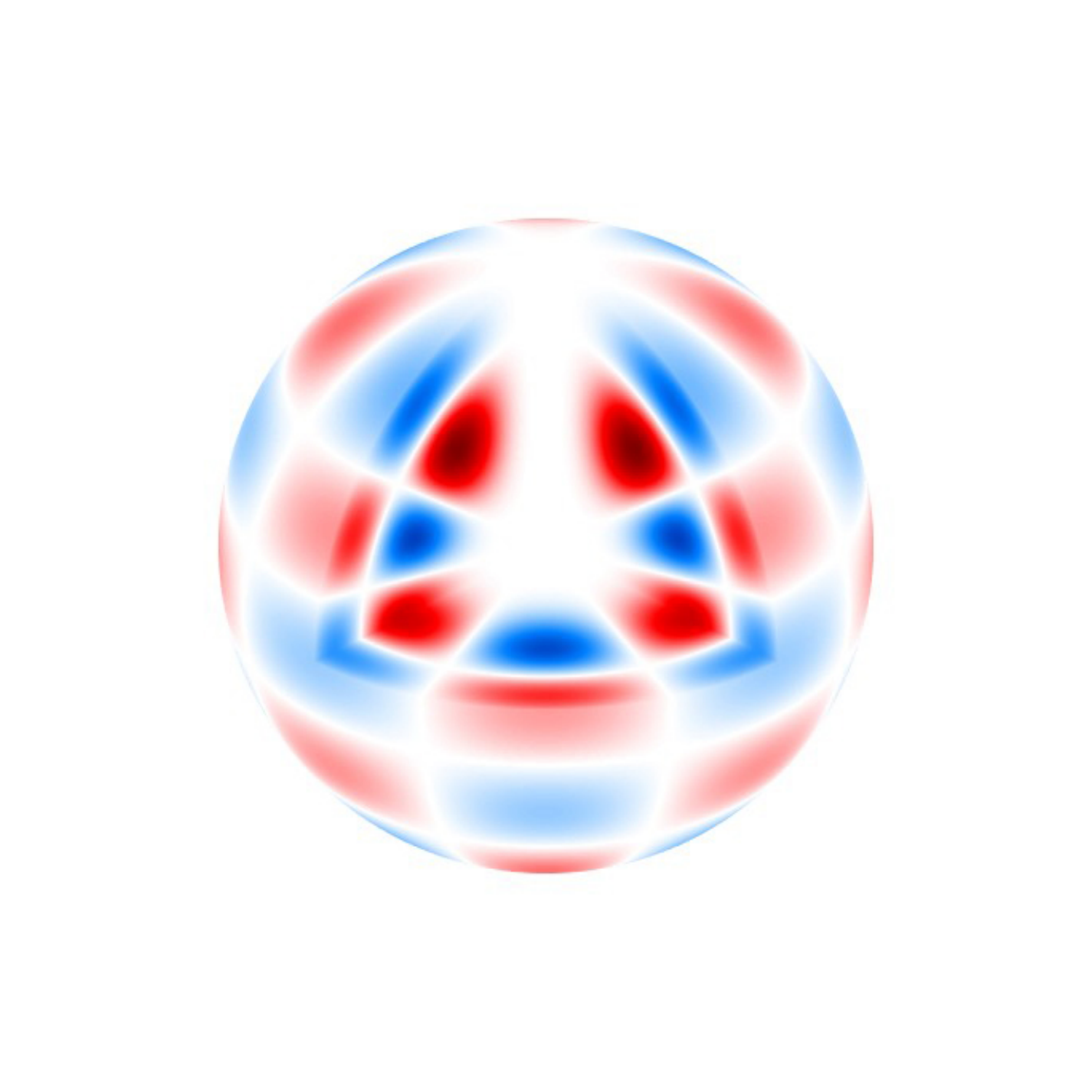}
\captionsetup{labelformat=empty}
\caption{Propagative: $\ell=2m=\kappa/2=8$}
\end{subfigure}\hfill
\begin{subfigure}{.31\textwidth}
\centering
\includegraphics[trim=100 100 100 100,clip,width=3.5cm,height=3.5cm]{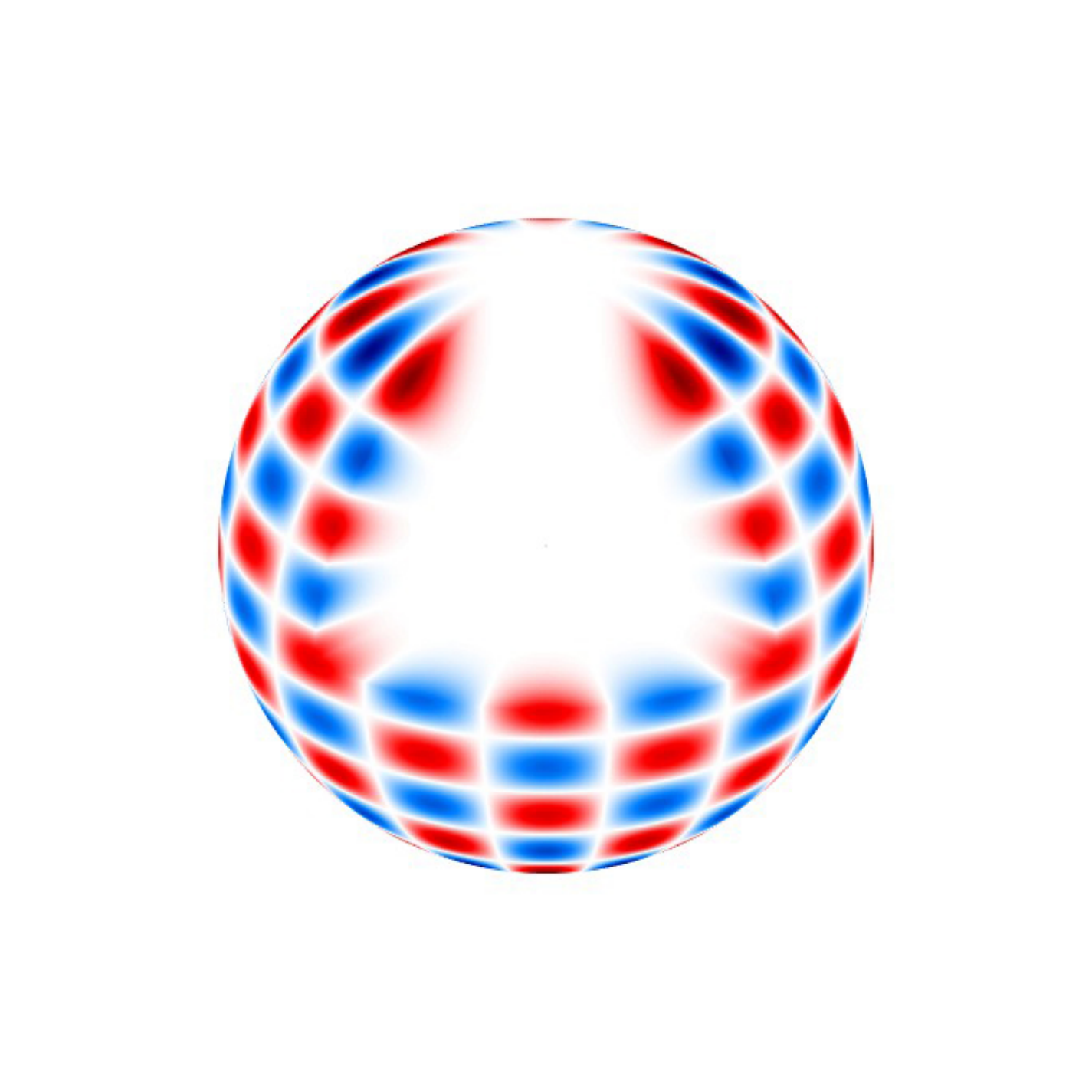}
\captionsetup{labelformat=empty}
\caption{Grazing: $\ell=2m=\kappa=16$}
\end{subfigure}\hfill
\begin{subfigure}{.31\textwidth}
\centering
\includegraphics[trim=100 100 100 100,clip,width=3.5cm,height=3.5cm]{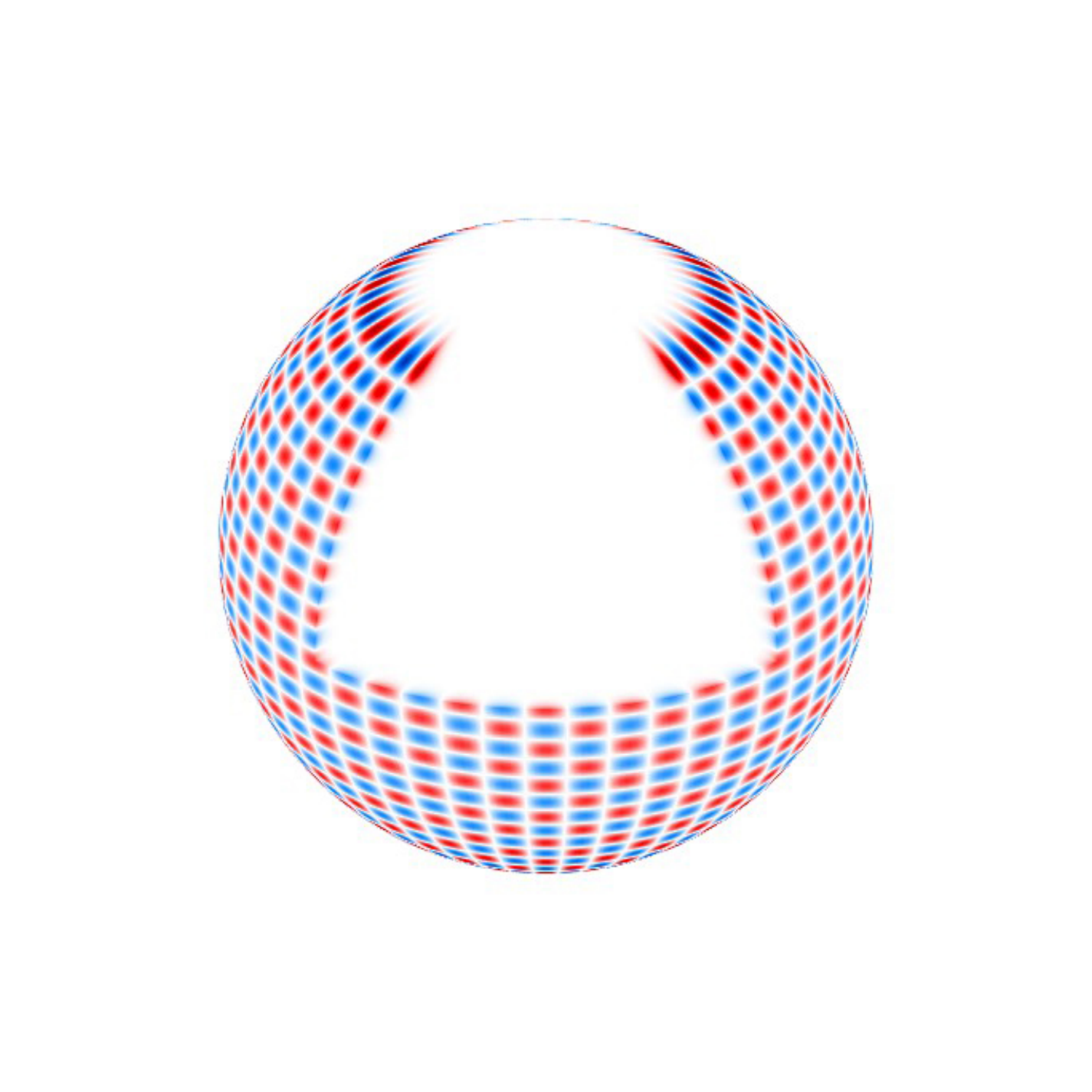}
\captionsetup{labelformat=empty}
\caption{Evanescent: $\ell=2m=3\kappa=48$}
\end{subfigure}
\caption{Real part of some spherical waves $b_{\ell}^m$ on $\partial \{B_1\setminus\{\mathbf{x} \in \mathbb{R}^3 : \mathbf{x}>0\}\}$. With increasing order $\ell$, these functions concentrates progressively closer to the boundary $\partial B_1$.}
\label{figure 2.2}
\end{figure}

\begin{definition}[Spherical waves]\label{def:B}
We define, for any $(\ell,m) \in \mathcal{I}$
\begin{equation}
b_\ell^m:=\beta_{\ell}\tilde{b}_\ell^m, \qquad \text{where} \qquad \tilde{b}_\ell^m(\mathbf{x}):=j_\ell(\kappa |\mathbf{x}|)Y_\ell^m(\mathbf{x}/|\mathbf{x}|) \quad  \forall \mathbf{x} \in B_1, \qquad \beta_\ell:=\|\tilde{b}_\ell^m\|^{-1}_\mathcal{B}. 
\label{b tilde definizione}
\end{equation}
Furthermore, we introduce the space $\mathcal{B}:=\overline{\textup{span}\{b_\ell^m\}_{(\ell,m) \in \mathcal{I}}}^{\|\cdot\|_{\mathcal{B}}} \subsetneq H^{1}(B_1)$.
\end{definition}

Thanks to \cite[eq.~(2.4.23)]{nedelec} and \cite[eq.~(10.47.1)]{nist} it is
clear that the spherical waves satisfy the Helmholtz
equation (\ref{Helmholtz equation}).
Although $b_{\ell}^m$ depends on the two indices $(\ell,m) \in \mathcal{I}$,
the normalization factor $\beta_{\ell}$ is independent of $m$, as shown later
in Lemma \ref{Lemma 2.6}.

Following common terminology, we refer to spherical waves with mode number
$\ell < \kappa$ (resp.\ $\ell \gg \kappa$) as \emph{propagative} (resp.\
\emph{evanescent}) modes;
their `energy' is distributed throughout the unit ball (resp.\ concentrated
near the unit sphere).
Lastly, waves with $\ell \approx \kappa$ are called \emph{grazing} modes.
Figure~\ref{figure 2.2} illustrates the behavior of several functions
$b_{\ell}^m$ on the boundary of the unit ball without the first octant.
We recall below a standard result on spherical waves, derived via separation of variables; see \cite[Lem.~1.2 and~1.3]{galante} for a proof.

\begin{lemma}
The space $(\mathcal{B},\!\|\cdot\|_{\mathcal{B}})$ is a Hilbert space and the family $\{b_{\ell}^m\}_{(\ell,m) \in \mathcal{I}}$ is a Hilbert basis:
\begin{equation*}
(b_{\ell}^m,b_q^n)_{\mathcal{B}}=\delta_{\ell, q}\delta_{m,n} \quad \forall(\ell,m),(q,n) \in \mathcal{I}, \qquad \text{and} \qquad u=\sum_{(\ell,m) \in \mathcal{I}}(u,b_{\ell}^m)_{\mathcal{B}}\,b_{\ell}^m
\quad \forall u \in \mathcal{B}.
\end{equation*}
Moreover,
$u \in H^1(B_1)$ satisfies the Helmholtz equation \textup{(\ref{Helmholtz equation})} if and only if $u \in \mathcal{B}$.
\end{lemma}

The upcoming analysis uses the asymptotics of the normalization coefficient
$\beta_{\ell}$, which grows super-exponentially with $\ell$
after a pre-asymptotic regime up to $\ell \approx \kappa$.

\begin{lemma} \label{Lemma 2.6}
We have for all $\ell \geq 0$
\begin{equation}
  \beta_{\ell}=
  \sqrt{\frac{2 \kappa}{\pi}}\left[(1+{\ell}/{\kappa^2})J^2_{\ell+\frac{1}{2}}(\kappa)-(J_{\ell-\frac{1}{2}}(\kappa)+J_{\ell+\frac{1}{2}}(\kappa)/\kappa )J_{\ell+\frac{3}{2}}(\kappa) \right]^{-\frac{1}{2}}
  \underset{\ell\rightarrow \infty}{\sim} 
  2^{\frac{3}{2}}\kappa\left(\frac{2}{e \kappa}\right)^{\ell}\ell^{\ell+\frac{1}{2}}.
\label{beta_l asymptotic}
\end{equation}
\end{lemma}
\begin{proof}
Since $b_{\ell}^m$ solves the Helmholtz equation (\ref{Helmholtz equation}), the expansion in (\ref{beta_l asymptotic}) stems from
\begin{equation}
\beta_{\ell}^{-2}=\|\tilde{b}_{\ell}^m\|^2_{\mathcal{B}}=2\|\tilde{b}_{\ell}^m\|^2_{L^2(B_1)}+\kappa^{-2}(\partial_{\mathbf{n}}\tilde{b}_{\ell}^m,\tilde{b}_{\ell}^m)_{L^2(\partial B_1)},
\label{B norm 2}
\end{equation}
and, using \cite[eqs.~(10.22.5) and (10.51.2)]{nist},
\begin{equation}
\|\tilde{b}_{\ell}^m\|^2_{L^2(B_1)}=\int_0^1\!\!j^2_{\ell}(\kappa r)r^2\textup{d}r=\frac{\pi}{2\kappa}\int_0^1\!\!J^2_{\ell+\frac{1}{2}}(\kappa r)r\textup{d}r=\frac{\pi}{4\kappa}\left(J^2_{\ell+\frac{1}{2}}(\kappa)-J_{\ell-\frac{1}{2}}(\kappa)J_{\ell+\frac{3}{2}}(\kappa) \right), \numberthis \label{L2}
\end{equation}
\begin{equation}
(\partial_{\mathbf{n}}\tilde{b}_{\ell}^m,\tilde{b}_{\ell}^m)_{L^2(\partial B_1)}=\kappa j'_{\ell}(\kappa)j_{\ell}(\kappa)=\ell j^2_{\ell}(\kappa)-\kappa j_{\ell}(\kappa)j_{\ell+1}(\kappa)=\frac{\pi}{2 \kappa} \left( \ell J^2_{\ell+\frac{1}{2}}(\kappa)-\kappa J_{\ell+\frac{1}{2}}(\kappa)J_{\ell+\frac{3}{2}}(\kappa)\right).
\label{partial}
\end{equation}

The proof of the asymptotic behavior consists in showing that we have as $\ell \rightarrow \infty$
\begin{equation*}
\|\tilde{b}_{\ell}^m\|^2_{L^2( B_1)} \sim \frac{1}{16}\left(\frac{e \kappa}{2}\right)^{2\ell} \ell^{-2\left(\ell+\frac{3}{2} \right)}, \qquad \text{and} \qquad (\partial_{\mathbf{n}}\tilde{b}_{\ell}^m,\tilde{b}_{\ell}^m)_{L^2(\partial B_1)} \sim \frac{1}{8}\left(\frac{e \kappa}{2}\right)^{2\ell}\ell^{-2\left(\ell+\frac{1}{2} \right)}.
\end{equation*}
Hence, thanks to (\ref{B norm 2}), the dominant term in $\|\tilde{b}_{\ell}^m\|_{\mathcal{B}}$ in the limit $\ell \rightarrow \infty$ is the boundary one.

Let us consider the $L^2(B_1)$ norm. From (\ref{L2}) and since \cite[eq.~(10.19.1)]{nist} holds, namely
\begin{equation}
J_{\ell}(r) \sim \frac{1}{\sqrt{2\pi \ell}}\left(\frac{er}{2 \ell} \right)^{\ell} \qquad \text{as}\,\,\,\ell \rightarrow \infty,
\label{bessel asymptotics}
\end{equation}
we get as $\ell \rightarrow \infty$
\begin{equation}
\|\tilde{b}_{\ell}^m\|^2_{L^2( B_1)} \sim \frac{1}{8 \kappa}\left(\frac{e \kappa}{2} \right)^{2\ell+1}\left(\ell+\frac{1}{2} \right)^{-2(\ell+1)}\Bigg[ 1- \frac{\left(\ell+\frac{1}{2} \right)^{2(\ell+1)}}{\left(\ell-\frac{1}{2} \right)^{\ell}\left(\ell+\frac{3}{2} \right)^{\ell+2}} \Bigg].
\label{www}
\end{equation}
Moreover, since for every $x,y,z \in \mathbb{R}$
\begin{equation}
(\ell+x)^{y\ell+z}\sim \ell^{y\ell+z}\exp\bigg\{{xy+\frac{x(2z-xy)}{2\ell}}\bigg\}\sim \ell^{y\ell+z}e^{xy} \qquad \text{as}\,\,\,\ell \rightarrow \infty,
\label{ell asymptotics}
\end{equation}
the term inside the square brackets in (\ref{www}) is equivalent to $\ell^{-1}$ at infinity.

Consider now the term $(\partial_{\mathbf{n}}\tilde{b}_{\ell}^m,\tilde{b}_{\ell}^m)_{L^2(\partial B_1)}$ in (\ref{B norm 2}). From (\ref{partial}) and (\ref{bessel asymptotics}), we get as $\ell \rightarrow \infty$
\begin{equation*}
(\partial_{\mathbf{n}}\tilde{b}_{\ell}^m,\tilde{b}_{\ell}^m)_{L^2(\partial B_1)} \sim \frac{1}{4 \kappa}\left(\frac{e \kappa}{2} \right)^{2\ell+1}\left(\ell+\frac{1}{2} \right)^{-2(\ell+1)}\left[\ell- \frac{e \kappa^2}{2} \frac{\left(\ell+ \frac{1}{2} \right)^{\ell+1}}{\left(\ell+ \frac{3}{2} \right)^{\ell+2}}\right],
\end{equation*}
and, thanks to (\ref{ell asymptotics}), it is readily checked that the second term inside the square brackets is dominated by the first one, since it is equivalent to $\kappa^2/2\ell$ at infinity.
\end{proof}
\vspace{-.1cm}

\subsection{Complex-direction Jacobi--Anger identity}

The explicit series expansion of PPWs in the spherical wave basis is given by
the Jacobi--Anger identity~\cite[eq.~(14)]{hiptmair-moiola-perugia4}, namely
\begin{equation}
  \textup{PW}_{\boldsymbol{\theta}}(\mathbf{x})=
  4 \pi \sum_{\ell=0}^{\infty}i^{\ell}\sum_{m=-\ell}^{\ell} \overline{Y_{\ell}^m\left(\boldsymbol{\theta}\right)}
  j_{\ell}(\kappa |\mathbf{x}|)
  Y_{\ell}^m(\mathbf{x}/|\mathbf{x}|)
  \qquad \forall \mathbf{x} \in B_1,\,\forall \boldsymbol{\theta} \in \Theta.
\label{jacobi-anger}
\end{equation}
The goal of this section is to obtain a similar expansion for EPWs, i.e.\ 
for complex-valued directions $\mathbf{d}(\mathbf{y})$, which to the best of
our knowledge is not available in the literature.
This generalization is not trivial and requires additional definitions
and lemmas.

\paragraph{Associated Legendre functions}
Following \cite[sect.~3.2, eq.~(6)]{erdelyi}, we adopt the convention
\begin{equation}
(w^2-1)^{m/2}:=\mathcal{P}\left[(w+1)^{m/2}\right]\,\mathcal{P}\left[(w-1)^{m/2}\right] \qquad \forall m \in \mathbb{Z},\,\forall w \in \mathbb{C},
\label{convention}
\end{equation}
where $\mathcal{P}[\,\,\cdot\,\,]$ indicates the standard principal branch.
For odd $m$, (\ref{convention}) allows to eliminate the branch cut along the imaginary axis simply by mirroring the function values from the right-half of the complex plane to the left-half (see \cite[sect.~4.2]{galante}).
Following \cite[eqs.~(14.7.14) and (14.9.13)]{nist}, the \emph{associated Legendre functions} are defined, for every $(\ell,m) \in \mathcal{I}$ and $w \in \mathbb{C}$, as
\begin{equation}
P_{\ell}^m(w):=\frac{1}{2^{\ell}\ell!}(w^2-1)^{\frac{m}{2}}\frac{\textup{d}^{\ell+m}}{\textup{d}w^{\ell+m}}(w^2-1)^{\ell}, \!\!\qquad\!\! \text{so that} \!\!\qquad\!\! P_{\ell}^{-m}(w)=\frac{(\ell-m)!}{(\ell+m)!}P_{\ell}^{m}(w).
\label{legendre2 polynomials}
\end{equation}
For every odd $m$, $P_{\ell}^m$ is a single-valued function on the complex plane with a branch cut along the interval $(-1,1)$, where it is continuous from above; otherwise, if $m$ is even, $P_{\ell}^m$ is a polynomial of degree $\ell$. Notably, $P_{\ell}(w):=P_{\ell}^0(w)=\mathsf{P}_{\ell}(w)$ for all $w \in \mathbb{C}$.
Moreover, \cite[eq.~(14.23.1)]{nist} explicitly provides:
\begin{equation}
\lim_{\epsilon 
\searrow 0}P_{\ell}^m(x\pm i\epsilon)=i^{\mp m}\mathsf{P}_{\ell}^m(x) \qquad \forall x \in (-1,1).
\label{on the cut}
\end{equation}
The next lemma extends the identity \cite[eq.~(2.46)]{colton-kress} to complex values of $t$:
\begin{equation}
e^{irt}=\sum_{\ell=0}^{\infty}i^{\ell}(2\ell+1)j_{\ell}(r)\mathsf{P}_{\ell}(t) \qquad \forall r \geq 0,\,\forall t \in [-1,1].
\label{pre jacobi-anger}
\end{equation}

\begin{lemma}
Let $\ell \geq 0$. We have for every $0 \leq m \leq \ell$ and $w \in \mathbb{C}$
\begin{equation}
P_{\ell}^m(w)=\frac{(\ell+m)!}{2^{\ell}\ell!}\sum_{k=0}^{\ell-m}\binom{\ell}{k}\binom{\ell}{m+k}\left(w-1\right)^{\ell-\left(m/2+k\right)}\left(w+1\right)^{m/2+k}.
\label{sum legendre expansion}
\end{equation}
In particular, due to \textup{(\ref{legendre2 polynomials})}, $P_{\ell}^m(z)\geq 0$ for every real $z\geq 1$ and $(\ell,m) \in \mathcal{I}$.
Moreover,
\begin{equation}\label{pre jacobi anger 2}
  e^{irw}=\sum_{\ell=0}^{\infty}i^{\ell}(2\ell+1)j_{\ell}(r)P_{\ell}(w)
  \qquad \forall r \geq 0, \,\forall w \in \mathbb{C}.
\end{equation}
\end{lemma}
\begin{proof}
It can be readily seen that
\begin{align*}
\frac{\textup{d}^{\ell+m}}{\textup{d}w^{\ell+m}}(w^2-1)^{\ell}&=\sum_{k=0}^{\ell+m}\binom{\ell+m}{k}\left(\frac{\textup{d}^k}{\textup{d}w^k}(w-1)^{\ell}\right)\left(\frac{\textup{d}^{\ell+m-k}}{\textup{d}w^{\ell+m-k}}(w+1)^{\ell}\right)\\
&=\sum_{k=m}^{\ell}\binom{\ell+m}{k}\left(\frac{\ell!}{(\ell-k)!}(w-1)^{\ell-k}\right)\left(\frac{\ell!}{(k-m)!}(w+1)^{k-m}\right)\\
&=\frac{(w-1)^{\ell}}{(w+1)^m}\sum_{k=m}^{\ell}\frac{(\ell+m)!}{k!(\ell+m-k)!}\frac{\ell!}{(\ell-k)!}\frac{\ell!}{(k-m)!}\left(\frac{w+1}{w-1}\right)^k\\
&=(\ell+m)!(w-1)^{\ell-m}\sum_{k=0}^{\ell-m}\binom{\ell}{k}\binom{\ell}{m+k}\left(\frac{w+1}{w-1}\right)^k.
\end{align*}
Therefore, thanks to the definitions (\ref{convention}) and (\ref{legendre2 polynomials}), the expansion (\ref{sum legendre expansion}) follows.

Let $R > 1$ and $B_R:=\{w \in \mathbb{C}: |w| <R\}$. We want to check that the right-hand side in (\ref{pre jacobi anger 2}) is well-defined for every $r \geq 0$ and $w \in B_R$. Due to (\ref{bessel asymptotics}) and (\ref{ell asymptotics}),
it is enough to prove
\begin{equation}
\sum_{\ell=0}^{\infty}\left(\frac{er}{2\ell+1}\right)^{\ell}\left| P_{\ell}(w)\right|<\infty \qquad \forall r \geq 0,\,\forall w \in B_R.
\label{other series}
\end{equation}
Thanks to (\ref{sum legendre expansion}), $P_\ell=P_\ell^0$, and the Vandermonde identity \cite[eq.~(1)]{Sokal}, it follows
\begin{equation*}
|P_{\ell}(w)|\leq \frac{1}{2^{\ell}}\sum_{k=0}^{\ell}\binom{\ell}{k}^2|w-1|^{\ell-k}|w+1|^k\leq\sum_{k=0}^{\ell}\binom{\ell}{k}^2\left(\frac{R+1}{2}\right)^{\ell}=\binom{2\ell}{\ell}\left(\frac{R+1}{2}\right)^{\ell},
\end{equation*}
and therefore, for every $r \geq 0$ and $w \in B_R$, the series (\ref{other series}) is dominated by
\begin{equation}
\sum_{\ell=0}^{\infty}c_{\ell}, \qquad \text{where} \qquad c_{\ell}:=\binom{2\ell}{\ell}\left[\frac{er(R+1)}{4\ell+2}\right]^{\ell}.
\label{other other series}
\end{equation}
The series (\ref{other other series}) is convergent, as confirmed by the ratio test: in fact, from (\ref{ell asymptotics}), we have
\begin{equation*}
\frac{c_{\ell+1}}{c_{\ell}}\sim \left(\frac{\ell+1/2}{\ell+3/2}\right)^{\ell+1}\frac{er(R+1)}{\ell+1}\sim\frac{r(R+1)}{\ell} \qquad \text{as }\ell \rightarrow \infty.
\end{equation*}
Thus, the right-hand side of (\ref{pre jacobi anger 2}) is well-defined for every $r \geq 0$ and $w \in B_R$. The functions $w \mapsto e^{irw}$ and $w \mapsto P_{\ell}(w)$ are analytic on $B_R$ and, since identity (\ref{pre jacobi-anger}) holds, that is (\ref{pre jacobi anger 2}) with $w \in [-1,1]$, it follows that (\ref{pre jacobi anger 2}) also holds for every $r \geq 0$ and $w \in B_R$ due to \cite[Th. 3.2.6]{ablowitz}. As $R >1$ is arbitrary, (\ref{pre jacobi anger 2}) is valid for every $w \in \mathbb{C}$.
\end{proof}

\paragraph{Wigner matrices, rotations of spherical harmonics and addition theorem}

The next definition aligns with \cite[eq.~(34)]{pendleton} and \cite[eq.~(1)]{feng}, albeit with a distinction: we invert the angle signs to ensure consistency with the notation of PPW directions in (\ref{propagative direction}).

\begin{definition}[Wigner matrices]
Let $(\boldsymbol{\theta},\psi) \in \Theta \times [0,2\pi)$ be the Euler angles and $\ell \geq 0$. The \textup{Wigner D-matrix} is the unitary matrix $D_{\ell}(\boldsymbol{\theta},\psi)=(D_{\ell}^{m,m'}\!(\boldsymbol{\theta},\psi))_{m,m'} \!\in \mathbb{C}^{(2\ell+1) \times (2\ell +1)}$, where
\begin{equation}
D_{\ell}^{m,m'}(\boldsymbol{\theta},\psi):=e^{im'\theta_2}d_{\ell}^{\,m,m'}(\theta_1)e^{im\psi}
\qquad \forall\,|m|,|m'|\leq \ell.
\label{DD matrix}
\end{equation}
In turn, $d_{\ell}(\theta):=(d_{\ell}^{\,m,m'}\!(\theta))_{m,m'} \!\in \mathbb{R}^{(2\ell+1) \times (2\ell +1)}$ is called \textup{Wigner d-matrix} and its entries are:
\begin{equation}
d_{\ell}^{\,m,m'}(\theta):=\sum_{k=k_{\textup{min}}}^{k_{\textup{max}}}w_{\ell,k}^{m,m'}\left(\cos{\frac{\theta}{2}}\right)^{2(\ell-k)+m'-m}\left(\sin{\frac{\theta}{2}}\right)^{2k+m-m'} \qquad \forall\,|m|,|m'|\leq \ell,
\label{d matrix}
\end{equation}
where
\begin{equation*}
w_{\ell,k}^{m,m'}:=\frac{(-1)^k\left[(\ell+m)!(\ell-m)!(\ell+m')!(\ell-m')!\right]^{1/2}}{(\ell-m-k)!(\ell+m'-k)!(k+m-m')!\,k!},
\end{equation*}
with $k_{\textup{min}}:=\max\{0,m'-m\}$ and $k_{\textup{max}}:=\max\{\ell-m,\ell+m'\}$.
\end{definition}
The Wigner D-matrix $D_{\ell}(\boldsymbol{\theta},\psi)$ is used to express the image of any spherical harmonic of degree $\ell$ under the rotation $R_{\boldsymbol{\theta},\psi}$ as a 
linear combination of spherical harmonics of the same degree. In fact the expansion formula \cite[sect.~4.1, eq.~(5)]{quantumtheory} holds, namely 
\begin{equation}
Y_{\ell}^{m}(\mathbf{x})=\sum_{m'=-\ell}^{\ell}\overline{D_{\ell}^{m,m'}(\boldsymbol{\theta},\psi)}Y_{\ell}^{m'}(R_{\boldsymbol{\theta},\psi}\mathbf{x}) \qquad \forall \mathbf{x} \in \mathbb{S}^2,\,\forall (\ell,m) \in \mathcal{I}.
\label{wigner property}
\end{equation}

We finally establish a generalized Legendre
addition theorem, extending  
e.g.~\cite[eq.~(2.30)]{colton-kress}.

\begin{lemma}
  For any
  $\ell \geq 0$, $\mathbf{x} \in \mathbb{S}^2$ and 
  \(\mathbf{y} = (\boldsymbol{\theta}, \psi, \zeta) \in Y\)
  we have
  \begin{equation}
    \sum_{{m}=-\ell}^{\ell} \sum_{{m'}=-\ell}^{\ell}
    \overline{D_{\ell}^{m',m}(\boldsymbol{\theta},\psi)}
    \gamma_{\ell}^{m'}i^{-m'} P_{\ell}^{m'}\left(1+{\zeta}/{2\kappa}\right)
    Y_\ell^m(\mathbf{x})
    = \frac{2\ell+1}{4\pi}
    P_{\ell}(\mathbf{d}(\mathbf{y})\cdot \mathbf{x}),
  \label{addition theorem 2}
  \end{equation}
\end{lemma}
\begin{proof}
  Let
  $\mathbf{x}=(\sin{\theta}\cos{\varphi},\sin{\theta}\sin{\varphi},\cos{\theta})\in
  \mathbb{S}^2$ with $(\theta,\varphi) \in \Theta$
  and let
  \(\mathbf{y} = (\boldsymbol{\theta}, \psi, \zeta) \in Y\)
  with \(z = 1 + \zeta/2\kappa\).
  We need to establish that
  \begin{equation}\label{eq:Pl(dx)}
    \sum_{m=-\ell}^{\ell}
    \gamma_{\ell}^m i^{-m}P_{\ell}^m(z)
    Y_{\ell}^m(\mathbf{x})
    = \frac{2\ell+1}{4\pi}
    P_{\ell}(\mathbf{d}_{\uparrow}(z)\cdot \mathbf{x}),
  \end{equation}
  from which the result follows using~\eqref{wigner property} 
  and $\mathbf d(\mathbf y)=R_{\boldsymbol{\theta},\psi}\mathbf{d_\uparrow}(z)$ from \eqref{complex direction}.
  On the one hand,
  \begin{equation*}
  P_{\ell}\left(z\cos\theta+i\sqrt{z^2-1}\sin \theta \cos \varphi\right)=P_{\ell}\left(\mathbf{d}_{\uparrow}(z)\cdot \mathbf{x}\right).
  \end{equation*}
  On the other hand, thanks to (\ref{legendre polynomials}), (\ref{spherical harmonic}), and (\ref{legendre2 polynomials}),
  \begin{align*}
  &\frac{4\pi}{2\ell+1}\!\sum_{m=-\ell}^{\ell}(\gamma_{\ell}^m)^2i^{-m}P_{\ell}^m(z)\mathsf{P}_{\ell}^m(\cos \theta)e^{im\varphi}=\!\!\sum_{m=-\ell}^{\ell}\frac{(\ell-m)!}{(\ell+m)!}i^{-m}P_{\ell}^m(z)\mathsf{P}_{\ell}^m(\cos \theta)e^{im\varphi}\\
  &=\sum_{m=0}^{\ell}\frac{(\ell-m)!}{(\ell+m)!}i^{-m}P_{\ell}^m(z)\mathsf{P}_{\ell}^m(\cos \theta)e^{im\varphi}+\sum_{m=1}^{\ell}\frac{(\ell-m)!}{(\ell+m)!}i^{-m}P_{\ell}^{m}(z)\mathsf{P}_{\ell}^{m}(\cos \theta)e^{-im\varphi}\\
  &=P_{\ell}(z)\mathsf{P}_{\ell}(\cos \theta)+2\sum_{m=1}^{\ell}\frac{(\ell-m)!}{(\ell+m)!}
  i^{-m}P_{\ell}^m(z)\mathsf{P}_{\ell}^m(\cos \theta)\cos(m\varphi).
  \end{align*}
  From (\ref{on the cut}) and the branch cut convention (\ref{convention}) we get
  \begin{align*}
  \lim_{\epsilon\searrow0}&\bigg[
  P_{\ell}(z)P_{\ell}(\cos \theta-i\epsilon)+2\sum_{m=1}^{\ell}\frac{(\ell-m)!}{(\ell+m)!}
  (-1)^mP_{\ell}^m(z)P_{\ell}^m(\cos \theta-i\epsilon)\cos(m\varphi)\bigg]
  \\
  &=P_{\ell}(z)\mathsf{P}_{\ell}(\cos \theta)+2\sum_{m=1}^{\ell}\frac{(\ell-m)!}{(\ell+m)!}
  i^{-m}P_{\ell}^m(z)\mathsf{P}_{\ell}^m(\cos \theta)\cos(m\varphi),\,\,\,
  \\
  \lim_{\epsilon\searrow0}&\bigg[
  P_{\ell}\left(z(\cos\theta-i\epsilon)-\sqrt{z^2-1}\sqrt{(\cos \theta-i\epsilon)^2-1}\cos \varphi\right)
  \bigg]
  \!
  =\!  P_{\ell}\left(z\cos\theta+i\sqrt{z^2-1}\sin \theta \cos \varphi\right).
  \end{align*}
  The arguments of the two limits coincide according to~\cite[eq.~(14.28.1)]{nist}, with $z_1 = z$ and $z_2 = \cos \theta - i \epsilon$ in the notation of the reference.
  Thanks to the previous computations made in this proof, the equality of the limits leads to \eqref{eq:Pl(dx)}.
\end{proof}

\begin{remark}
To clarify, \textup{\cite[eq.~(14.28.1)]{nist}} only states that the equality between the limit arguments
holds when $\theta \in [0,\pi/2)$. 
Consequently, identities \textup{(\ref{addition theorem 2})} are established solely in this case.
The limitation likely arises because \textup{\cite{nist}} does not adopt the convention \textup{(\ref{convention})} in the definition of the associated Legendre functions. Thus, \textup{\cite[eq.~(14.28.1)]{nist}} is applicable only to values with positive real part.
Nevertheless, since all terms in \textup{(\ref{addition theorem 2})} are analytic in $(0,\pi)$ as functions of~$\theta$ (making explicit the dependence of $\mathbf{x}$ on $\theta$), these identities can be easily extended to this interval due to \textup{\cite[Th.\ 3.2.6]{ablowitz}}. Furthermore, they hold if $\mathbf{x}=(0,0,1)$, namely $\theta=\pi$: in fact $\mathsf{P}_{\ell}^m(-1)=(-1)^{\ell}\delta_{0,m}$ and, due to \textup{\cite[eq.~(14.7.17)]{nist}}, $P_{\ell}(-z)=(-1)^{\ell}P_{\ell}(z)$ for every $z\geq 1$.
\end{remark}

\paragraph{Generalized Jacobi--Anger identity}

\begin{theorem} \label{Theorem 2.10}
EPWs admit the following modal expansion:
for any $\mathbf{x} \in B_{1}$,
$\mathbf{y} = (\boldsymbol{\theta}, \psi, \zeta) \in Y$,
\begin{equation}\label{complex expansion}
  \textup{EW}_{\mathbf{y}}(\mathbf{x})
  = 4\pi \sum_{\ell=0}^{\infty} i^{\ell}\!\! \sum_{{m}=-\ell}^{\ell}
  \left[
    \sum_{{m'}=-\ell}^{\ell} \overline{D_{\ell}^{m',m}(\boldsymbol{\theta},\psi)}
    \gamma_{\ell}^{m'}i^{-m'}P_{\ell}^{m'}\!\!\left(1+\frac{\zeta}{2\kappa}\right)
  \right] j_\ell(\kappa |\mathbf{x}|) Y_\ell^m(\mathbf{x}/|\mathbf{x}|).
\end{equation}
\end{theorem}
\begin{proof}
  The result follows from applying~\eqref{addition theorem 2}
  after using~\eqref{pre jacobi anger 2} to obtain
  \begin{equation*}
    \textup{EW}_{\mathbf{y}}(\mathbf{x})
    := e^{i\kappa \mathbf{d}(\mathbf{y})\cdot \mathbf{x}}
     = \sum_{\ell=0}^{\infty} i^{\ell} (2\ell+1) j_{\ell}(\kappa |\mathbf{x}|)
     P_{\ell}\big(\mathbf{d}(\mathbf{y}) \cdot \mathbf{x}/|\mathbf{x}|\big)
     \qquad\forall \mathbf{x} \in B_{1},\ \forall\mathbf{y} \in Y.
  \end{equation*}  
\end{proof}

\begin{remark}
Thanks to \cite[eq.~(35)]{pendleton} and \eqref{spherical harmonic}, for any \((\boldsymbol{\theta},\psi) \in \Theta \times [0,2\pi)\) and \((\ell,m) \in \mathcal{I}\) it holds
  \begin{equation}
  D_{\ell}^{0,m}(\boldsymbol{\theta},\psi)=\sqrt{\frac{4\pi}{2\ell+1}}Y_\ell^m(\boldsymbol{\theta})=\sqrt{\frac{(\ell-m)!}{(\ell+m)!}}\,e^{im\theta_2}\mathsf{P}_{\ell}^m(\cos \theta_1),
  \label{D-matrix0}
  \end{equation}
  and moreover $P_{\ell}^{m}(1)=\delta_{0,m}$ due to \textup{(\ref{legendre2 polynomials})}. Hence, assuming $\zeta=0$ in \textup{(\ref{complex expansion})}, we recover the Jacobi--Anger expansion \textup{(\ref{jacobi-anger})} for PPWs for any $(\boldsymbol{\theta},\psi) \in \Theta \times [0,2\pi)$.
\end{remark}

\subsection{Modal analysis of plane waves}\label{ss:ModalAnalysis}
The Jacobi--Anger identity~\eqref{complex expansion} plays a crucial
role in the upcoming analysis.
As we develop below, this modal expansion with respect to the $b_\ell^m$ basis also offers direct insights on the
approximation properties of EPWs, hinting at why such waves are better suited
for approximating less regular Helmholtz solutions compared to PPWs.

\begin{figure}
\centering
\scalebox{0.72}{
\begin{tabular}{cccccc}
$\widehat{\textup{EW}}_{\ell}^m\!\!\left(\frac{\pi}{4},\cdot,0\right)$ & $\widehat{\textup{EW}}_{\ell}^m\!\!\left(\frac{\pi}{4},\frac{\pi}{2},60\right)$ & $\widehat{\textup{EW}}_{\ell}^m\!\!\left(\frac{\pi}{4},\frac{3\pi}{2},60\right)$ &
$\widehat{\textup{EW}}_{\ell}^m\!\!\left(\frac{\pi}{4},\frac{\pi}{4},120\right)$ & $\widehat{\textup{EW}}_{\ell}^m\!\!\left(\frac{\pi}{4},\frac{7\pi}{4},120\right)$ & $\widehat{\textup{EW}}_{\ell}^m\!\!\left(\frac{\pi}{4},0,180\right)$\\
\includegraphics[width=.2\linewidth]{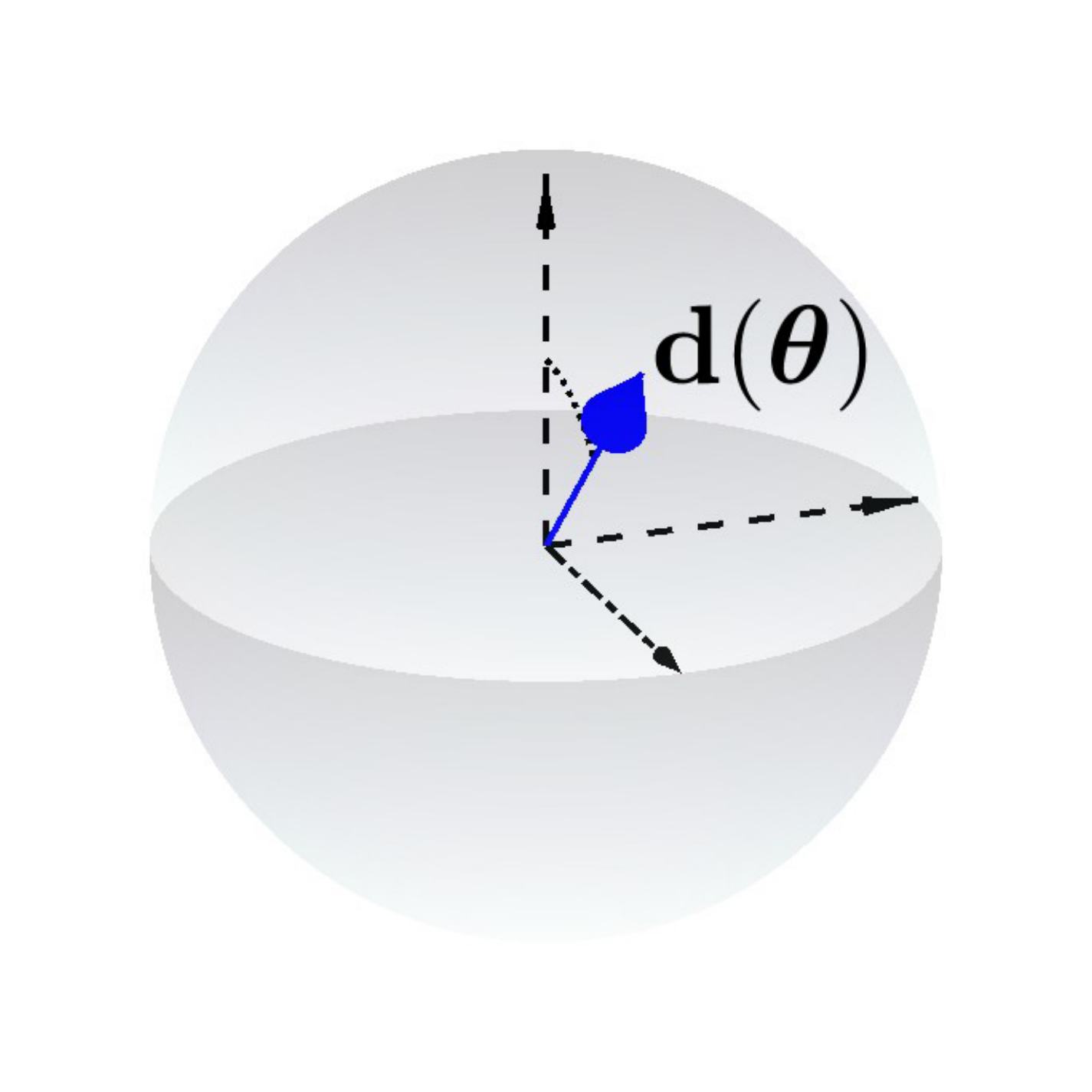} & \includegraphics[width=.2\linewidth]{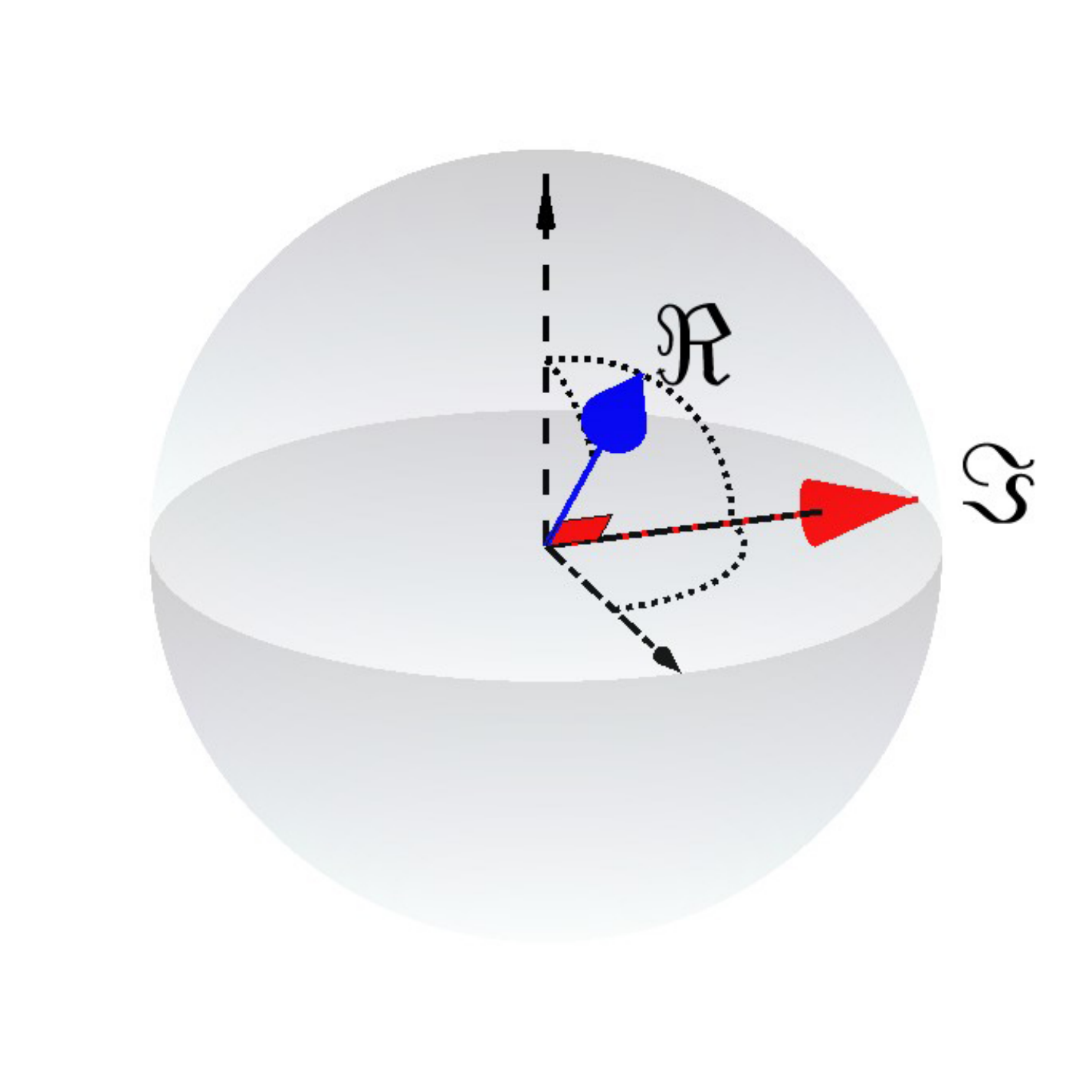} & \includegraphics[width=.2\linewidth]{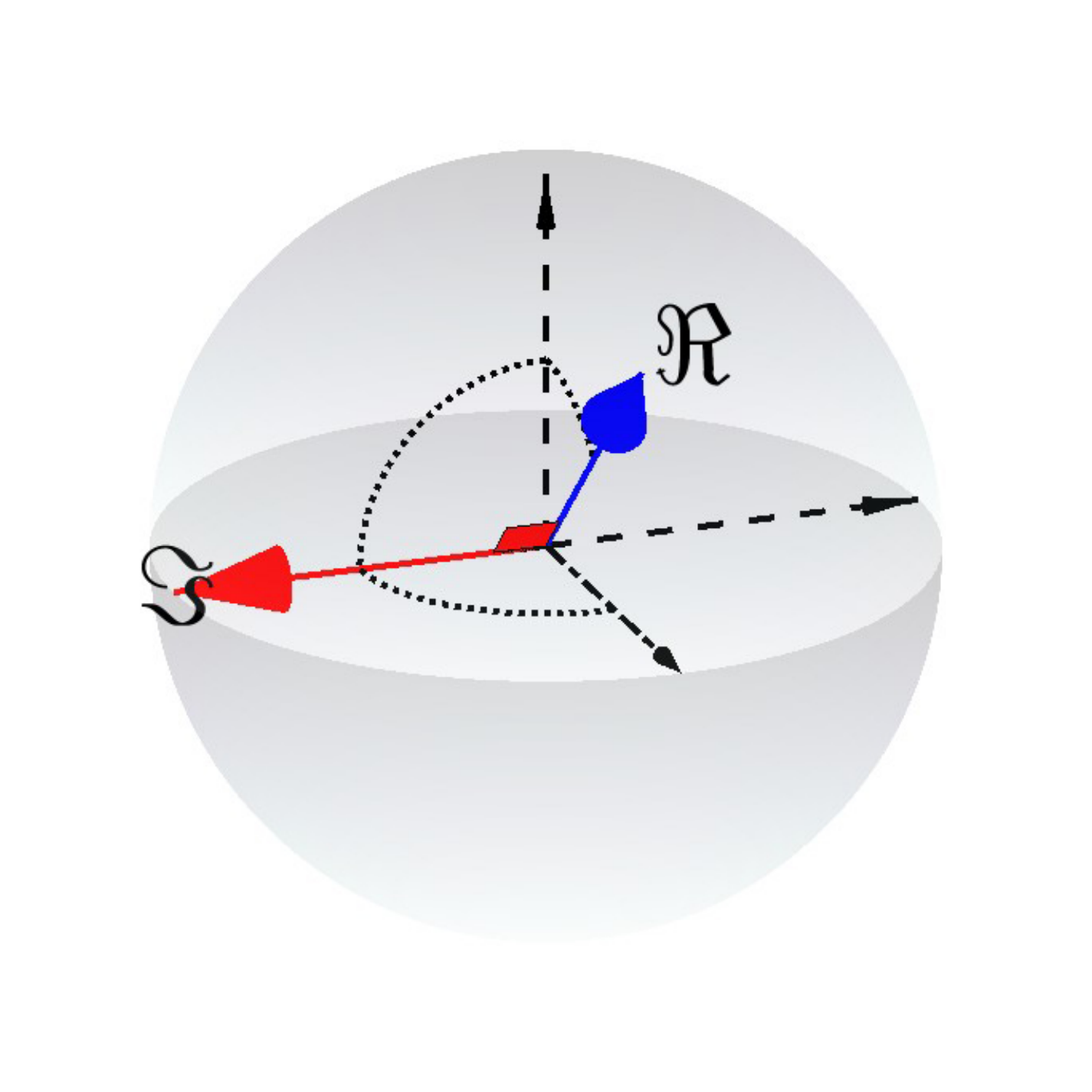} & 
\includegraphics[width=.2\linewidth]{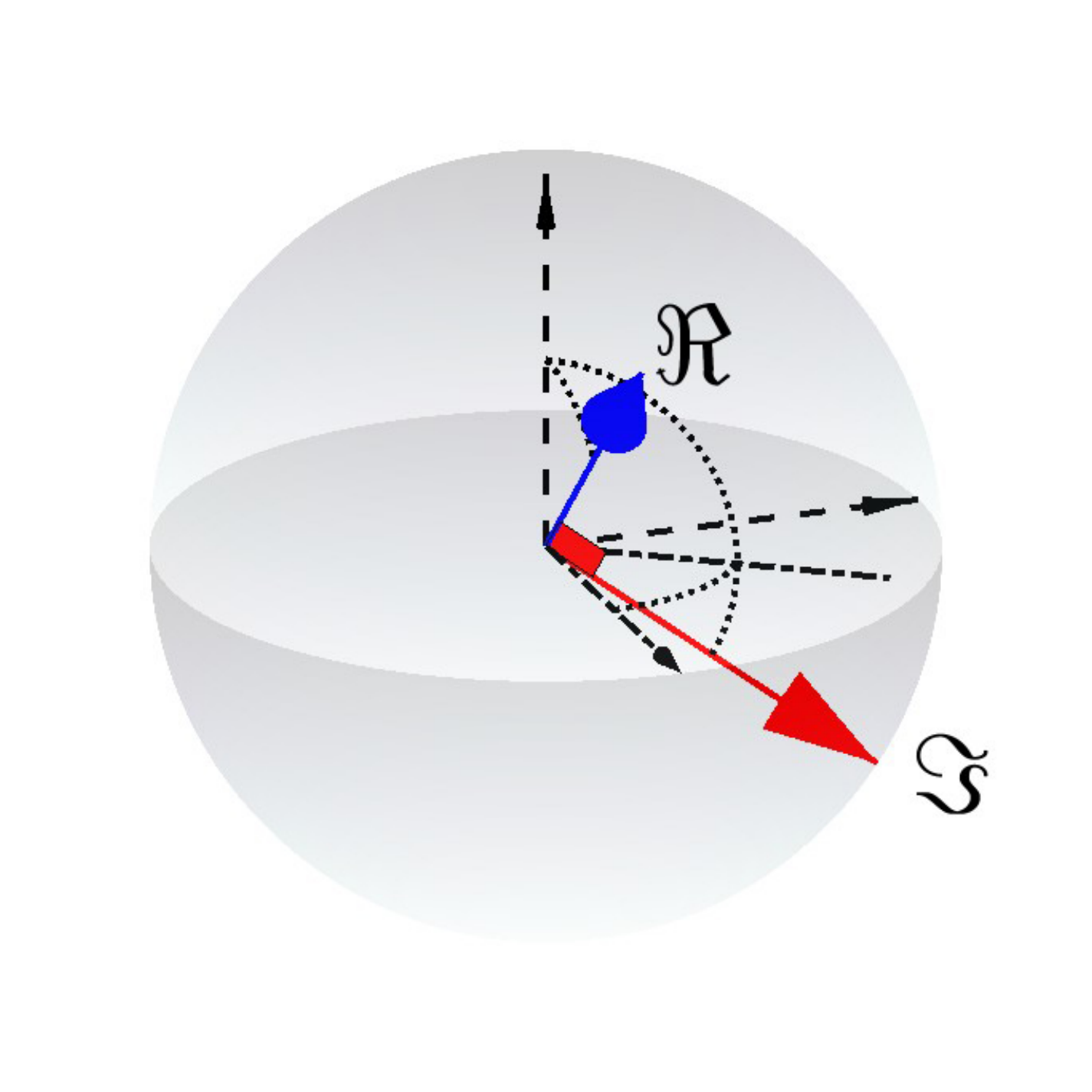} & \includegraphics[width=.2\linewidth]{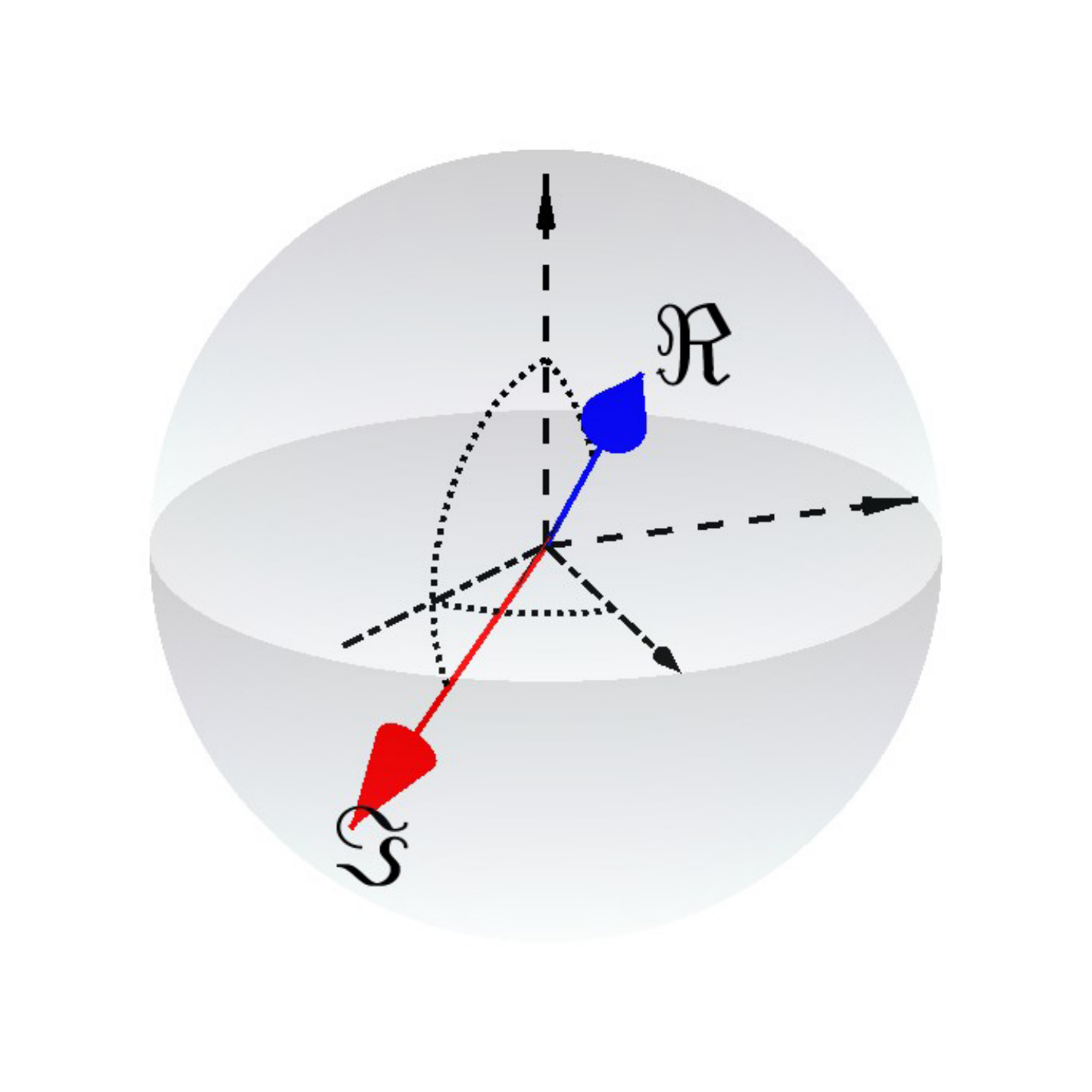} & \includegraphics[width=.2\linewidth]{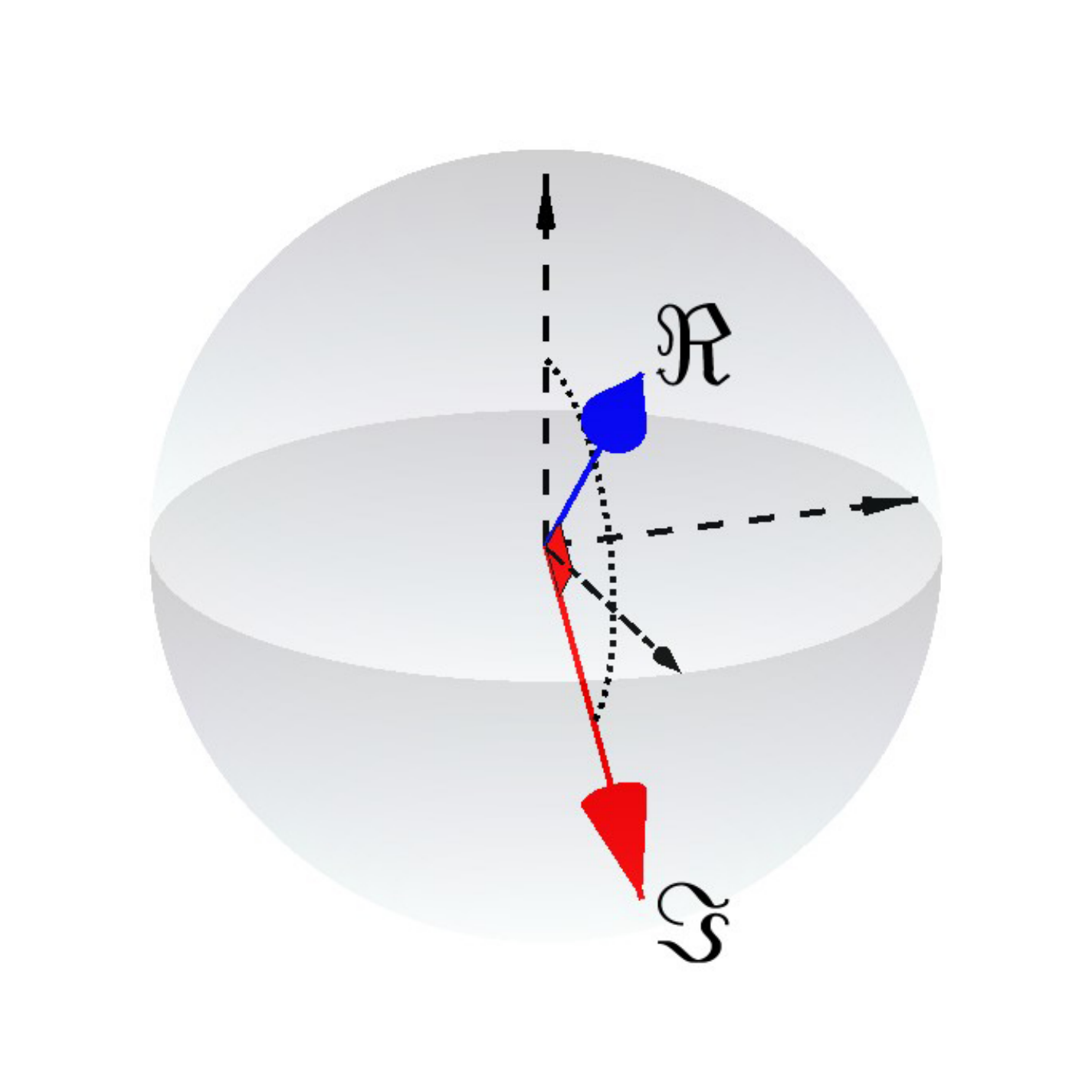}\\
\includegraphics[width=.2\linewidth]{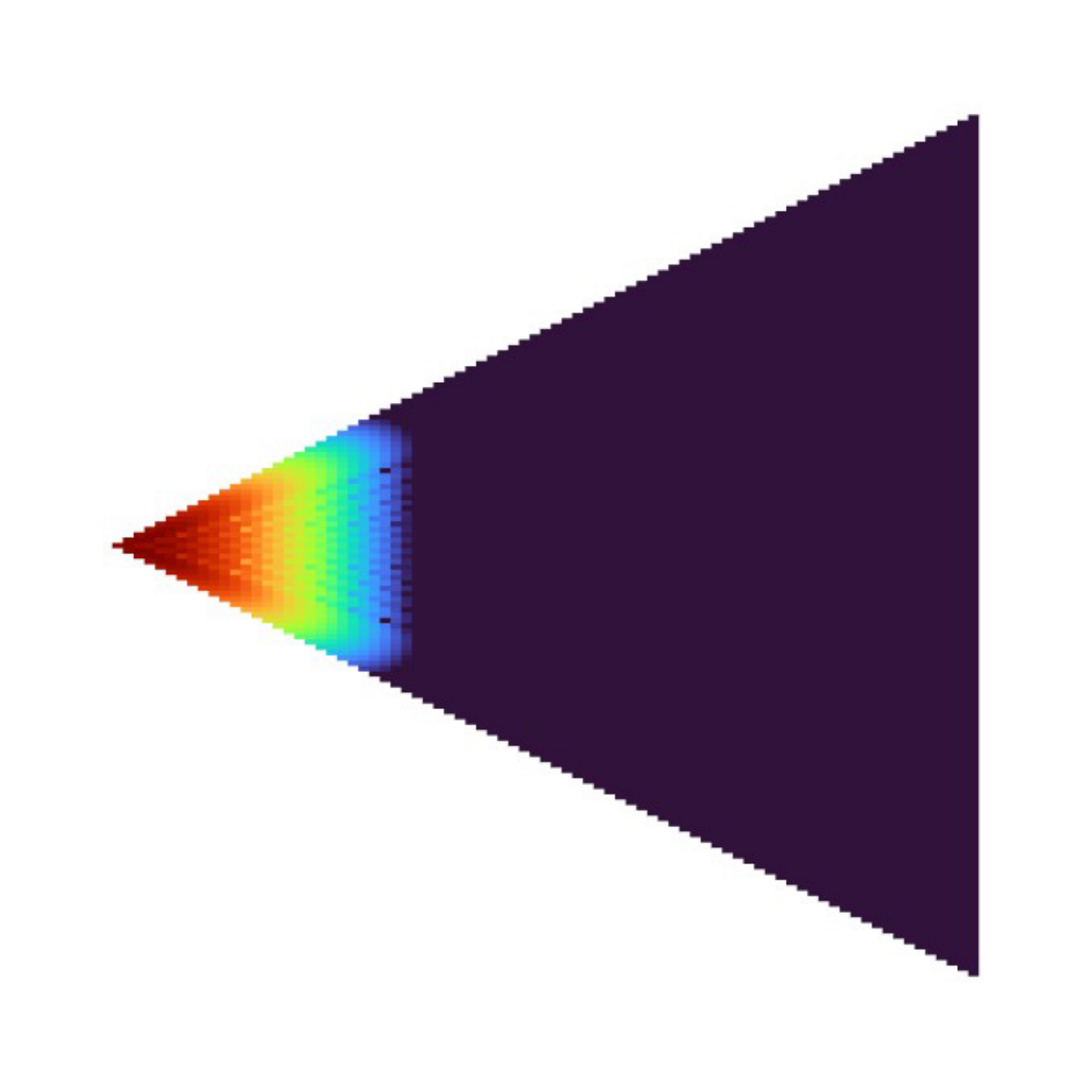} & \includegraphics[width=.2\linewidth]{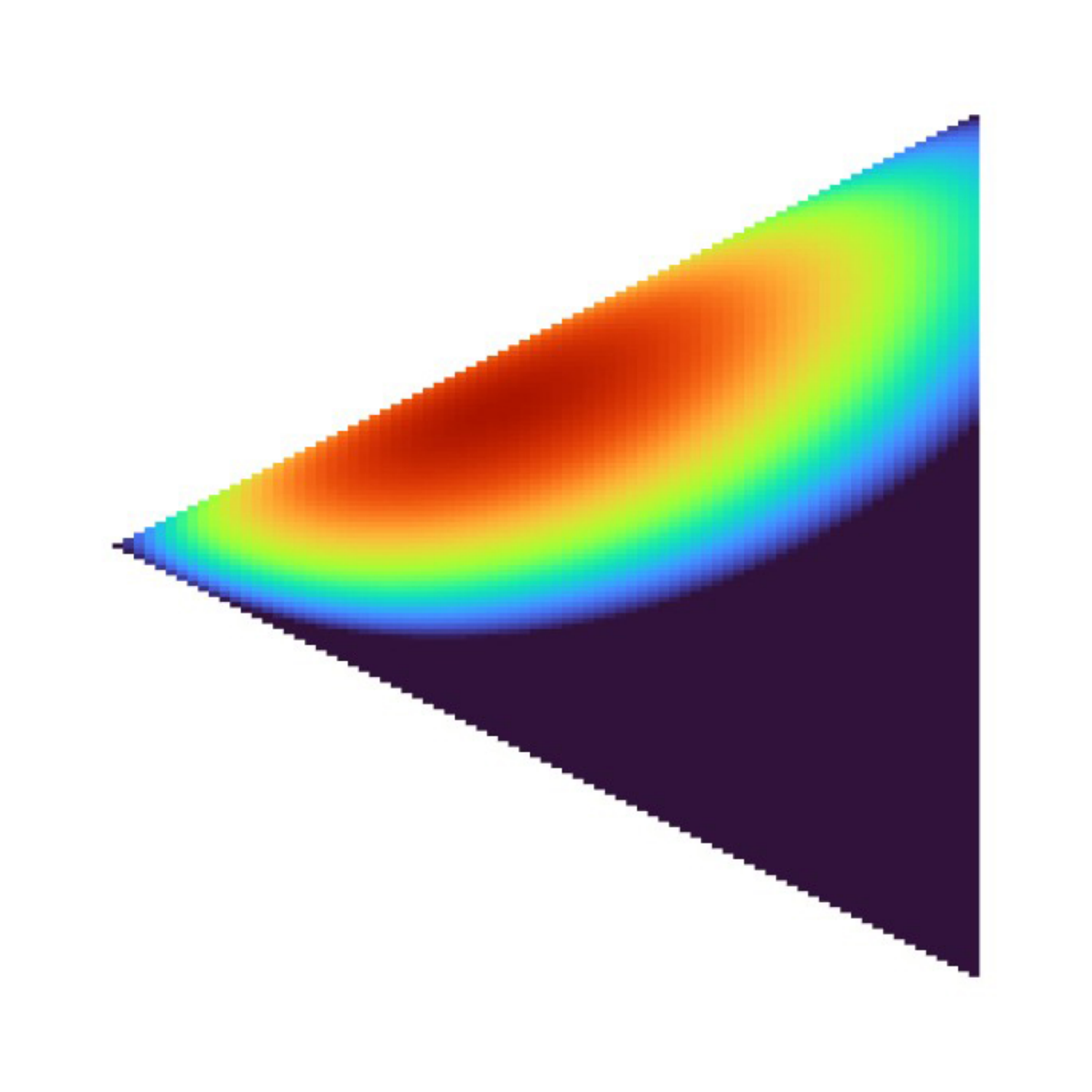} & \includegraphics[width=.2\linewidth]{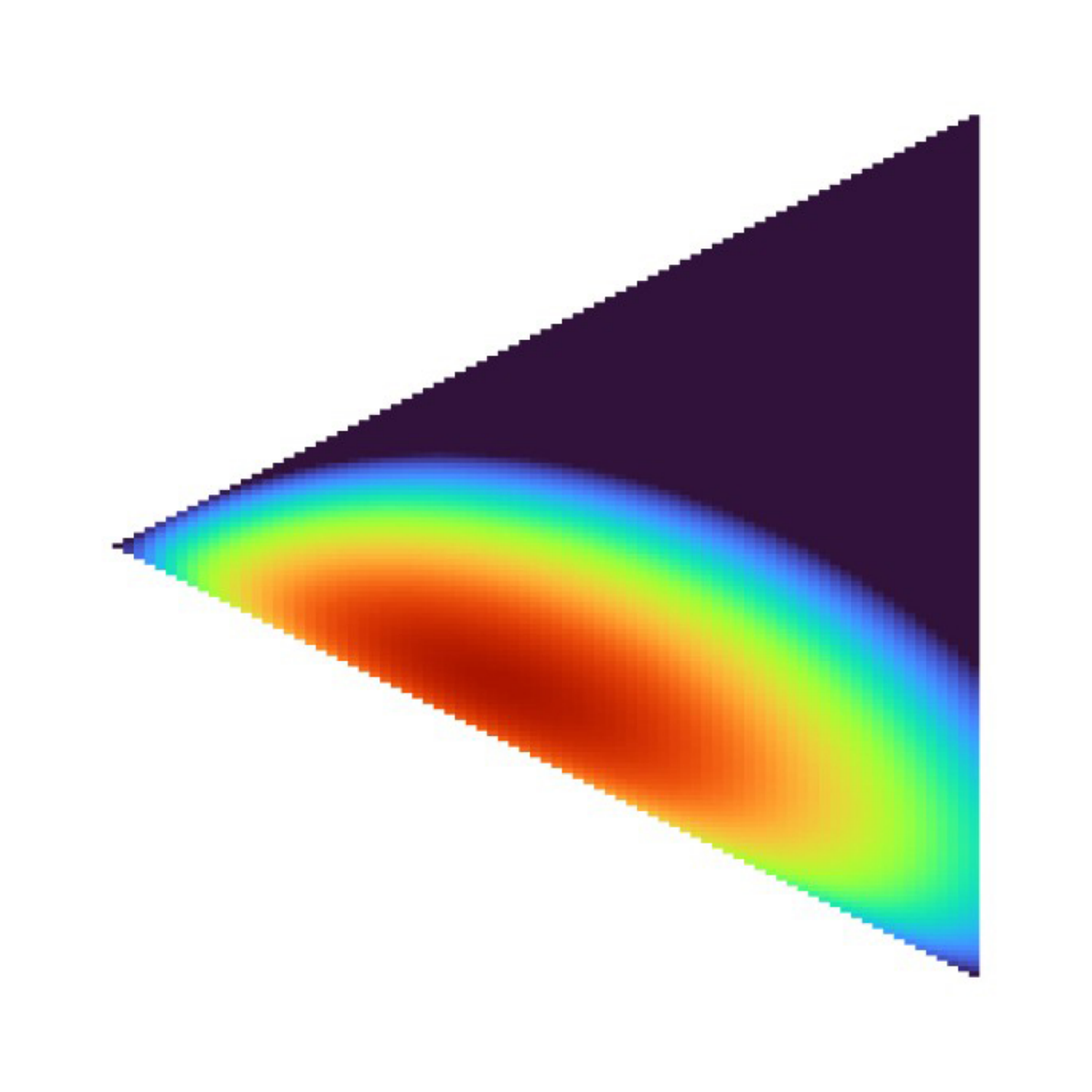} & 
\includegraphics[width=.2\linewidth]{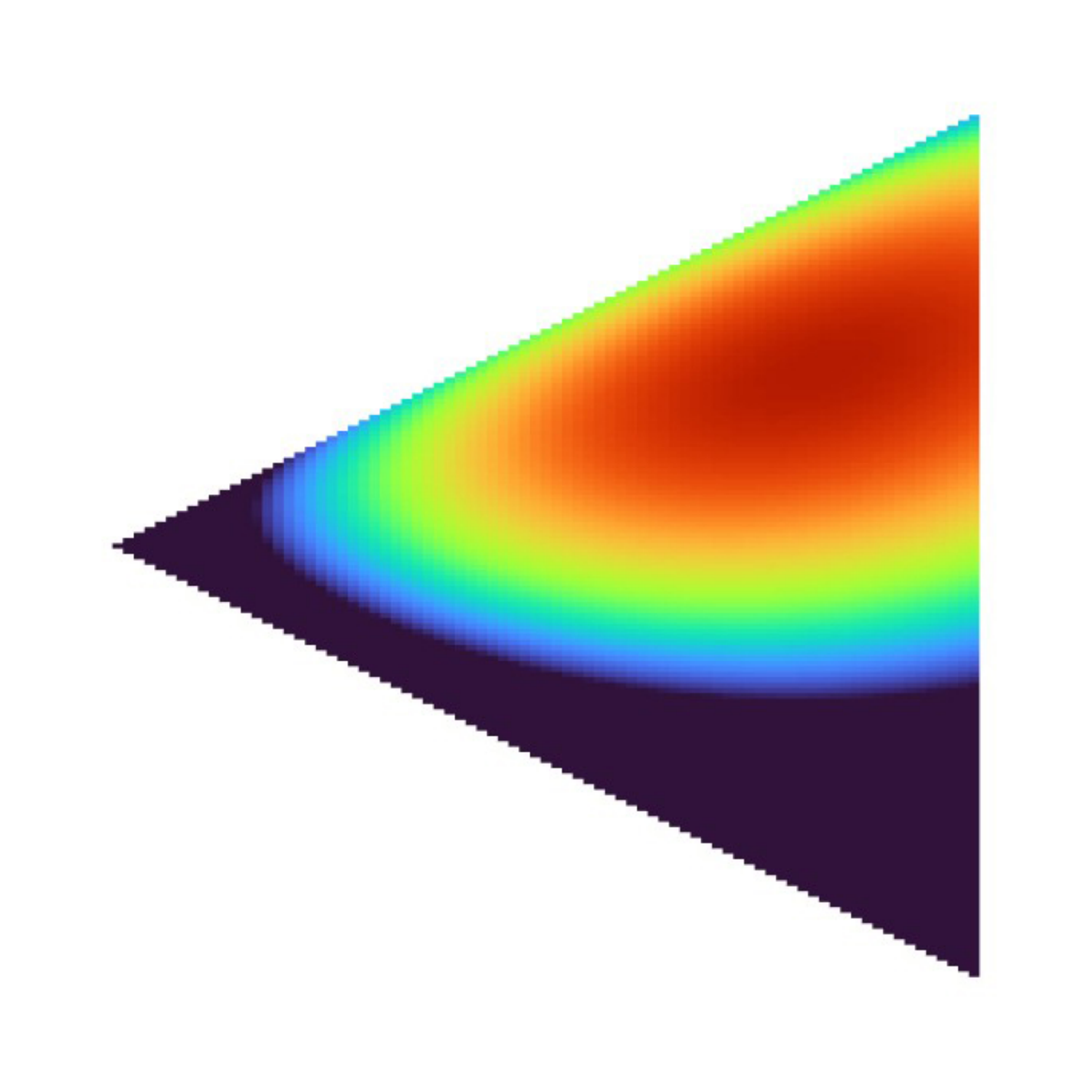} & \includegraphics[width=.2\linewidth]{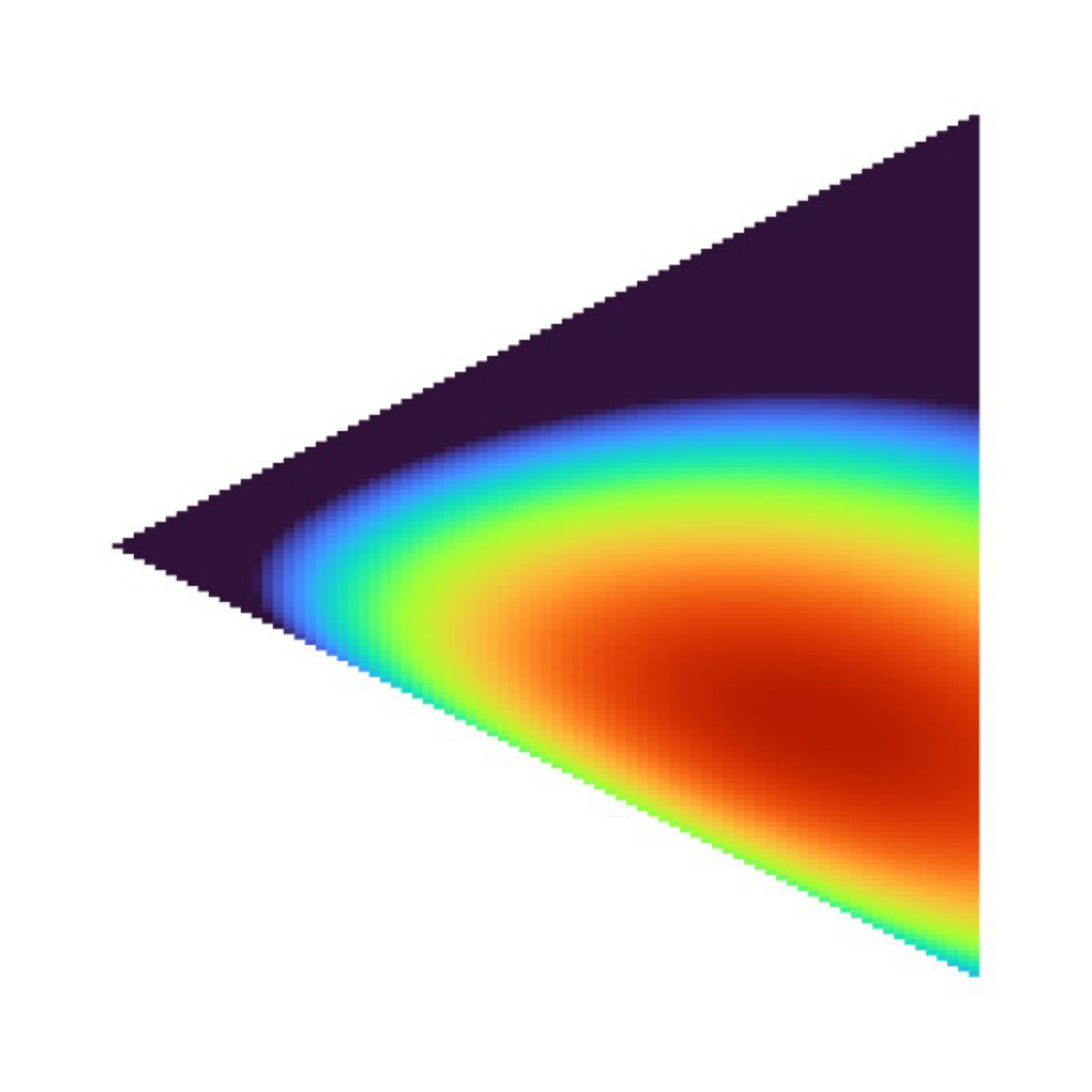} & \includegraphics[width=.2\linewidth]{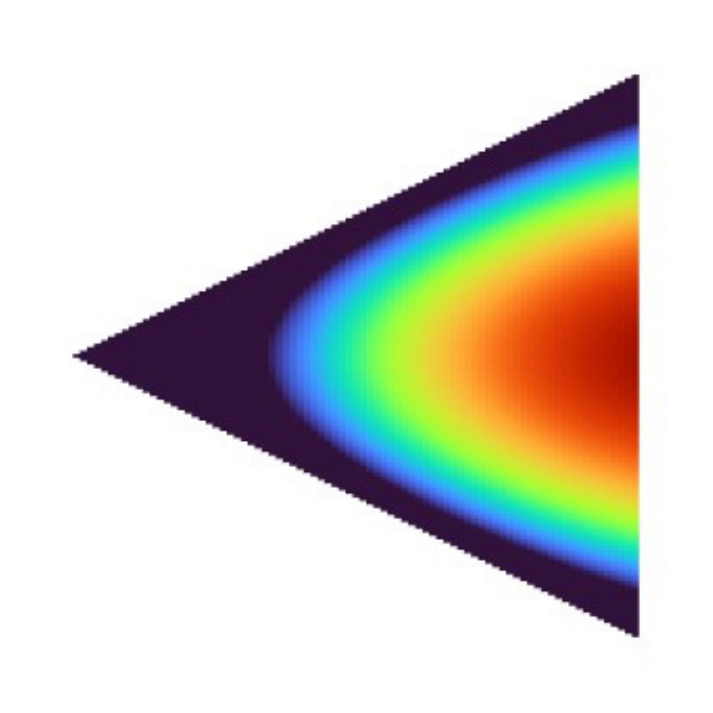}\\
\vspace{1mm}\\
\multicolumn{6}{c}{\includegraphics[width=1.3\linewidth]{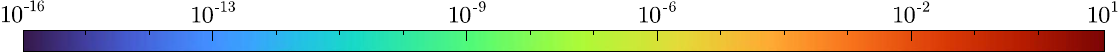}}\\
\end{tabular}
}
\caption{Modal analysis of EPWs. For each wave: (top) representations of both real and imaginary components of direction vectors $\mathbf{d}(\mathbf{y})$ with fixed angle $\theta_2=0$; (bottom) related distributions of the coefficients $\widehat{\textup{EW}}_{\ell}^m\!(\theta_1,\psi,\zeta)$ in (\ref{evanescent coefficients}). To mitigate the numerical cancellation issue of Wigner's formula \textup{(\ref{d matrix})}, we rely on \textup{\cite{feng}} using \textup{\cite[eq.~(4.26)]{galante}}. The index $\ell$ varies on the abscissa within the range $0 \leq \ell \leq 80$, while the index $m$ varies on the ordinate within the range $0 \leq |m| \leq \ell$, forming a triangle. The coefficients have been suitably normalized using a normalization factor described in the subsequent sections, see (\ref{isomorphism approximation sets}), which depends solely on $\zeta$. Wavenumber $\kappa=6$.
More plots of this kind can be seen in \textup{\cite[sect.~4.4]{galante}}.}
\label{figure 2.3}
\end{figure}

For conciseness, we use the notation $\mathbf{D}^{m}_{\ell}(\boldsymbol{\theta},\psi)$, for $0 \leq |m| \leq \ell$, to represent the columns of the Wigner D-matrix $D_{\ell}(\boldsymbol{\theta},\psi)$ in \eqref{DD matrix} and we let
\begin{equation}
\mathbf{P}_{\ell}(\zeta):=
\big(\gamma_{\ell}^{m}i^{m}P_{\ell}^{m}(1+\zeta/2\kappa)\big)_{m=-\ell}^{\ell} \in \mathbb{C}^{2\ell+1} \qquad \forall \ell \geq 0.
\label{vector P}
\end{equation}
Recalling~\eqref{b tilde definizione}, the Jacobi--Anger expansion~\eqref{complex expansion} can be written as
\begin{equation}\label{eq:jacobi-anger-final}
\textup{EW}_{\mathbf{y}}(\mathbf{x})=\sum_{\ell=0}^{\infty}\sum_{{m}=-\ell}^{\ell}\left[4 \pi i^{\ell}\beta_{\ell}^{-1}\overline{\mathbf{D}^{m}_{\ell}(\boldsymbol{\theta},\psi) \cdot \mathbf{P}_{\ell}(\zeta)} \right] b_{\ell}^m(\mathbf{x})
\qquad\forall \mathbf{x} \in B_{1},\forall\mathbf{y} \in Y.
\end{equation}
The moduli of the coefficients in the above modal expansion, namely
\begin{equation}
\widehat{\textup{EW}}_{\ell}^m\!(\theta_1,\psi,\zeta)\!:=\!\left|\left(\textup{EW}_{\mathbf{y}},b_{\ell}^m\right)_{\mathcal{B}}\right|\!=\!\frac{4\pi}{\beta_{\ell}}\left|\sum_{m'=-\ell}^{\ell}\!\!\gamma_{\ell}^{m'}i^{-m'}\!d_{\ell}^{\,m',m}(\theta_1)e^{-im'\psi}P_{\ell}^{m'}\!\!\left(1+\frac{\zeta}{2\kappa}\right)\right|,
\label{evanescent coefficients}
\end{equation}
are depicted in Figure~\ref{figure 2.3}.
We also define, for any $\ell \geq 0$ and
$\mathbf{y}=(\boldsymbol{\theta},\psi,\zeta) \in Y$,
\begin{equation}
b_\ell[\mathbf{y}]:=\widehat{\textup{EW}}_{\ell}^{-1}(\zeta)\,\tilde{b}_\ell[\mathbf{y}], \quad \text{where} \quad \tilde{b}_\ell[\mathbf{y}]:=\sum_{m=-\ell}^{\ell}\left(\textup{EW}_{\mathbf{y}},b_{\ell}^m\right)_{\mathcal{B}}b_{\ell}^m, \quad \widehat{\textup{EW}}_{\ell}(\zeta):=\big\|\tilde{b}_\ell[\mathbf{y}]\big\|_{\mathcal{B}}. 
\label{yuy}
\end{equation}
In fact, thanks to (\ref{yuy}), we can expand the EPWs as
\begin{equation*}
\textup{EW}_{\mathbf{y}}
=\sum_{\ell=0}^{\infty}\sum_{m=-\ell}^{\ell}\left(\textup{EW}_{\mathbf{y}},b_{\ell}^m\right)_{\mathcal{B}}b_{\ell}^m
=\sum_{\ell=0}^{\infty}\tilde b_{\ell}[\mathbf{y}]
=\sum_{\ell=0}^{\infty}\widehat{\textup{EW}}_{\ell}(\zeta)\,b_{\ell}[\mathbf{y}],
\end{equation*}
where $b_{\ell}[\mathbf{y}] \in \text{span}\{b_{\ell}^m\}_{m=-\ell}^{\ell}$ are orthonormal and
\begin{equation}
\widehat{\textup{EW}}_{\ell}(\zeta)=\,\left(\sum_{m=-\ell}^{\ell}\left|\left(\textup{EW}_{\mathbf{y}},b_{\ell}^m\right)_{\mathcal{B}}\right|^2\right)^{1/2}=\left(\sum_{m=-\ell}^{\ell}\left[\widehat{\textup{EW}}_{\ell}^m\!(\theta_1,\psi,\zeta)\right]^2\right)^{1/2}=\,\frac{4\pi}{\beta_{\ell}}\left|\mathbf{P}_{\ell}(\zeta)\right|.
\label{evanescent l2 coefficients}
\end{equation}
The last equality in (\ref{evanescent l2 coefficients}) holds due to (\ref{evanescent coefficients}) and the unitarity condition \cite[sect.~4.1, eq.~(6)]{quantumtheory}.
Figure~\ref{figure 2.4} shows the coefficient distribution (\ref{evanescent l2 coefficients}) for various values of $\zeta$.

\begin{figure}
\centering
\includegraphics[trim=5 60 5 0,clip,width=6.9cm]{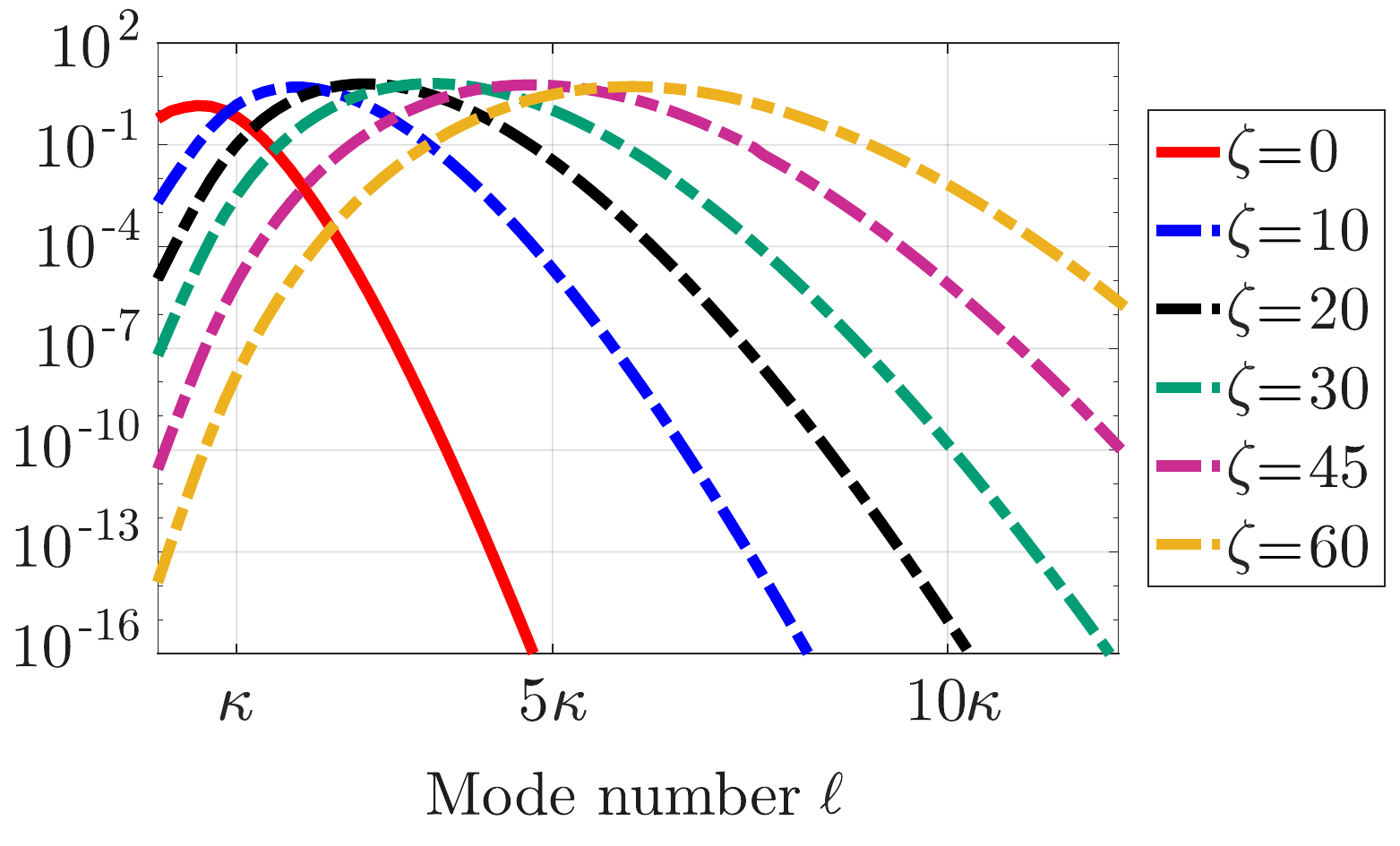}
\caption{Modal analysis of EPWs: distributions of the coefficients (\ref{evanescent l2 coefficients}) for various values of the evanescence parameter $\zeta$.
For each $\zeta$, this involves computing the $\ell^2$-norms along the vertical segments of the coefficient distribution triangles, such as those in Figure~\ref{figure 2.3}.
As $\zeta$ increases, higher-$\ell$ Fourier modes can be encompassed, transitioning away from the propagative case represented by $\zeta=0$.
The coefficients have been suitably normalized using a normalization factor described in the subsequent sections, see (\ref{isomorphism approximation sets}). Wavenumber $\kappa=6$.}
\label{figure 2.4}
\end{figure}

\begin{remark}
Assuming $\zeta=0$, the functions $b_{\ell}\left[\mathbf{y}\right]$ coincide, up to the multiplicative constant $i^{\ell}$, with the spherical waves $b_{\ell}^0$ rotated according to the \textup{PPW} angles $\boldsymbol{\theta} \in \Theta$. This is readily checked thanks to \textup{(\ref{spherical harmonic})} and \textup{(\ref{b tilde definizione})}, along with the identities \textup{(\ref{wigner property})} and \textup{(\ref{D-matrix0})}.
\end{remark}

If we consider PPWs and thus assume $\zeta=0$, the coefficients
(\ref{evanescent coefficients}) are independent of $\psi$.
For any $\boldsymbol{\theta} \in \Theta$, the PPW coefficients decay
super-exponentially fast in the evanescent-mode regime $\ell \gg \kappa$. 
This is visible in the leftmost  
triangle of Figure~\ref{figure 2.3} and in the continuous line in Figure~\ref{figure 2.4}.
Consequently, any PPW approximation of Helmholtz solutions with a high-$\ell$
Fourier modal content will require exponentially large coefficients and
cancellation to capture these modes, resulting in numerical instability.
This assertion is made precise later in
Lemma~\ref{lem:ppw-lack-continuous-stability} and Lemma~\ref{Lemma 4.4}.

On the contrary, by tuning the evanescence parameters $\psi$ and $\zeta$, the
Fourier modal content of the EPWs can be shifted to higher Fourier regimes.
Specifically, raising $\zeta$ enables us to reach higher degrees (larger values
of $\ell$), while varying $\psi$ allows us to cover different orders $m$.
As a consequence we expect EPWs with large \(\zeta\) to be able to approximate
high Fourier modes with relatively small coefficients, curing the numerical
instability experienced with PPWs.
However, accurately selecting the evanescence parameters $\psi$ and $\zeta$ to
build reasonably sized approximation spaces remains a significant challenge.
We will tackle this issue in the following sections.

\section{Stable continuous approximation}\label{sec:herglotz transform}

PPWs and EPWs families are naturally indexed by continuous sets, the
parametric domains \(\Theta\) and \(Y\) in Definition~\ref{def:EPW}.
Although in applications a finite discrete subset is selected, it is fruitful
to first analyse the properties of the continuous set of plane waves.
This is the purpose of this section which first introduces the notion of
\emph{stable continuous approximation}.
We then present the \emph{Herglotz density space}, showing its close link with
the Helmholtz solution space through the Jacobi--Anger
identity~\eqref{eq:jacobi-anger-final}.
This connection leads to the definition of the \emph{Herglotz transform}, an
integral operator enabling the representation of any Helmholtz solution in the
unit ball as a continuous superposition of EPWs.
This continuous representation is proven to be stable, as opposed to PPWs,
which fail to produce such a result due to their inability to stably represent
evanescent spherical modes, i.e.\ solutions dominated by high-order Fourier modes.

\subsection{The concept of stable continuous approximation}\label{ss:StableContApprox}

Let $(X,\mu)$ be a \(\sigma\)-finite measure space and denote by $L^2_\mu(X)$ the corresponding Lebesgue space.
Following \cite[sect.~5.6]{christensen}, we define a \emph{Bessel family} in the Hilbert space $\mathcal B$ as a set
\begin{equation*} 
\mathbf{\Phi}_X:=\{\phi_{x}\}_{x \in X} \subset \mathcal{B}, \qquad \text{such that}  
\qquad \int_X|\left(u,\phi_{x}\right)_{\mathcal{B}}|^2\textup{d}\mu(x)\leq B\|u\|^2_{\mathcal{B}} \qquad \forall u \in \mathcal{B},
\end{equation*}
for some $B>0$.
For any such $\mathbf{\Phi}_X$, the  
\emph{synthesis} operator can be defined as:
\begin{equation*}
\Tcont{\!X}{\,}:L_{\mu}^2(X) \rightarrow \mathcal{B}, \quad v \mapsto \int_X v(x)\phi_{x}\textup{d}\mu(x).
\end{equation*}

\begin{definition}[Stable continuous approximation]\label{def:SCA}
  The Bessel family $\mathbf{\Phi}_X$ is said to be a stable continuous
  approximation for $\mathcal{B}$ if, for any tolerance $\eta >0$, there
  exists a stability constant $C_{\textup{cs}} \geq 0$ such that
  \begin{equation}\label{condition stable continuous representation}
    \forall u \in \mathcal{B},\ \exists v \in L_{\mu}^2(X):\quad
    \left\|u-\Tcont{\!X}{\,}v\right\|_{\mathcal{B}} \leq \eta \|u\|_{\mathcal{B}},
    \quad \text{and} \quad
    \|v\|_{L^2_{\mu}(X)}\leq C_{\textup{cs}} \|u\|_{\mathcal{B}}.
  \end{equation}
\end{definition}

A stable continuous approximation allows approximating any Helmholtz solution
to a given accuracy as an expansion $\Tcont{\!X}{\,}v$, where the
density $v$ has a bounded norm in $L_{\mu}^2(X)$.

\subsection{Herglotz density space}

We consider the space $L_{\nu}^2(Y)$ on the EPW parametric domain $Y$, with
the positive measure $\nu$ given by
\begin{equation}
\textup{d}\nu(\mathbf{y}):=\textup{d}\sigma(\boldsymbol{\theta})\,\textup{d}\psi\,w(\zeta)\textup{d}\zeta, \qquad \text{where} \qquad w(\zeta):=\zeta^{1/2}e^{-\zeta} \quad \forall \zeta \in [0,+\infty),
\label{weight}
\end{equation}
and $\sigma$ is the standard measure on $\mathbb{S}^2$. The $L_{\nu}^2$
Hermitian product and the associated norm are
\begin{equation*}
(v,u)_{\mathcal{A}}:=\int_{Y}v(\mathbf{y})\overline{u(\mathbf{y})}\textup{d}\nu(\mathbf{y}), \qquad
\|v\|^2_{\mathcal{A}}:=(v,v)_{\mathcal{A}} \qquad \qquad \forall v,u \in L_{\nu}^2(Y).
\end{equation*}
Let us define a proper subspace of $L_{\nu}^2(Y)$, denoted by $\mathcal{A}$ and named \emph{space of Herglotz densities}.

\begin{definition}[Herglotz densities] \label{definition herglotz densities}
We define, for any $(\ell,m) \in \mathcal{I}$
\begin{equation}
a_\ell^m:=\alpha_{\ell}\tilde{a}_\ell^m, \quad \text{where} \quad \tilde{a}_\ell^m(\mathbf{y}):=\mathbf{D}^{m}_{\ell}(\boldsymbol{\theta},\psi) \cdot \mathbf{P}_{\ell}(\zeta)\quad \forall \mathbf{y}=(\boldsymbol{\theta},\psi,\zeta) \in Y, \quad \alpha_{\ell}:=\|\tilde{a}_\ell^m\|^{-1}_\mathcal{A},
\label{a tilde definizione}
\end{equation}
where $\mathbf{D}^{m}_{\ell}(\boldsymbol{\theta},\psi)$ and $\mathbf{P}_{\ell}(\zeta)$ are defined in \textup{(\ref{vector P})}, and $\mathcal{A}:=\overline{\textup{span}\{a_\ell^m\}_{(\ell,m) \in \mathcal{I}}}^{\|\cdot\|_{\mathcal{A}}} \subsetneq L_{\nu}^2(Y)$.
\end{definition}

Just like the spherical waves (\ref{b tilde definizione}), the Herglotz
densities also depend on $(\ell,m) \in \mathcal{I}$, while the normalization
coefficient $\alpha_{\ell}$ is independent of $m$, as will be clarified later
(see Lemma \ref{Lemma 3.4}).
The wavenumber $\kappa$ appears explicitly in the definition (\ref{vector P})
of $\mathbf{P}_{\ell}(\zeta)$, hence each $a_{\ell}^m$ depends on it.
Some densities $a_\ell^m$, weighted by $w^{1/2}$, are shown in Figure \ref{figure alm}; additional plots are available in \cite[Fig.~5.1]{galante}.

\begin{figure}[t]
\centering
\begin{tabular}{c|cccccc}
\, & \, $\zeta=10^{\text{-}3}$ & \, $\zeta=10^{\text{-}2}$ & \, $\zeta=10^{\text{-}1}$ & \, $\zeta=10^{0}$ & \, $\zeta=10^{1}$ & \, $\zeta=10^{2}$\\
\hline
\raisebox{5.15ex}{$m=0$} & \includegraphics[trim=30 30 30 30,clip,width=.118\linewidth]{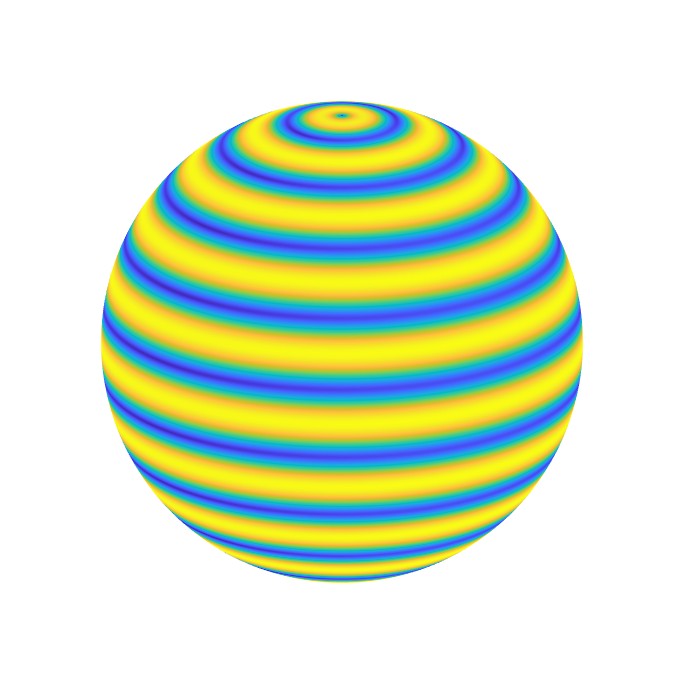} & \includegraphics[trim=30 30 30 30,clip,width=.118\linewidth]{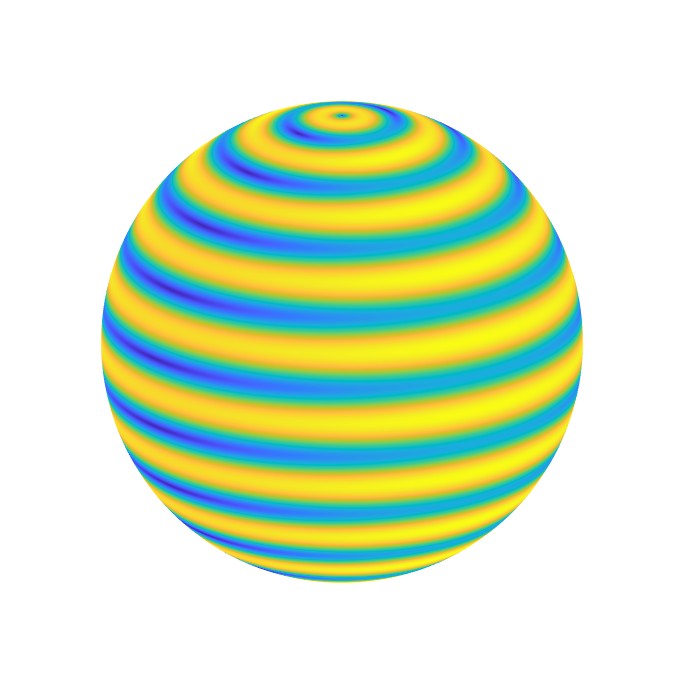} & \includegraphics[trim=30 30 30 30,clip,width=.118\linewidth]{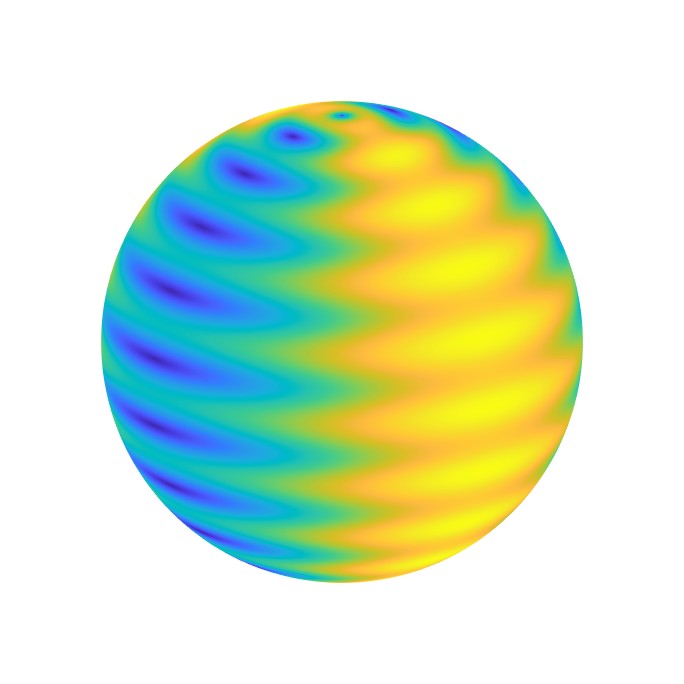} & \includegraphics[trim=30 30 30 30,clip,width=.118\linewidth]{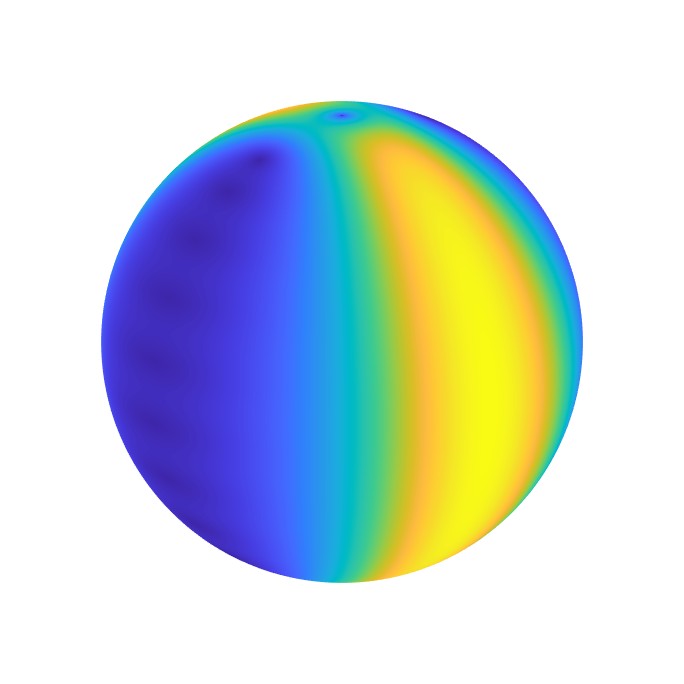} & \includegraphics[trim=30 30 30 30,clip,width=.118\linewidth]{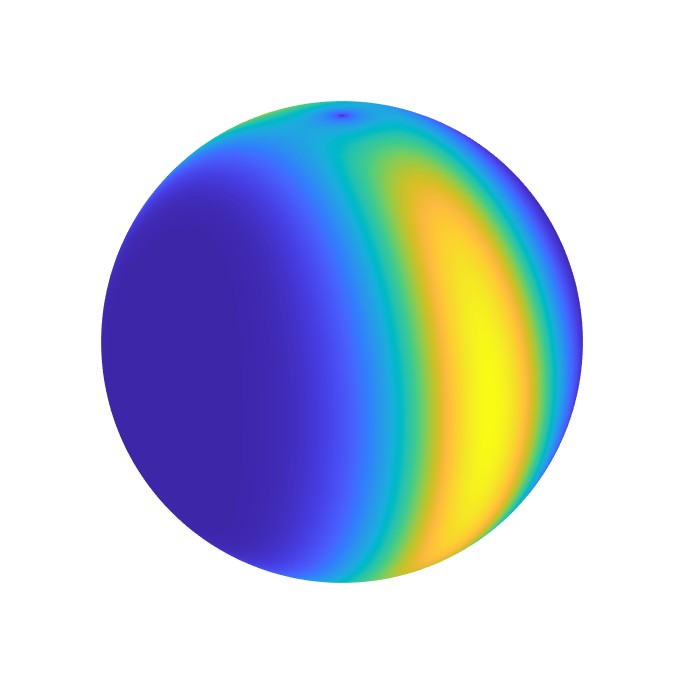} & \includegraphics[trim=30 30 30 30,clip,width=.118\linewidth]{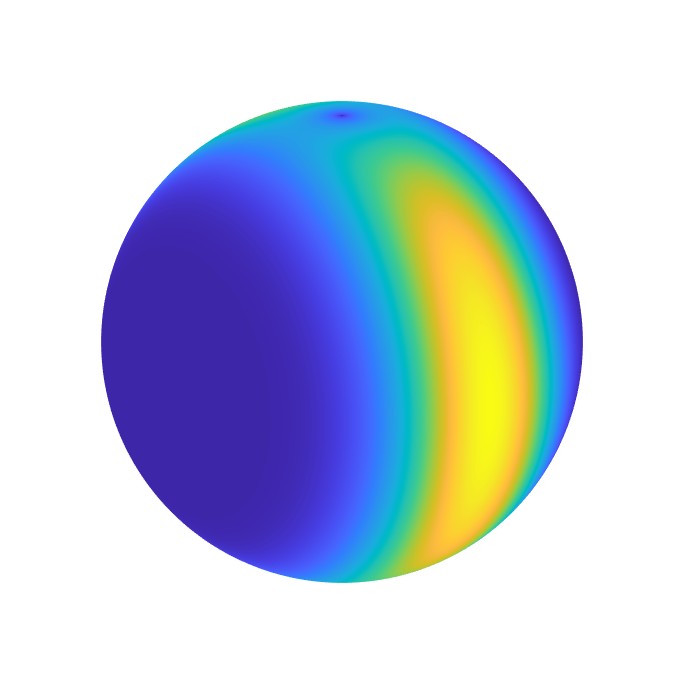}\\
\raisebox{5.15ex}{$m=10$} & \includegraphics[trim=30 30 30 30,clip,width=.118\linewidth]{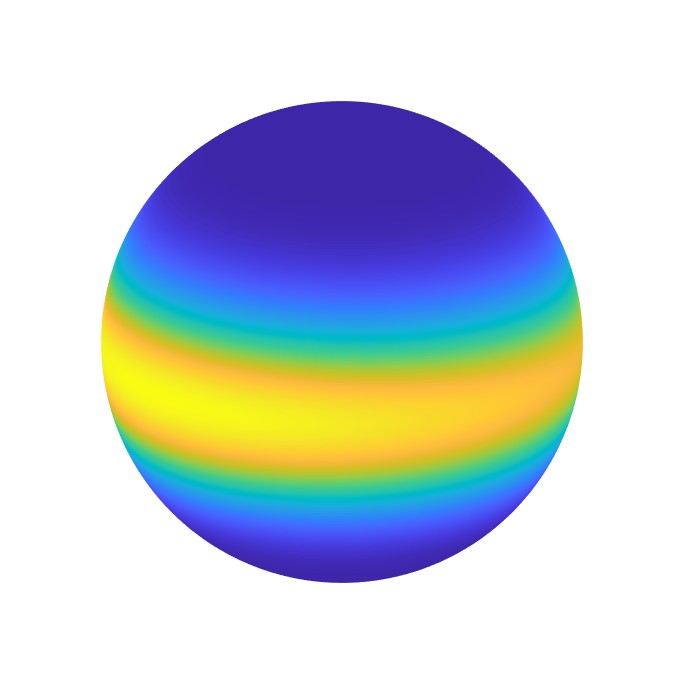} & \includegraphics[trim=30 30 30 30,clip,width=.118\linewidth]{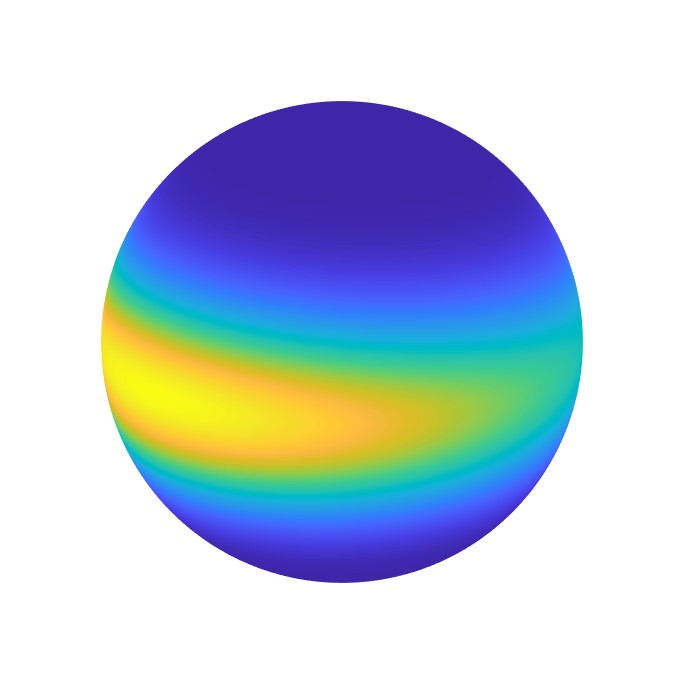} & \includegraphics[trim=30 30 30 30,clip,width=.118\linewidth]{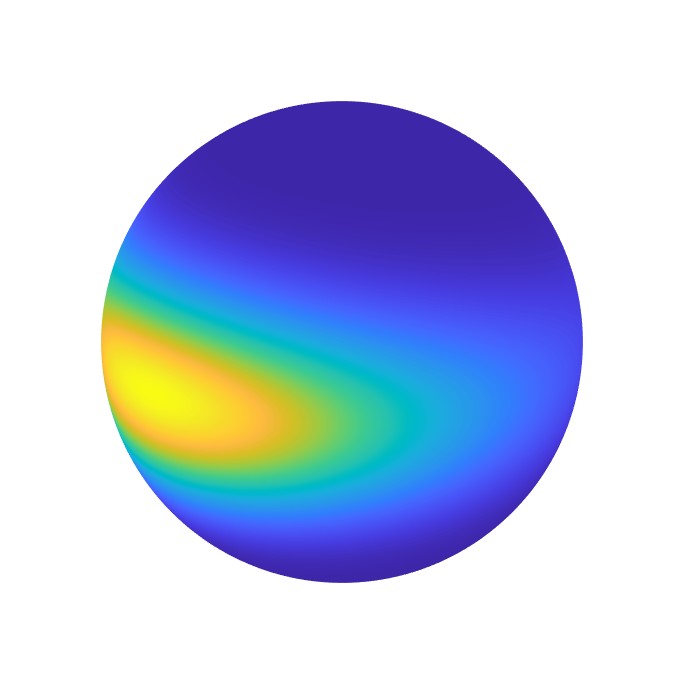} & \includegraphics[trim=30 30 30 30,clip,width=.118\linewidth]{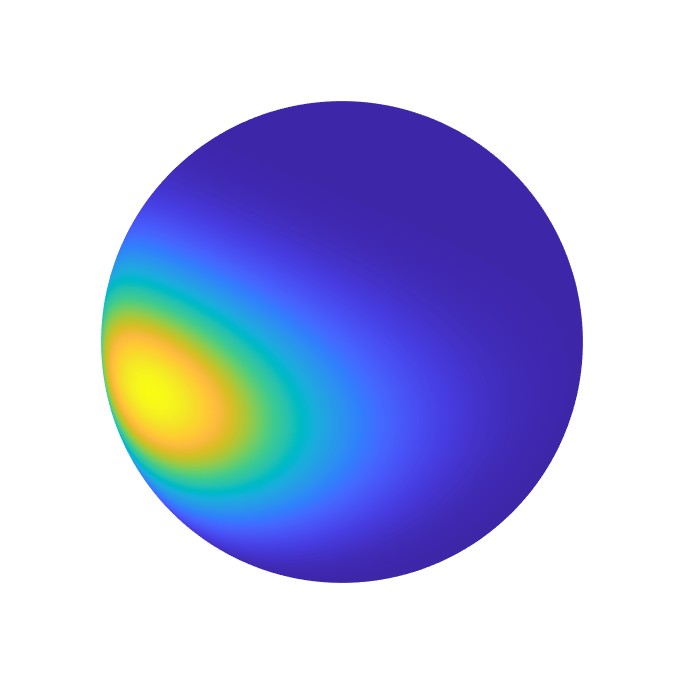} & \includegraphics[trim=30 30 30 30,clip,width=.118\linewidth]{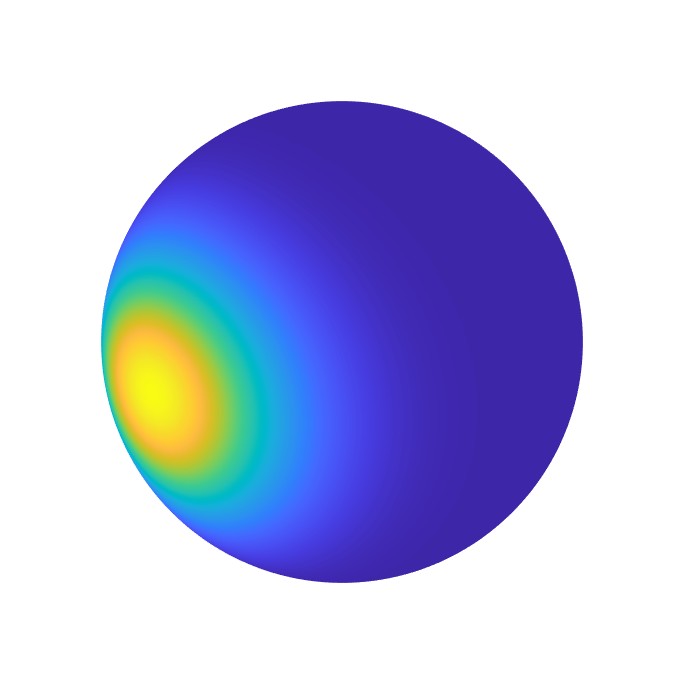} & \includegraphics[trim=30 30 30 30,clip,width=.118\linewidth]{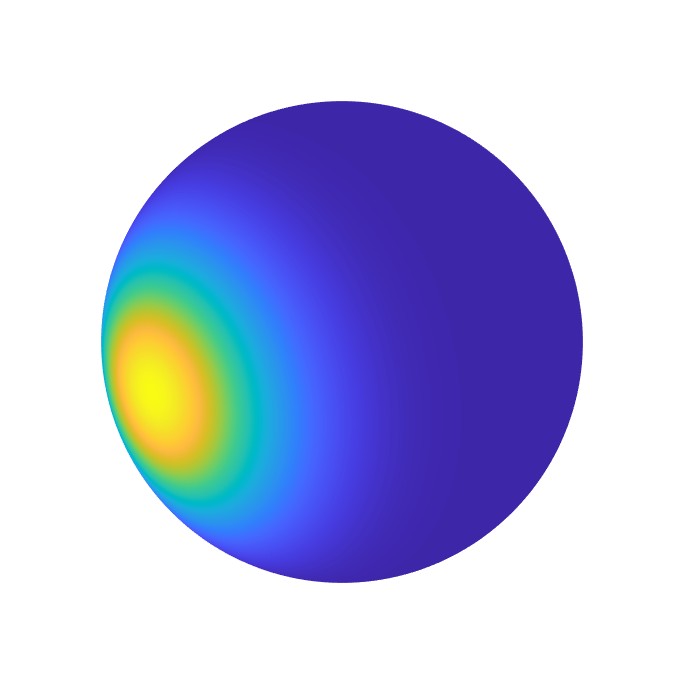}\\
$\max$ & \footnotesize{$3\times10^{\text{-}6}$} & \footnotesize{$7.1\times10^{\text{-}6}$} & \footnotesize{$3\times10^{\text{-}5}$} & \footnotesize{$3.6\times10^{\text{-}4}$} & \footnotesize{$4.1\times10^{\text{-}2}$} & \footnotesize{$5\times10^{\text{-}14}$}\\
\hline
\end{tabular}
\caption{
Plots of $|w^{1/2} a_\ell^m|$, with $\ell = 10$, two values of $m$, and varying $\zeta$. The function depends on $(\theta_1, \theta_3, \zeta)$ and is evaluated on a sphere. Each column shares a color scale; maxima are shown in the last row.
Wavenumber $\kappa=6$.}
\label{figure alm}
\end{figure}

\begin{lemma} \label{Lemma 3.5}
\!The space $(\mathcal{A},\!\|\cdot\|_{\mathcal{A}})$ is a Hilbert space and the family $\{a_{\ell}^m\}_{\!(\ell,m) \in \mathcal{I}}$ is a Hilbert basis:
\begin{equation*}
  (a_{\ell}^m,a_q^n)_{\mathcal{A}}=\delta_{\ell, q}\delta_{m,n}\quad
  \forall (\ell,m),(q,n) \in \mathcal{I},\qquad \text{and} \qquad
  v=\sum_{(\ell,m) \in \mathcal{I}}(v,a_{\ell}^m)_{\mathcal{A}}\,a_{\ell}^m
  \quad \forall v \in \mathcal{A}.
\end{equation*}
\end{lemma}

Using Definition~\ref{definition herglotz densities}, the Jacobi–Anger expansion~\eqref{eq:jacobi-anger-final} takes the simple form
\begin{equation}
\textup{EW}_{\mathbf{y}}(\mathbf{x})=\sum_{(\ell,m) \in \mathcal{I}}\tau_{\ell}\,\overline{a_{\ell}^m(\mathbf{y})}\,b_{\ell}^m(\mathbf{x}), \qquad \text{where} \qquad \tau_{\ell}:=4\pi i^{\ell}(\alpha_{\ell}\beta_{\ell})^{-1} \qquad \forall \ell \geq 0.
\label{tau jacobi-anger}
\end{equation}
The formula (\ref{tau jacobi-anger}) holds a crucial role as it establishes a link between the spherical wave basis (\ref{b tilde definizione}) of the space $\mathcal{B}$ and the Herglotz-density basis (\ref{a tilde definizione}) of the space $\mathcal{A}$ through EPWs in (\ref{evanescent wave}).

The behavior of \(\tau_{\ell}\) for large \(\ell\) will intervene in the
upcoming analysis.
To study this, we start with a lemma useful to analyze the asymptotic
behavior of the normalization coefficients $\alpha_{\ell}$.

\begin{lemma}
We have for all $(\ell,m) \in \mathcal{I}$ and $z \geq 1$
\begin{equation}
(z-1)^{\ell}\leq \frac{\sqrt{\pi}(\ell-m)!P_{\ell}^m(z)}{2^{\ell}\Gamma\left(\ell+1/2\right)}\leq (z+1)^{\ell}.
\label{disuguaglianza lemma}
\end{equation}
\end{lemma}
\begin{proof}
Due to (\ref{legendre2 polynomials}), $(\ell+m)!P_{\ell}^{-m}(z)=(\ell-m)!P_{\ell}^{m}(z)$ for every $(\ell,m) \in \mathcal{I}$, allowing us to assume $m \geq 0$ from here on.
From \cite[eq.~(5.5.5)]{nist}, we have
\begin{equation*}
    \Gamma\left(\ell+\frac{1}{2}\right)=\frac{\sqrt{\pi}(2\ell)!}{2^{2\ell}\ell!} \qquad \forall \ell \geq 0.
\end{equation*}
Using this identity together with equation \eqref{sum legendre expansion}, it follows
\begin{align*}
A_{\ell}^m(z)&:=\frac{\sqrt{\pi}(\ell-m)!P_{\ell}^m(z)}{2^{\ell}\Gamma\left(\ell+1/2\right)}=\frac{2^{\ell}\ell!(\ell-m)!P_{\ell}^m(z)}{(2\ell)!}\\
&=\binom{2\ell}{\ell+m}^{-1}\left(z-1\right)^{\ell}\sum_{k=0}^{\ell-m}\binom{\ell}{k}\binom{\ell}{m+k}\left(\frac{z+1}{z-1}\right)^{m/2+k}.
\end{align*}
Thanks to the Vandermonde identity \cite[eq.~(1)]{Sokal} and $\binom ab=\binom a{a-b}$, we derive:
\begin{align*}
A_{\ell}^m(z) &\leq \binom{2\ell}{\ell+m}^{-1}\!\left(\frac{z+1}{z-1}\right)^{\ell-m/2}\!\!\!(z-1)^{\ell}\sum_{k=0}^{\ell-m}\binom{\ell}{k}\binom{\ell}{m+k}=\left(\frac{z-1}{z+1}\right)^{m/2}\!\!\!(z+1)^{\ell}\leq (z+1)^{\ell},\\
A_{\ell}^m(z) &\geq \binom{2\ell}{\ell+m}^{-1}(z-1)^{\ell}\sum_{k=0}^{\ell-m}\binom{\ell}{k}\binom{\ell}{m+k}=(z-1)^{\ell}. \qedhere
\end{align*}
\end{proof}

\begin{remark}
Numerical evidence suggests that a sharper upper bound in \textup{(\ref{disuguaglianza lemma})} is $z^{\ell}$.
\end{remark}

After a pre-asymptotic regime up to $\ell \approx \kappa$, the coefficients $\alpha_{\ell}$ exhibit super-exponential decay with respect to $\ell$. The specific asymptotic behavior is detailed in the next  
lemma.

\begin{lemma} \label{Lemma 3.4}
For a constant $c(\kappa)$ only depending on $\kappa$,
we have
\begin{equation}
\alpha_{\ell} \sim c(\kappa)\left(\frac{e\kappa}{2} \right)^{\ell}\ell^{-\left(\ell+\frac{1}{2}\right)} \qquad \text{as}\ \ell \rightarrow \infty.
\label{behaviour alpha_lm}
\end{equation}
\end{lemma}

\begin{proof}
We have that
\begin{align*}
\|\tilde{a}_{\ell}^m\|^2_{\mathcal{A}}&=\int_{Y}|\mathbf{D}^{m}_{\ell}(\boldsymbol{\theta},\psi) \cdot \mathbf{P}_{\ell}(\zeta)|^2\textup{d}\nu(\mathbf{y})\\
&=\sum_{m'=-\ell}^{\ell}\int_{\Theta}\int_0^{2\pi}\left|D_{\ell}^{m',m}(\boldsymbol{\theta},\psi)\right|^2\textup{d}\sigma(\boldsymbol{\theta})\textup{d}\psi\int_0^{+\infty}\left[\gamma_{\ell}^{m'}P_{\ell}^{m'}\left(1+\zeta/2\kappa\right)\right]^2w(\zeta)\textup{d}\zeta\\
&=\frac{8\pi^2}{2\ell+1}\sum_{m'=-\ell}^{\ell}\int_0^{+\infty}\left[\gamma_{\ell}^{m'}P_{\ell}^{m'}\left(1+\zeta/2\kappa\right)\right]^2\zeta^{1/2}e^{-\zeta}\textup{d}\zeta. \numberthis \label{1 lemma 5.3}
\end{align*}
In what follows, we study the integral in (\ref{1 lemma 5.3}), henceforth denoted by $B_{\ell}^{m'}$. Thanks to (\ref{disuguaglianza lemma}):
\begin{align*}
B_{\ell}^{m'}&\geq\left(\frac{2^{\ell}\gamma_{\ell}^{m'}\Gamma(\ell+1/2)}{\sqrt{\pi}(\ell-m')!}\right)^2\int_{0}^{+\infty}\left(\frac{\zeta}{2\kappa}\right)^{2\ell}\zeta^{1/2}e^{-\zeta}\textup{d}\zeta\\
&=\frac{1}{4\pi^2\kappa^{2\ell}}\frac{(2\ell+1)\Gamma^{\,2}(\ell+1/2)}{(\ell+m')!(\ell-m')!} \,\Gamma\left(2\ell+\frac{3}{2}\right)=:C_{\ell}^{m'}, \numberthis \label{3 lemma 5.3}
\end{align*}
and analogously
\begin{align*}
B_{\ell}^{m'}&\leq\left(\frac{2^{\ell}\gamma_{\ell}^{m'}\Gamma(\ell+1/2)}{\sqrt{\pi}(\ell-m')!}\right)^2\int_{0}^{+\infty}\left(\frac{\zeta}{2\kappa}+2\right)^{2\ell}\zeta^{1/2}e^{-\zeta}\textup{d}\zeta\\
&=\left(\frac{2^{\ell}\gamma_{\ell}^{m'}\Gamma(\ell+1/2)}{\sqrt{\pi}(\ell-m')!}\right)^2\int_{4\kappa}^{+\infty}\left(\frac{\eta}{2\kappa}\right)^{2\ell}(\eta-4\kappa)^{1/2}e^{-(\eta-4\kappa)}\textup{d}\eta\\
&<\frac{2\ell+1}{4\pi^2}\frac{2^{2\ell}\Gamma^{\,2}(\ell+1/2)}{(\ell+m')!(\ell-m')!} \int_{0}^{+\infty}\left(\frac{\eta}{2\kappa}\right)^{2\ell}\eta^{1/2}e^{-\eta}e^{4\kappa}\textup{d}\eta\\
&=\frac{e^{4\kappa}}{4\pi^2\kappa^{2\ell}}\frac{(2\ell+1)\Gamma^{\,2}(\ell+1/2)}{(\ell+m')!(\ell-m')!} \,\Gamma\left(2\ell+\frac{3}{2}\right)=e^{4\kappa}C_{\ell}^{m'}, \numberthis \label{2 lemma 5.3}
\end{align*}
where we used \cite[eq.~(5.2.1)]{nist} and $\eta=\zeta+4\kappa$.
Thanks to the Stirling’s formula \cite[eq.~(5.11.7)]{nist}, it is easily checked that as $\ell \rightarrow +\infty$
\begin{equation*}
    \Gamma\left(2\ell+3/2\right)\sim \sqrt{2\pi}e^{-2\ell}\left(2\ell\right)^{2\ell+1}, \quad \Gamma^{\,2}\left(\ell+1/2\right)\sim 2\pi e^{-2\ell}\ell^{2\ell}, \quad (\ell+m')!(\ell-m')!\sim 2\pi e^{-2\ell}\ell^{2\ell+1},
\end{equation*}
where $|m'|\leq \ell$ is fixed, and hence
\begin{equation*}
\frac{(2\ell+1)\Gamma^{\,2}(\ell+1/2)}{(\ell+m')!(\ell-m')!}\Gamma\left(2\ell+\frac{3}{2}\right) \sim 2\sqrt{2\pi}e^{-2\ell}(2\ell)^{2\ell+1} \qquad \text{as }\ell \rightarrow +\infty.
\end{equation*}
By combining (\ref{3 lemma 5.3}) and (\ref{2 lemma 5.3}), it follows that, as $\ell \rightarrow +\infty$, there exists a constant $c_1(\kappa)$, only dependent on the wavenumber $\kappa$, such that
\begin{equation*}
C_{\ell}^{m'} \sim \frac{\sqrt{2}}{\pi\sqrt{\pi}}\left(\frac{2}{e\kappa}\right)^{2\ell}\ell^{2\ell+1} \qquad \Rightarrow \qquad B_{\ell}^{m'} \sim c_1(\kappa)\left(\frac{2}{e\kappa}\right)^{2\ell}\ell^{2\ell+1}.
\end{equation*} 
Moreover, also $\|\tilde{a}_{\ell}^m\|^2_{\mathcal{A}}$ has the same behavior as $B_{\ell}^{m'}$ at infinity: in fact, thanks to (\ref{1 lemma 5.3}), we have
\begin{align*}
\|\tilde{a}_{\ell}^m\|^2_{\mathcal{A}} &\sim \frac{4\pi^2}{\ell}\sum_{m'=-\ell}^{\ell}c_1(\kappa)\left(\frac{2}{e\kappa}\right)^{2\ell}\ell^{2\ell+1} \sim c_2(\kappa)\left(\frac{2}{e\kappa}\right)^{2\ell}\ell^{2\ell+1},
\end{align*}
for some constant $c_2(\kappa)$ only dependent on $\kappa$; the claimed result (\ref{behaviour alpha_lm}) follows from~\eqref{a tilde definizione}.
\end{proof}

From the asymptotics provided in Lemma \ref{Lemma 2.6} and Lemma \ref{Lemma 3.4}, we can deduce the next result.

\begin{corollary} 
The coefficients $\tau_\ell$ in \eqref{tau jacobi-anger} are uniformly bounded in $\ell$,
namely
\begin{equation}
\tau_{-}:=\inf_{\ell \geq 0}|\tau_{\ell}|>0 \qquad \text{and} \qquad \tau_{+}:=\sup_{\ell \geq 0}|\tau_{\ell}|<\infty.
\label{uniform bounds tau}
\end{equation}
\end{corollary}

\begin{figure}
\centering
\begin{minipage}{.44\textwidth}
  \centering
  \includegraphics[width=0.95\textwidth]{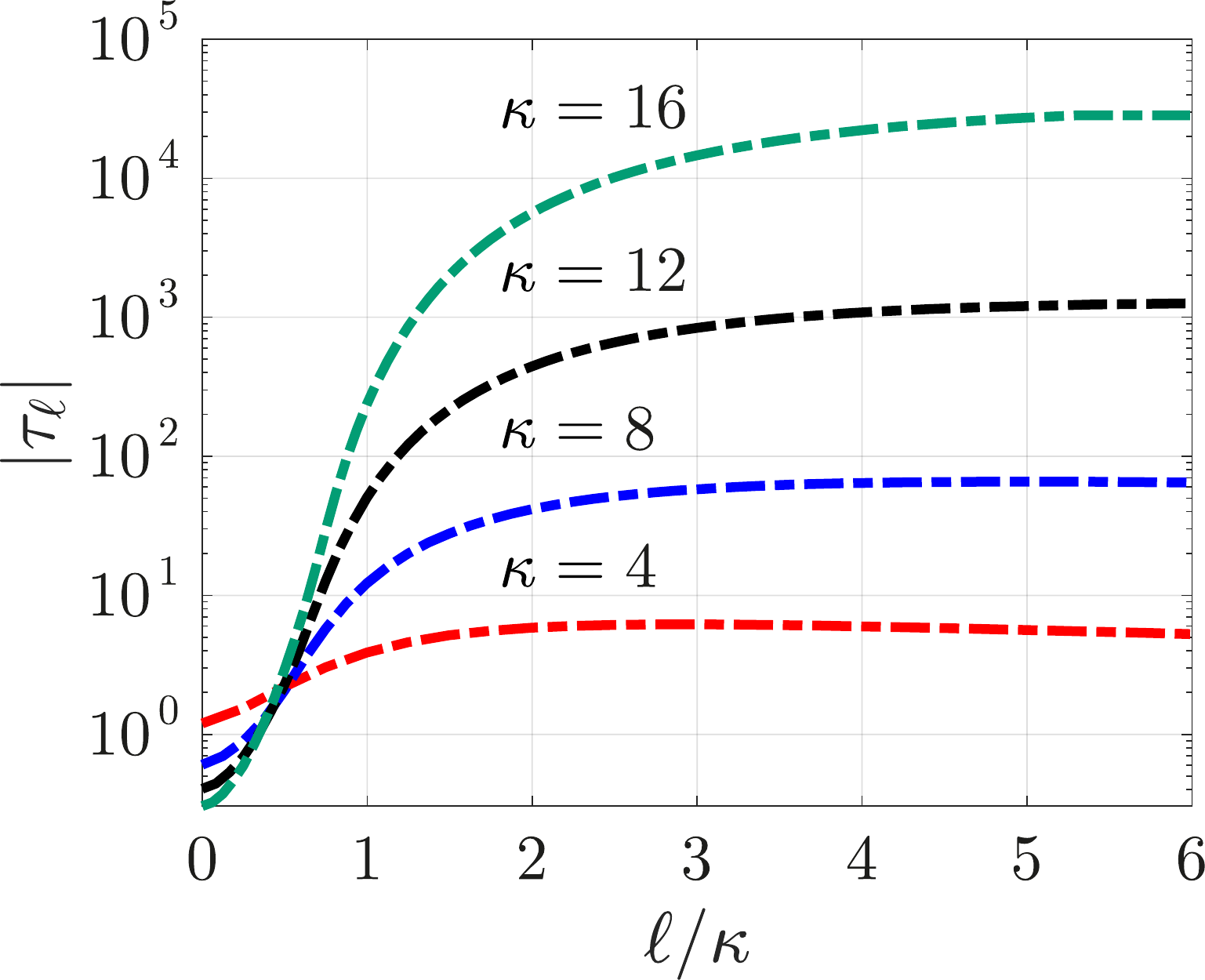}
  \captionof{figure}{Dependence of $|\tau_{\ell}|$ on the mode number $\ell$ for various wavenumber $\kappa$.}
  \label{figure 3.1}
\end{minipage}
\hfill
\begin{minipage}{.44\textwidth}
  \centering
  \includegraphics[width=0.90\textwidth]{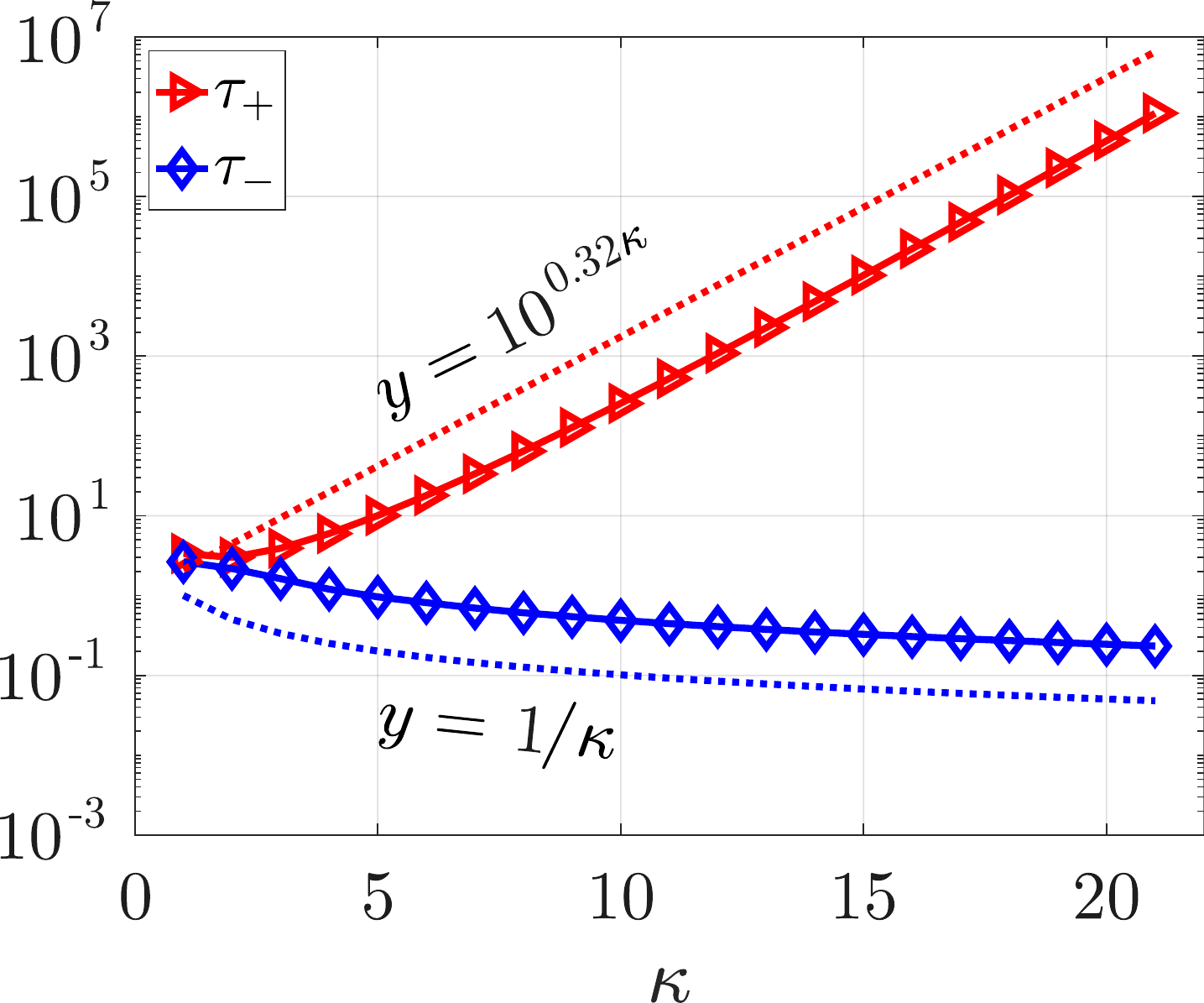}
  \captionof{figure}{Dependence of $\tau_{\pm}$ defined in (\ref{uniform bounds tau}) on the wavenumber $\kappa$.}
  \label{figure 3.2}
\end{minipage}
\end{figure}

The behavior of $|\tau_{\ell}|$ is shown  
in Figure~\ref{figure 3.1} for different wavenumbers $\kappa$. 
This plot aligns with the results in (\ref{uniform bounds tau}), displaying a flat asymptotic behavior for larger values of $\ell$.
Moreover, the values $\tau_{\pm}$ depend on the wavenumber $\kappa$, as shown in Figure~\ref{figure 3.2}.

\subsection{Herglotz integral representation}

Introducing the \emph{Herglotz transform} $\Tcont{Y}{EW}$, we can
represent any Helmholtz solution in $\mathcal{B}$ as a linear combination of
EPWs, each weighted by an element of $\mathcal{A}$.
This integral operator is well-defined on $\mathcal{A}$ thanks to the following result.

\begin{lemma}
$\{\textup{EW}_{\mathbf{y}}\}_{\mathbf{y} \in Y}$ is a Bessel family for $\mathcal{B}$, where the optimal Bessel bound is $B=\tau_{+}^2$.
\end{lemma}
\begin{proof}
This can be seen directly from Lemma \ref{Lemma 3.5} and the Jacobi--Anger identity (\ref{tau jacobi-anger}):
\begin{equation*}
\int_Y|(u,\textup{EW}_{\mathbf{y}})_{\mathcal{B}}|^2\textup{d}\nu(\mathbf{y})=\sum_{(\ell,m) \in \mathcal{I}}|\tau_{\ell}|^2|(u,b_{\ell}^m)_{\mathcal{B}}|^2\leq \tau_{+}^2\|u\|^2_{\mathcal{B}} \qquad \forall u \in \mathcal{B}. \qedhere
\end{equation*}
\end{proof}

The synthesis operator $\Tcont{Y}{EW}$ associated to the EPW
family is defined, for any $v \in L_{\nu}^2(Y)$, by
\begin{equation}
\boxed{\left(\Tcont{Y}{EW} v\right)(\mathbf{x}):=\int_{Y}v(\mathbf{y})\textup{EW}_{\mathbf{y}}(\mathbf{x})\textup{d}\nu(\mathbf{y}) \qquad \forall \mathbf{x} \in B_1.}
\label{Herglotz transform}
\end{equation}
The following theorem, extension to 3D of~\cite[Th.~6.7]{parolin-huybrechs-moiola}, 
shows that for any Helmholtz solution $u \in \mathcal{B}$ there exists a unique
corresponding Herglotz density $v \in \mathcal{A}$ such that
$u=\Tcont{Y}{EW}v$ and justifies the use of the term `transform'
associated to $\Tcont{Y}{EW}$.
\begin{theorem} \label{Theorem 3.9}
The operator $\Tcont{Y}{EW}$ is bounded and invertible from $\mathcal{A}$ to $\mathcal{B}$:
\begin{equation} \label{bound Herglotz operator}
\Tcont{Y}{EW}\!:\mathcal{A} \rightarrow \mathcal{B}, \;\; v \mapsto \!\!\!\!\!\sum_{(\ell,m) \in \mathcal{I}}\!\!\!\tau_{\ell}(v,a_{\ell}^m)_{\mathcal{A}}b_{\ell}^m, \quad \text{and} \quad \tau_{-}\|v\|_{\mathcal{A}}\leq\left\|\Tcont{Y}{EW} v\right\|_{\mathcal{B}}\leq\tau_+\|v\|_{\mathcal{A}}\;\; \forall v \in \mathcal{A}.
\end{equation}
In particular, $\Tcont{Y}{EW}$ is diagonal on the space bases, namely $\Tcont{Y}{EW} a_{\ell}^m=\tau_{\ell}\,b_{\ell}^m$ for all $(\ell,m) \in \mathcal{I}$.
\end{theorem}
\begin{proof}
Thanks to the Jacobi--Anger identity (\ref{tau jacobi-anger}), for any $v \in \mathcal{A}$ and $\mathbf{x} \in B_1$ we have
\begin{equation*}
(\Tcont{Y}{EW} v)(\mathbf{x})=\sum_{(\ell,m) \in \mathcal{I}}\tau_{\ell}\int_{Y}b_{\ell}^m(\mathbf{x})\overline{a_{\ell}^m(\mathbf{y})}v(\mathbf{y})w(\mathbf{y})\textup{d}\mathbf{y}=\sum_{(\ell,m) \in \mathcal{I}}\tau_{\ell}(v,a_{\ell}^m)_{\mathcal{A}}\,b_{\ell}^m(\mathbf{x}),
\end{equation*}
and so (\ref{bound Herglotz operator}) follows from \eqref{uniform bounds tau}. Moreover, it can be readily checked that the inverse is
\begin{equation*}
\left(\Tcont{Y}{EW}\right)^{-1}u=\sum_{(\ell,m) \in \mathcal{I}}\tau_{\ell}^{-1}(u,b_{\ell}^m)_{\mathcal{B}}\,a_{\ell}^m \qquad \forall u \in \mathcal{B}. \qedhere
\end{equation*}
\end{proof}

As a direct consequence of the isomorphism property of $\boldsymbol{\mathrm{T}}_{Y}$,
EPWs allow stable continuous approximation  
of Helmholtz solutions.
In fact, a stronger property holds: all Helmholtz solutions in $\mathcal B$ are continuous superpositions of EPW, and the bounding constant $C_{\textup{cs}}$ in \eqref{condition stable continuous representation} is independent of the tolerance~$\eta$.
Although stated at the continuous level, such a property lays the foundation
for stable discrete expansions, as we will see in more detail in
section~\ref{sec:numerical recipe}.

\begin{corollary}
  The Bessel family $\{\textup{EW}_{\mathbf{y}}\}_{\mathbf{y} \in Y}$ is a stable
  continuous approximation for $\mathcal{B}$. 
\end{corollary}

Adopting the point of view of \emph{Frame Theory} (for a reference, see
\cite{christensen}), another consequence of Theorem~\ref{Theorem 3.9} is that
EPWs form a continuous frame for the Helmholtz solution space $\mathcal{B}$.
For more details on these aspects, see~\cite[sect.~5.2]{galante} (see also
\cite[sect.~6.2]{parolin-huybrechs-moiola}). 
In particular for the proof of the next theorem, see \cite[Th.~5.13]{galante}.

\begin{theorem} \label{Theorem 3.10}
  The Bessel family $\{\textup{EW}_{\mathbf{y}}\}_{\mathbf{y} \in Y}$ is a
  continuous frame for $\mathcal{B}$, namely:
  for any $u\in\mathcal{B}$, $\mathbf{y} \mapsto (u,\textup{EW}_{\mathbf{y}})_{\mathcal{B}}$ is
  measurable in $Y$, and 
\begin{equation}
  \tau_{-}^{2} \|u\|_{\mathcal{B}}^2
  \leq \int_{Y} |(u,\textup{EW}_{\mathbf{y}})_{\mathcal{B}}|^2 \textup{d}\nu(\mathbf{y}) \leq
  \tau_{+}^{2} \|u\|_{\mathcal{B}}^2.
\end{equation}
\end{theorem}

\subsection{Propagative plane waves are not a stable continuous approximation}
\label{propagative plane wave continuous instability}

We show that PPWs are not a stable continuous approximation for the Helmholtz
solution space.

\begin{lemma}
$\{\textup{PW}_{\boldsymbol{\theta}}\}_{\boldsymbol{\theta} \in \Theta}$ is a Bessel family for $\mathcal{B}$.
\end{lemma}
\begin{proof}
From \eqref{eq:PPW_definition} and \eqref{eq:B_norm}, it can be easily seen that $\|\textup{PW}_{\boldsymbol{\theta}}\|^2_{\mathcal{B}} =2|B_{1}|=8\pi/3$. 
Hence, we have
\begin{equation*}
\int_{\Theta}|(u,\textup{PW}_{\boldsymbol{\theta}})_{\mathcal{B}}|^2\textup{d}\sigma(\boldsymbol{\theta}) \leq \int_{\Theta}\|u\|_{\mathcal{B}}^2\|\textup{PW}_{\boldsymbol{\theta}}\|_{\mathcal{B}}^2\,\textup{d}\sigma(\boldsymbol{\theta})
=|\mathbb{S}^{2}| \frac{8\pi}{3} \|u\|_{\mathcal{B}}^2
=\frac{32\pi^2}3 \|u\|^2_{\mathcal{B}}\qquad \forall u \in \mathcal{B}. \qedhere
\end{equation*}
\end{proof}

We can therefore define the synthesis operator associated with PPWs: 
for any $v \in L^2(\mathbb{S}^2)$,
\begin{equation}
\left(\Tcont{\Theta}{PW} v\right)(\mathbf{x}):=\int_{\Theta}v(\boldsymbol{\theta})\textup{PW}_{\boldsymbol{\theta}}(\mathbf{x})\textup{d}\sigma(\boldsymbol{\theta})\qquad\forall \mathbf{x} \in B_1.
\label{herglotz}
\end{equation}
Such continuous superpositions of PPWs
$\Tcont{\Theta}{PW} v \in C^{\infty}(\mathbb{R}^3)$
for some $v\in L^{2}(\mathbb{S}^2)$
are Helmholtz solutions and known as \emph{Herglotz
functions} in the literature~\cite[eq.~(3.43)]{colton-kress}.
However, not all $u \in \mathcal{B}$ can be expressed in the form
(\ref{herglotz}) for some $v \in L^2(\mathbb{S}^2)$, for instance PPWs
themselves.

The next result shows that the two requirements in (\ref{condition stable continuous representation}), i.e.\ accurate approximation and bounded density
norm, are mutually exclusive.
As soon as the spherical wave $b_{\ell}^m$ is accurately represented by PPWs,
the density norm must increase super-exponentially fast in $\ell$ in virtue of
Lemma \ref{Lemma 2.6}. Hence, the Bessel family
$\{\textup{PW}_{\boldsymbol{\theta}}\}_{\boldsymbol{\theta} \in \Theta}$ is not
a stable continuous approximation for~$\mathcal{B}$.

\begin{lemma}\label{lem:ppw-lack-continuous-stability}
Let $(\ell,m) \in \mathcal{I}$ and $0 < \eta \leq 1$ be given. 
For a given $v \in L^2(\mathbb{S}^2)$,
\begin{equation}
\text{if}\qquad
\left\|b_{\ell}^m-\Tcont{\Theta}{PW}v\right\|_{\mathcal{B}} \leq \eta \|b_{\ell}^m\|_{\mathcal{B}} 
\qquad \text{then}
\qquad \|v\|_{L^2(\mathbb{S}^2)} \geq \left(1-\eta \right) \frac{\beta_{\ell}}{4\pi}\|b_{\ell}^m\|_{\mathcal{B}}.
\label{plane wave continuous instability}
\end{equation}
\end{lemma}
\begin{proof}
Let $v \in L^2(\mathbb{S}^2)$.
Using the classical  
Jacobi--Anger identity (\ref{jacobi-anger}), we obtain
\begin{equation}
\left(\Tcont{\Theta}{PW} v \right)(\mathbf{x})=4 \pi\int_{\Theta}v(\boldsymbol{\theta})\sum_{q=0}^{\infty}i^q \sum_{n=-q}^q\overline{Y_q^n\left(\boldsymbol{\theta}\right)}\,\tilde{b}_q^n(\mathbf{x})\textup{d}\sigma(\boldsymbol{\theta})=\sum_{(q,n) \in \mathcal{I}} c_q^n \tilde{b}_q^n(\mathbf{x}),
\label{alpaka 1}
\end{equation}
where, since $\|Y_{q}^n\|_{L^2(\mathbb{S}^2)}=1$ for any $(q,n) \in \mathcal{I}$, the coefficients
\begin{equation}
c_q^n:=4\pi i^q\left(v,Y_{q}^n\right)_{L^2(\mathbb{S}^2)} \qquad \text{satisfy} \qquad |c_q^n|\leq 4\pi\|v\|_{L^2(\mathbb{S}^2)} \qquad  \forall (q,n) \in \mathcal{I}.
\label{alpaka_}
\end{equation}
From (\ref{alpaka 1}), it follows:
\begin{equation*}
\left\|b_{\ell}^m-\Tcont{\Theta}{PW}v\right\|^2_{\mathcal{B}}=\sum_{(q,n) \in \mathcal{I}}\left|\delta_{\ell, q}\delta_{m,n}\!-c_q^n\beta^{-1}_q\right|^2 \leq \eta^2 \quad \Rightarrow \quad \left|\delta_{\ell, q}\delta_{m,n}\!-c_q^n\beta^{-1}_{\ell}\right| \leq \eta \quad \forall (q,n) \in \mathcal{I}.
\end{equation*}
Due to (\ref{alpaka_}), for $(q,n)=(\ell,m)$, this reads
\begin{equation*}
\eta \geq \left|1-c_{\ell}^m\beta^{-1}_{\ell}\right| \geq 1-|c_{\ell}^m|\beta^{-1}_{\ell} \geq 1-4\pi\beta^{-1}_{\ell}\|v\|_{L^2(\mathbb{S}^2)},
\end{equation*}
which can be written as (\ref{plane wave continuous instability}), recalling that $\|b_{\ell}^m\|_{\mathcal{B}}=1$.
\end{proof}

\begin{theorem}\label{Thm:PPW-SCA}
The Bessel family $\{\textup{PW}_{\boldsymbol{\theta}}\}_{\boldsymbol{\theta} \in \Theta}$ is not a stable continuous approximation for $\mathcal{B}$. 
\end{theorem}

\section{Stable discrete approximation}\label{sec:stable numerical approximation}

After having considered integral, or continuous, approximations,
this section introduces the complementary notion of \emph{stable discrete
approximation}.
A practical numerical scheme based on sampled Dirichlet data for the
approximation of Helmholtz solutions in the ball is analysed.
Already introduced in \cite[sect.~2.2]{galante} (see also
\cite[sect.~3.2]{parolin-huybrechs-moiola}), it relies on regularized SVD and
oversampling following the recommendations of \cite{huybrechs1,huybrechs2}.
This procedure is proved to yield accurate solutions in finite-precision
arithmetic, provided the approximation set has the stable discrete
approximation property and appropriate sampling points have been chosen.
We show that PPWs are inherently unstable also in the discrete setting.
The EPW sets constructed later in section~\ref{sec:numerical recipe} are
empirically shown in section~\ref{sec:numerical results} to satisfy the
discrete stability notion presented here.

\subsection{The concept of stable discrete approximation}

Let us first review the definition of stable discrete approximation proposed in
\cite[Def.\ 2.1]{galante} and \cite[Def.\ 3.1]{parolin-huybrechs-moiola}.
We consider a sequence $\{\Phi_P\}_{P\in\mathbb{N}}$ of \emph{finite
approximation set} $\Phi_P:=\{\phi_p\}_p \subset \mathcal{B}$.
For each $P \in \mathbb{N}$, we define the synthesis operator associated with
$\Phi_{P}$ by
\begin{equation}\label{eq:discrete_synthesis_op}
  \Tdisc{\!P\,}{\,}\!:\mathbb{C}^{|\Phi_P|} \rightarrow \mathcal{B},
  \ \boldsymbol{\mu}=(\mu_p)_p \mapsto \sum\nolimits_p \mu_p \phi_p.
\end{equation}
In sections \ref{subsec:regularized boundary sampling method} and \ref{subsec:error estimates}, we consider general finite
approximation sets $\Phi_P$, while sections \ref{subsec:propagative plane wave discrete instability}
and \ref{sec:numerical recipe} are specifically devoted to PPW and EPW approximation sets, respectively.

\begin{definition}[Stable discrete approximation]\label{def:SDA}
  The sequence of approximation sets $\{\Phi_P\}_{P \in \mathbb{N}}$ is said to be a
  stable discrete approximation for $\mathcal{B}$ if, for any tolerance $\eta
  >0$, there exist a stability exponent $s_{\textup{ds}} \geq 0$, a
  stability constant $C_{\textup{ds}} \geq 0$ such that
  \begin{equation}
    \forall u \in \mathcal{B},\ 
    \exists P \in \mathbb{N},\ \boldsymbol{\mu} \in \mathbb{C}^{|\Phi_P|}: \quad
    \left\|u-\Tdisc{\!P\,}{\,}\boldsymbol{\mu}\right\|_{\mathcal{B}} \leq \eta \|u\|_{\mathcal{B}},
    \quad \text{and} \quad
    \|\boldsymbol{\mu}\|_{\ell^2} \leq C_{\textup{ds}}|\Phi_P|^{s_{\textup{ds}}}\|u\|_{\mathcal{B}}.
  \label{stable discrete approximation}
  \end{equation}
\end{definition}

This definition serves as the discrete counterpart to the concept of stable continuous approximation in (\ref{condition stable continuous representation}). With a sequence of stable discrete approximation sets, we can accurately approximate any Helmholtz solution in the form of a finite expansion $\Tdisc{\!P\,}{\,}\boldsymbol{\mu}$, where the coefficients $\boldsymbol{\mu}$ have a bounded $\ell^2$-norm, except for some algebraic growth.
Due to the Hölder inequality, the $\ell^2$-norm in (\ref{stable discrete approximation}) can be replaced by any discrete $\ell^p$-norm, possibly changing the exponent $s_{\textup{ds}}$.

\subsection{Regularized boundary sampling method}\label{subsec:regularized boundary sampling method}

We now outline a practical approach for computing the expansion coefficients using a sampling-type strategy, following \cite{huybrechs1,huybrechs2} and in line with \cite{huybrechs3}. Let us consider the Helmholtz problem with Dirichlet boundary conditions: find $u \in H^1(B_1)$ such that
\begin{equation*}
\Delta u+\kappa^2 u =0,\,\,\,\,\text{in}\,\,\,B_1, \qquad \text{and} \qquad \gamma u=g,\,\,\,\,\text{on}\,\,\,\partial B_1,
\end{equation*}
where $g \in H^{1/2}(\partial B_1)$ and $\gamma$ is the Dirichlet trace operator; this problem is known to be well-posed if $\kappa^2$ is not an eigenvalue of the Dirichlet Laplacian.
In all our numerical experiments, we aim to reconstruct a solution $u \in \mathcal{B}$ using its boundary trace $\gamma u$. Hence, for simplicity, we assume $u \in \mathcal{B} \cap C^0(\overline{B_1})$, allowing us to consider point evaluations of the Dirichlet trace.

So let $u \in \mathcal{B} \cap C^0(\overline{B_1})$ be our approximation target.
Given a finite approximation set $\Phi_P \subset \mathcal B$, we seek a coefficient
vector $\boldsymbol{\xi} \in \mathbb{C}^{|\Phi_P|}$ such that
$\Tdisc{\!P\,}{\,}\boldsymbol{\xi} \approx u$.
The solution $u$ is supposed to be known at $S \geq |\Phi_P|$ sampling points
$\{\mathbf{x}_s\}_{s=1}^S \subset \partial B_1$.
We assume that, as the number of such sampling points increases, there is
convergence of a cubature rule, namely that
\begin{equation}
\lim_{S \rightarrow \infty}\sum_{s=1}^S w_s v(\mathbf{x}_s)=\int_{\partial B_1}v(\mathbf{x})\textup{d}\mathbf{x} \qquad \forall v \in C^0(\partial B_1),
\label{Riemann sum}
\end{equation}
where $\mathbf{w}_S=(w_s)_s \in \mathbb{R}^S$ is a vector of positive weights associated with the point set $\{\mathbf{x}_s\}_{s=1}^S$.
Introducing non-uniform weights is a slight modification from \cite[sect.~3.2]{parolin-huybrechs-moiola} that provides more generality. Unlike the two-dimensional case \cite[eq.~(3.6)]{parolin-huybrechs-moiola}, there is no obvious way to determine such a set.
In the following numerical experiments, we use \emph{extremal systems of points} and associated weights~\cite{marzo-cerda,reimer,sloan-womersley1,sloan-womersley2}, which satisfy the identity (\ref{Riemann sum}).

Defining the matrix $A=(A_{s,p})_{s,p} \in \mathbb{C}^{S \times |\Phi_P|}$ and the vector $\mathbf{b}=(b_s)_s \in \mathbb{C}^S$ as follows
\begin{equation}
A_{s,p}:=w_s^{1/2}\phi_p(\mathbf{x}_s), \qquad \mathbf{b}_s:=w_s^{1/2}(\gamma u)(\mathbf{x}_s), \qquad 1 \leq p \leq |\Phi_P|,\,\,\, 1 \leq s \leq S,
\label{A matrix definition}
\end{equation}
the sampling method consists in approximately solving the possibly overdetermined linear system
\begin{equation}
A\boldsymbol{\xi}=\mathbf{b}.
\label{linear system}
\end{equation}
The matrix $A$ may often be ill-conditioned~\cite{hiptmair-moiola-perugia1} as
a result of the redundancy of the approximating functions, potentially leading
to inaccurate numerical solutions.
A corollary of ill-conditioning is non-uniqueness of the solution of the linear
system in computer arithmetic.
If all solutions may approximate $u$ with comparable accuracy,
only those with 
small coefficient norm can be computed accurately in
finite precision arithmetic in practice.
To achieve this, we rely on the combination of oversampling and regularization
techniques developed in \cite{huybrechs1,huybrechs2}.
The regularized solution procedure is divided into the following steps:
\begin{itemize}
\item Firstly, the Singular Value Decomposition (SVD) $A=U\Sigma V^*$ of the matrix $A$ is performed.
Let $\sigma_p$ denote the singular values of $A$ for $p=1,\ldots,|\Phi_P|$, assuming they are sorted in descending order. For clarity, we relabel the largest singular value as $\sigma_{\textup{max}}:=\sigma_1$.
\item Then, the regularization involves discarding the relatively small singular values by setting them to zero.
A threshold parameter $\epsilon \in (0,1]$ is selected, and the diagonal matrix $\Sigma$ is replaced by  
$\Sigma_{\epsilon}$ by zeroing all $\sigma_p$ such that $\sigma_p < \epsilon \sigma_{\textup{max}}$.
This results in an approximate factorization of $A$, that is $A_{S,\epsilon}:=U\Sigma_{\epsilon}V^*$.
\item Lastly, an approximate solution for the linear system in (\ref{linear system}) is obtained by
\begin{equation}
\boldsymbol{\xi}_{S,\epsilon}:=A^{\dagger}_{S,\epsilon}\mathbf{b}=V\Sigma_{\epsilon}^{\dagger}U^*\mathbf{b}.
\label{xi Se solution}
\end{equation}
Here $\Sigma_{\epsilon}^{\dagger} \in \mathbb{R}^{|\Phi_P| \times S}$ denotes the pseudo-inverse of the matrix $\Sigma_{\epsilon}$, i.e.\ the diagonal matrix defined by $(\Sigma_{\epsilon}^{\dagger})_{j,j}=(\Sigma_{j,j})^{-1}$ if $\Sigma_{j,j} \geq \epsilon \sigma_{\textup{max}}$ and $(\Sigma_{\epsilon}^{\dagger})_{j,j}=0$ otherwise.
To robustly compute $\boldsymbol{\xi}_{S,\epsilon}$, the products at the right-hand side of (\ref{xi Se solution}) should be evaluated from right to left to avoid mixing small and large values on the diagonal of $\Sigma_{\epsilon}^{\dagger}$.
\end{itemize}

We chose to use a regularized SVD for its stability and robustness, despite its relatively high computational cost. Other alternatives are possible. For instance, \cite{Barucq2021} proposes an element-wise regularization strategy based on SVD or QR decompositions, effectively addressing ill-conditioning and reducing the size of the resulting linear systems. Further investigation into such scalable and effective strategies is deferred to future research.

\subsection{Error estimates} \label{subsec:error estimates}

Using regularization and oversampling ($S \geq |\Phi_P|$), accurate approximations can be achieved if the set sequence is a stable discrete approximation according to Definition \ref{stable discrete approximation}, and (\ref{Riemann sum}) holds for the chosen sampling points and weights. This result is the main conclusion of \cite[Th.~5.3]{huybrechs1} and \cite[Th.~1.3 and~3.7]{huybrechs2}, forming the basis of the investigation into stable discrete approximation sets for the solutions of the Helmholtz equation.
We have the following results from \cite[Prop.~2.4 and Cor.~2.5]{galante} (to which we refer for the proofs), which build on \cite[Prop.~3.2 and Cor.~3.3]{parolin-huybrechs-moiola} respectively.

\begin{proposition}
  Let 
  $u \in \mathcal{B} \cap C^0(\overline{B_1})$ and $P \in \mathbb{N}$. 
  Given some approximation set $\Phi_P=\{\phi_p\}_p$ such that
  $\phi_p \in \mathcal{B} \cap C^0(\overline{B_1})$ for any $p$, sampling point sets $\{\mathbf{x}_s\}_{s=1}^S \subset \partial B_1$ along with 
  positive weights $\mathbf{w}_S \in \mathbb{R}^S$ satisfying
  \textup{(\ref{Riemann sum})}, and some regularization parameter $\epsilon \in(0,1]$, 
  let $\boldsymbol{\xi}_{S,\epsilon} \in \mathbb{C}^{|\Phi_P|}$ be 
  the approximate solution of the linear system \textup{(\ref{linear system})}, 
  as defined in \textup{(\ref{xi Se solution})}.
  Then $\forall \boldsymbol{\mu}\in\mathbb{C}^{|\Phi_P|}$, $\exists S_0 \in \mathbb{N}$ such that $\forall S \geq S_0$
\begin{equation*}
\left\| \gamma(u-\Tdisc{\!P\,}{\,}\boldsymbol{\xi}_{S,\epsilon})\right\|_{L^2(\partial B_1)} \leq 3\left\|\gamma(u-\Tdisc{\!P\,}{\,}\boldsymbol{\mu})\right\|_{L^2(\partial B_1)} + \sqrt{2}\epsilon \sigma_{\textup{max}}\|\mathbf{w}_S\|^{1/2}_{\ell^\infty}\|\boldsymbol{\mu}\|_{\ell^2}.
\end{equation*}
Assume moreover that $\kappa^2$ is not an eigenvalue of the Dirichlet Laplacian in $B_1$. Then there exists a constant $C_{\textup{err}} >0$ independent of $u$ and $\Phi_P$ such that $\forall \boldsymbol{\mu} \in \mathbb{C}^{|\Phi_P|}$, $\exists S_0 \in \mathbb{N}$ such that $\forall S \geq S_0$
\begin{equation*}
\left\|u-\Tdisc{\!P\,}{\,}\boldsymbol{\xi}_{S,\epsilon}\right\|_{L^2( B_1)} \leq C_{\textup{err}} \left(\left\|u-\Tdisc{\!P\,}{\,}\boldsymbol{\mu}\right\|_{\mathcal{B}} + \epsilon \sigma_{\textup{max}}\|\mathbf{w}_S\|^{1/2}_{\ell^\infty}\|\boldsymbol{\mu}\|_{\ell^2} \right).
\end{equation*}
\end{proposition}

\begin{corollary} \label{Corollary 4.3}
  Let $\delta >0$.
  Assume that the sequence of approximation sets $\{\Phi_P\}_{P \in \mathbb{N}}$ is stable in the sense of 
  Definition~\textup{\ref{stable discrete approximation}} and that
  the sets $\{\mathbf{x}_s\}_{s=1}^S \subset \partial B_1$ of sampling points and the positive weight
  vectors $\mathbf{w}_S \in \mathbb{R}^S$, defined for any $S \in \mathbb{N}$, satisfy the cubature-convergence condition \textup{(\ref{Riemann sum})}.
  Assume also that $\kappa^2$ is not a Dirichlet eigenvalue in $B_1$. Then,
  $\forall u \in \mathcal{B} \cap C^0(\overline{B_1})$, $\exists P \in \mathbb{N}$,
  $S_0 \in \mathbb{N}$ and $\epsilon_0 \in (0,1]$ such that $\forall S \geq S_0$ and
  $\epsilon \in (0,\epsilon_0]$
\begin{equation*}
\|u-\Tdisc{\!P\,}{\,}\boldsymbol{\xi}_{S,\epsilon}\|_{L^2(B_1)} \leq \delta \|u\|_{\mathcal{B}},
\end{equation*}
where $\boldsymbol{\xi}_{S,\epsilon} \in \mathbb{C}^{|\Phi_P|}$ is defined in \textup{(\ref{xi Se solution})}. The regularization parameter $\epsilon$ can be taken as large as
\begin{equation*}
\epsilon_0=\delta\left(2C_{\textup{err}}C_{\textup{ds}}\sigma_{\textup{max}}|\Phi_P|^{s_{\textup{ds}}}\|\mathbf{w}_S\|^{1/2}_{\ell^\infty}\right)^{-1}.
\end{equation*}
\end{corollary}

The previous error bounds on $u-\Tdisc{\!P\,}{\,}\boldsymbol{\xi}_{S,\epsilon}$
apply to the solution obtained by the sampling method using finite precision
arithmetic.
In particular, Corollary \ref{Corollary 4.3} shows that the vector
$\boldsymbol{\xi}_{S,\epsilon}$, which is stably computable  
in floating-point arithmetic using the regularized SVD (\ref{xi Se solution}),
yields an accurate approximation $\Tdisc{\!P\,}{\,}\boldsymbol{\xi}_{S,\epsilon}$ of
$u$. 
In contrast, rigorous best-approximation error bounds from the classical theory
of approximation by PPWs, e.g.~\cite{hiptmair-moiola-perugia4}, 
are often not achievable numerically due to the need for large coefficients and
cancellation which leads to numerical instability.

Lastly, to measure the approximation error, we introduce the following relative residual
\begin{equation}
\mathcal{E}=\mathcal{E}(u,\Phi_P,S,\epsilon):=\frac{\|A\boldsymbol{\xi}_{S,\epsilon}-\mathbf{b}\|_{\ell^2}}{\|\mathbf{b}\|_{\ell^2}},
\label{relative residual}
\end{equation}
where $\boldsymbol{\xi}_{S,\epsilon}$ is the solution (\ref{xi Se solution}) of the regularized system.
Following the argument of the proof of \cite[Prop.~2.4]{galante}, it can be shown that for sufficiently large $S$, the residual $\mathcal{E}$ in
(\ref{relative residual}) satisfies, for a constant $\widetilde{C}$ independent
of $u$, $\Phi_P$ and $S$,
\begin{equation*}
\|u-\Tdisc{\!P\,}{\,}\boldsymbol{\xi}_{S,\epsilon}\|_{L^2( B_1)} \leq \widetilde{C}\|u\|_{\mathcal{B}}\,\mathcal{E}.
\end{equation*}

\subsection{Propagative plane wave discrete instability} \label{subsec:propagative plane wave discrete instability}

We consider any PPW approximation set of $P\in\mathbb{N}$ elements
\begin{equation}
  \Phi_P:=\bigl\{P^{-1/2}\, \textup{PW}_{\mathbf{\boldsymbol{\theta}}_p}\bigr\}_{p=1}^{P}, 
  \qquad\text{with}\qquad\{\boldsymbol{\theta}_p\}_{p=1}^{P} \subset \Theta
\label{plane waves approximation set}
\end{equation}
and denote by $\Tdisc{P}{PW}$ the corresponding synthesis operator \eqref{eq:discrete_synthesis_op}.
In practice, isotropic approximations are attained by using nearly-uniform directions, and in our numerical experiments we use the extremal systems \cite{marzo-cerda,reimer,sloan-womersley1,sloan-womersley2} due to their well-distributed nature. However, the next results are valid for any set.

Analogously to section~\ref{propagative plane wave continuous instability}, let us consider the problem of approximating a spherical wave $b_{\ell}^m$ for some $(\ell,m) \in \mathcal{I}$ using the PPW approximation sets (\ref{plane waves approximation set}). Likewise, the two conditions in (\ref{stable discrete approximation}), namely low error and small coefficients, are incompatible.

\begin{lemma} \label{Lemma 4.4}
  Let $(\ell,m) \in \mathcal{I}$, $0 < \eta \leq 1$ and $P \in \mathbb{N}$ be given.
  For any \textup{PPW} approximation set $\Phi_P$ as in~\eqref{plane waves approximation set}, and every coefficient vector 
  $\boldsymbol{\mu} \in \mathbb{C}^{P}$,
  \begin{equation}
  \text{if}\qquad
  \left\|b_{\ell}^m-\Tdisc{P}{PW}\boldsymbol{\mu}\right\|_{\mathcal{B}} \leq \eta \|b_{\ell}^m\|_{\mathcal{B}} 
  \qquad \text{then}
  \qquad \|\boldsymbol{\mu}\|_{\ell^2} \geq \left(1-\eta \right)\frac{\beta_{\ell}}{2\sqrt{\pi(2\ell+1)}}\|b_{\ell}^m\|_{\mathcal{B}}.
  \label{plane wave instability}
  \end{equation}
\end{lemma}
\begin{proof}
Let $\boldsymbol{\mu} \in \mathbb{C}^{P}$. Using the standard Jacobi--Anger identity (\ref{jacobi-anger}), we obtain
\begin{equation}
\left(\Tdisc{P}{PW} \boldsymbol{\mu} \right)(\mathbf{x})=\frac{4 \pi}{\sqrt{P}}\sum_{p=1}^{P}\mu_p\sum_{q=0}^{\infty}i^q \sum_{n=-q}^q\overline{Y_q^n(\mathbf{d}_p)}\,\tilde{b}_q^n(\mathbf{x})=\sum_{(q,n) \in \mathcal{I}} c_q^n \tilde{b}_q^n(\mathbf{x}),
\label{alpaka 2}
\end{equation}
where, thanks to \cite[eq.~(2.4.106)]{nedelec}, the coefficients
\begin{equation}
c_q^n:=\frac{4\pi i^q}{\sqrt{P}}\sum_{p=1}^{P}\mu_p\overline{Y_q^n(\mathbf{d}_p)} \qquad \text{satisfy} \qquad |c_q^n|\leq 2\sqrt{\pi(2q+1)}\|\boldsymbol{\mu}\|_{\ell^2} \qquad  \forall (q,n) \in \mathcal{I}.
\label{alpaka}
\end{equation}
From (\ref{alpaka 2}), it follows:
\begin{equation*}
\left\|b_{\ell}^m-\Tdisc{P}{PW}\boldsymbol{\mu}\right\|^2_{\mathcal{B}}=\sum_{(q,n) \in \mathcal{I}}\left|\delta_{\ell, q}\delta_{m,n}\!-c_q^n\beta^{-1}_q\right|^2 \leq \eta^2 \quad \Rightarrow \quad \left|\delta_{\ell, q}\delta_{m,n}\!-c_q^n\beta^{-1}_{\ell}\right| \leq \eta \quad \forall (q,n) \in \mathcal{I}.
\end{equation*}
Due to (\ref{alpaka}), for $(q,n)=(\ell,m)$, this reads
\begin{equation*}
\eta \geq \left|1-c_{\ell}^m\beta^{-1}_{\ell}\right| \geq 1-|c_{\ell}^m|\beta^{-1}_{\ell} \geq 1-2\beta^{-1}_{\ell}\sqrt{\pi(2\ell+1)}\|\boldsymbol{\mu}\|_{\ell^2},
\end{equation*}
which can be written as (\ref{plane wave instability}), recalling that $\|b_{\ell}^m\|_{\mathcal{B}}=1$.
\end{proof}

Bound \eqref{plane wave instability} states that in order to accurately approximate spherical waves $b_{\ell}^m$ using PPW expansions $\Tdisc{P}{PW}\boldsymbol{\mu}$ with a specified accuracy $\eta > 0$, the coefficient norms must increase super-exponentially fast in $\ell$ (recall that $\beta_\ell\sim(\frac{2\ell}{e\kappa})^\ell$ by 
Lemma \ref{Lemma 2.6}). In this context, it is not possible to achieve both accuracy and stability. 
Similarly to \cite[sect.~4.3]{parolin-huybrechs-moiola}, we condense this result in the following theorem.

\begin{theorem} \label{Theorem 4.5}
There is no  
sequence of approximation sets made of \textup{PPWs} that is a stable discrete approximation for the space of Helmholtz solutions in the ball.
\end{theorem}

\section{Numerical recipe for EPW selection}\label{sec:numerical recipe}

We describe a method for the construction of EPW sets in practice.
The core idea is to link the Helmholtz approximation problem to that of the
corresponding Herglotz density.
Following the approach in \cite[sect.~7]{parolin-huybrechs-moiola}, we adapt
the sampling technique from \cite{Cohen_Migliorati,Hampton} (referred to as
\emph{coherence-optimal sampling}) to our setting, generating sampling nodes in
$Y$ to reconstruct the Herglotz density.
Such a technique can also be interpreted as discretizing the integral
representation (\ref{Herglotz transform}), by constructing a cubature rule
valid for finite-dimensional subspaces (see \cite{Migliorati_Nobile}).
Section~\ref{sec:numerical results} shows numerically the effectiveness of the method,
suggesting  
that our construction satisfies the stable discrete
approximation property presented in the previous section.

\subsection{Reproducing kernel property}
A significant consequence of the continuous frame result from Theorem \ref{Theorem 3.10} is highlighted in the next proposition, sourced from \cite[Prop.\ 6.12]{parolin-huybrechs-moiola}; the proof can be found there. For a general reference on Reproducing Kernel Hilbert Spaces (RKHS), consult \cite{reproducing_kernels}.
\begin{proposition} \label{Proposition 5.1}
  The space $\mathcal{A}$ has the reproducing kernel property. The reproducing kernel is
  \begin{equation*}
  K(\mathbf{z},\mathbf{y})=K_{\mathbf{y}}(\mathbf{z})=\left(K_{\mathbf{y}},K_{\mathbf{z}}\right)_{\mathcal{A}}=\sum_{(\ell,m) \in \mathcal{I}}\overline{a_{\ell}^m(\mathbf{y})}a_{\ell}^m(\mathbf{z}) \qquad \forall \mathbf{y},\mathbf{z} \in Y,
  \end{equation*}
  with pointwise convergence of the series and where $K_{\mathbf{y}} \in
  \mathcal{A}$ is the (unique) Riesz representation of the evaluation
  functional at $\mathbf{y} \in Y$, satisfying
  $v(\mathbf{y})=\left(v,K_{\mathbf{y}} \right)_{\mathcal{A}}$ for any $v \in \mathcal{A}$.
\end{proposition}

The reproducing kernel property ensures that the linear evaluation functional
at any point in $Y$ is a continuous operator on $\mathcal{A}$
\cite[Def.\ 1.2]{reproducing_kernels}.
The interest of this property in our setting is clear in the following result,
which is borrowed from \cite[Cor.\ 5.15]{galante} and \cite[Cor.\
6.13]{parolin-huybrechs-moiola}, and stems directly from Proposition
\ref{Proposition 5.1}, Theorem \ref{Theorem 3.9}, and the Jacobi–Anger identity
(\ref{tau jacobi-anger}).

\begin{corollary}
The \textup{EPWs} are the images under the Herglotz transform $\Tcont{Y}{EW}$ of the Riesz representation of the evaluation functionals:
\begin{equation*}
\textup{EW}_{\mathbf{y}}=\Tcont{Y}{EW}K_{\mathbf{y}} \qquad \forall \mathbf{y} \in Y.
\end{equation*}
\end{corollary}
Hence, approximating a Helmholtz solution $u \in \mathcal{B}$ using EPWs is equivalent to approximating its Herglotz density $v=\left(\Tcont{Y}{EW}\right)^{-1}u \in \mathcal{A}$ by an expansion of evaluation functionals:
\begin{equation*}
v \approx \sum_{p=1}^P\mu_p K_{{\mathbf{y}}_p} \qquad   \rightleftarrowss{\Tcont{Y}{EW}}{\left(\Tcont{Y}{EW}\right)^{-1}} \qquad   u \approx \sum_{p=1}^P\mu_p \textup{EW}_{{\mathbf{y}}_p}
\end{equation*}
for some coefficient vector $\boldsymbol{\mu}=\{\mu_{p}\}_{p=1}^P$.
Section \ref{sec:numerical results} provides numerical evidence that 
the procedure outlined in sections \ref{s:ProbDens}--\ref{s:ITS} allows 
to build such approximations (up to some normalization of the families $\{K_{{\mathbf{y}}_p}\}_p$ and $\{\textup{EW}_{{\mathbf{y}}_p}\}_p$).

\subsection{Probability densities}\label{s:ProbDens}

Given a target solution $u \in \mathcal{B}$ and its corresponding Herglotz density $v:=\left(\Tcont{Y}{EW}\right)^{-1}u \in \mathcal{A}$, the strategy for constructing finite-dimensional approximation sets involves
the hierarchy of finite-dimensional subspaces formed by truncating the Hilbert bases $\{a_{\ell}^m\}_{(\ell,m) \in \mathcal{I}}$ and $\{b_{\ell}^m\}_{(\ell,m) \in \mathcal{I}}$.

\begin{definition}[Truncated spaces]\label{def:ALBL}
For any $L \geq 0$, we define, respectively, the truncated Herglotz density space and the truncated Helmholtz solution space as
\begin{equation*}
\mathcal{A}_{L}:=\textup{span}\{a_\ell^m\}_{(\ell,m) \in \mathcal{I}\,:\,\ell\, \leq\, L} \subsetneq \mathcal{A}, \qquad \mathcal{B}_{L}:=\textup{span}\{b_\ell^m\}_{(\ell,m) \in \mathcal{I}\,:\,\ell\, \leq\, L} \subsetneq \mathcal{B}.
\end{equation*}
Moreover, we denote their dimension with
$N=N(L):=\dim \mathcal{A}_L=\dim \mathcal{B}_L=(L+1)^2 \in \mathbb{N}$.
\end{definition}

Let us fix a truncation parameter $L\geq 0$.
Our goal is to approximate with EPWs the projection $u_L \in \mathcal{B}_L$ (or equivalently $v_L=\left(\Tcont{Y}{EW}\right)^{-1}u_L \in \mathcal{A}_L$).
The key idea involves approximating $\mathcal{A}_L$ elements by constructing a set of $P$ sampling nodes $\{\mathbf{y}_p\}_{p=1}^P \subset Y$, following the distribution in \cite[sect.~2.1]{Hampton}, \cite[sect.~2.2]{Cohen_Migliorati}, and \cite[sect.~2]{Migliorati_Nobile}.
The probability density $\rho_N$ \cite[eq.~(2.6)]{Cohen_Migliorati} is defined (up to normalization) as the reciprocal of the \emph{$N\!$-term Christoffel function} $\mu_N$, that is
\begin{equation}
\rho_N(\mathbf y):=\frac{w(\zeta)}{N\mu_N(\mathbf y)}, 
\!\!\qquad\!\! \text{where} \!\!\qquad \!\!\mu^{-1}_N(\mathbf{y}):=\sum_{\ell=0}^L\sum_{m=-\ell}^{\ell}\left|a_{\ell}^m(\mathbf{y}) \right|^2=\sum_{\ell=0}^L\alpha_{\ell}^2\left|\mathbf{P}_{\ell}(\zeta) \right|^2\quad \forall \mathbf{y} \in Y.
\label{rho density}
\end{equation}
Due to the Wigner D-matrix unitarity condition \cite[sect.~4.1, eq.~(6)]{quantumtheory}, $\mu_N$ is independent of $\boldsymbol{\theta}$ and $\psi$.
Hence, $\rho_N$ as a function of $\mathbf y=(\boldsymbol \theta,\psi,\zeta)\in Y$ only depends on $\zeta$, 
and the sampling problem is one-dimensional, with $\zeta$ as the key parameter.
The top row of Figure~\ref{figure 5.1} illustrates the probability density functions
\begin{equation}
\widehat{\rho}_{N}(\zeta):=\int_{\Theta}\int_0^{2\pi}\rho_{N}(\boldsymbol{\theta},\psi,\zeta)\,\textup{d} \psi\textup{d}\sigma(\boldsymbol{\theta}) \qquad \forall\zeta \in [0,+\infty),
\label{rho zeta density}
\end{equation}
with respect to the ratio $\zeta/\kappa$.
The main mode of the densities $\widehat{\rho}_N$ is centered at $\zeta=0$, representing pure PPWs.
As $\kappa$ grows, the peak at $\zeta=0$ gets higher, reflecting the increasing number of
propagative modes, and the numerical support of the density gets larger (note the abscissas scaling).
Eventually, the probability approaches zero exponentially as $\zeta \rightarrow \infty$.
For $L \leq \kappa$, the densities form unimodal distributions, while they exhibit multimodal behavior for $L \gg \kappa$, introducing an extra mode for relatively large values of $\zeta$ (see the wide peak around $\zeta=5\kappa$ in the black curve).
The associated cumulative distribution functions (bottom row of Figure~\ref{figure 5.1}) are defined as
\begin{equation}
\Upsilon_{N}(\zeta):=\int_{0}^{\zeta}\widehat{\rho}_{N}(\eta)\,\textup{d}\eta \qquad \forall\zeta \in [0,+\infty).
\label{cumulative distribution}
\end{equation}

\begin{figure}
\centering
\includegraphics[width=0.89\linewidth]{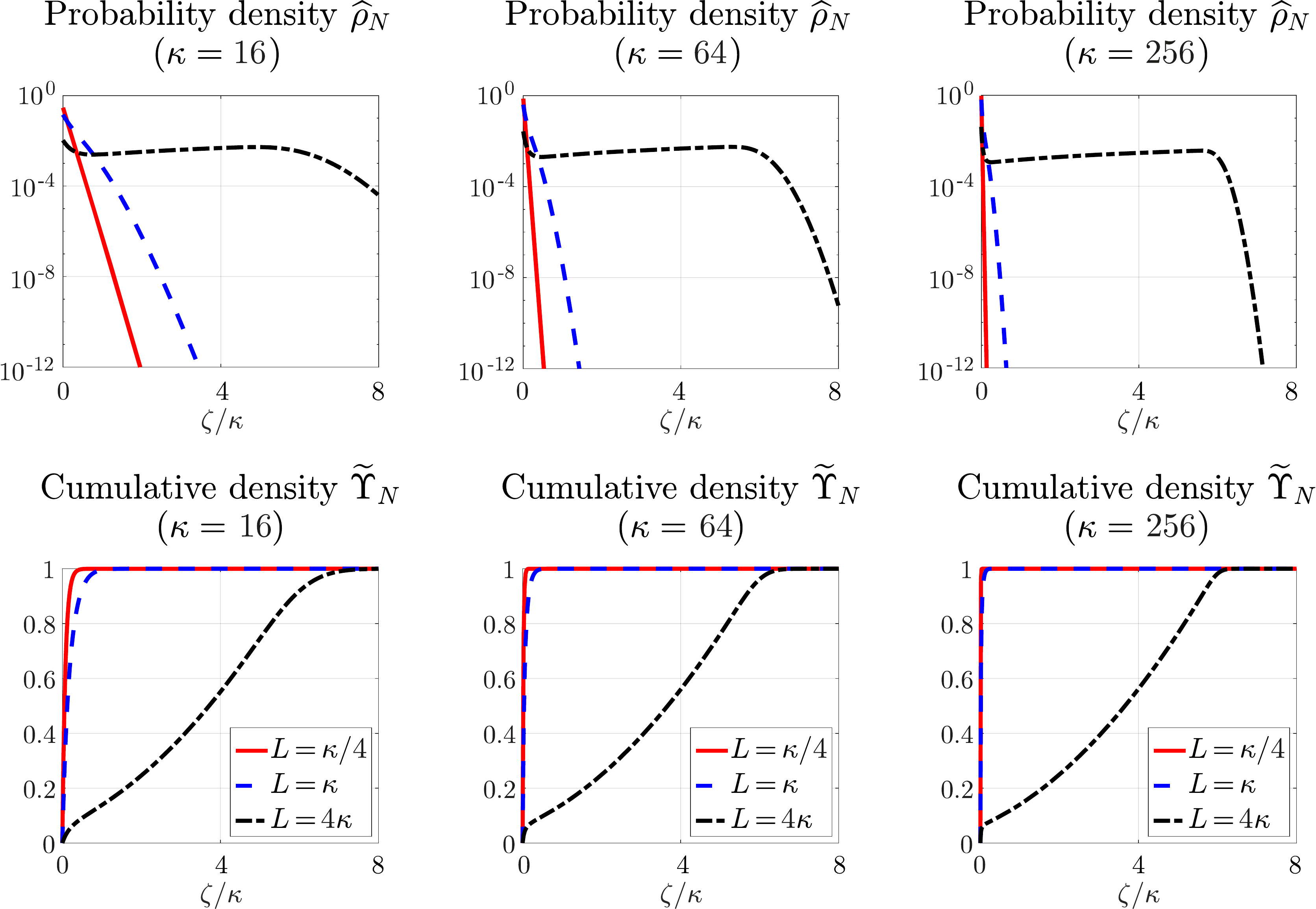}
\caption{Sampling density functions $\widehat{\rho}_N$ in (\ref{rho zeta density}) (top) and $\widetilde{\Upsilon}_N$ in (\ref{approx zeta cumulative}) (bottom) with respect to the $\kappa$-scaled evanescence parameter $\zeta$.}
\label{figure 5.1}
\end{figure}

\subsection{Inversion transform sampling}\label{s:ITS}

Similar to \cite[sect.~8.1]{parolin-huybrechs-moiola}, $P \propto N(L)$ samples in $Y$ are generated via the \emph{Inversion Transform Sampling} (ITS) technique suggested by \cite[sect.~5.2]{Cohen_Migliorati}.
We propose two alternative versions:
\begin{itemize}
\item The first one involves generating sampling sets in $[0,1]^4$ converging (in a suitable sense) to the uniform distribution $\mathcal{U}_{[0,1]^4}$ as $P$ goes to infinity, namely
\begin{equation*}
\{\mathbf{z}_p\}_{p=1}^P, \qquad \text{with} \qquad \mathbf{z}_p=(z_{p,\theta_1},z_{p,\theta_2},z_{p,\psi},z_{p,\zeta}) \in [0,1]^4, \qquad p=1,...,P,
\label{points in [0,1]}
\end{equation*}
and mapping them back to $Y$, to obtain sampling sets that converge to $\rho_N$ as $P \rightarrow \infty$, i.e.\
\begin{equation*}
\{\mathbf{y}_p\}_{p=1}^P, \qquad \text{with} \qquad \mathbf{y}_p=\left(\arccos{(1-2z_{p,\theta_1})},2\pi z_{p,\theta_2},2\pi z_{p,\psi},\Upsilon_N^{-1}(z_{p,\zeta})\right) \in Y.
\label{points back in Y}
\end{equation*}

\item The second one exploits the fact that the parameters $\boldsymbol{\theta}=(\theta_1,\theta_2)$ should be distributed in such a way that  the resulting points on the sphere converge to the uniform distribution $\mathcal{U}_{\mathbb{S}^2}$ as $P$ goes to infinity.
This enables us to employ the spherical coordinates $\{(\widehat{\theta}_{p,1},\widehat{\theta}_{p,2})\}_{p=1}^P$ of nearly-uniform point sets on $\mathbb{S}^2$ (e.g.\ extremal systems \cite{marzo-cerda,reimer,sloan-womersley1,sloan-womersley2} in our numerical experiments) to determine the wave propagation directions, restricting the ITS technique to the evanescence parameters $(\psi,\zeta)$.
Thus, we only need to generate sampling sets in $[0,1]^2$ that converge to the uniform distribution $\mathcal{U}_{[0,1]^2}$ as $P \rightarrow \infty$, i.e.\
\begin{equation}
\{\mathbf{z}_p\}_{p=1}^P, \qquad \text{with} \qquad \mathbf{z}_p=(z_{p,\psi},z_{p,\zeta}) \in [0,1]^2, \qquad p=1,...,P.
\label{points in [0,1] 2}
\end{equation}\
Then, we map them back to the evanescence domain $[0,2\pi)\times[0,+\infty)$, obtaining
\begin{equation}
\{\mathbf{y}_p\}_{p=1}^P, \qquad \text{with} \qquad \mathbf{y}_p=\left(\widehat{\theta}_{p,1},\widehat{\theta}_{p,2},2\pi z_{p,\psi},\Upsilon_N^{-1}(z_{p,\zeta})\right) \in Y.
\label{points back in Y 2}
\end{equation}
\end{itemize}
Computing the inverse $\Upsilon_{N}^{-1}$ can be achieved through elementary root-finding methods.
In our numerical experiments, we use the bisection method due to its simplicity and reliability.

\begin{figure}[t!]
\centering
\begin{tabular}{ccccc}
\multicolumn{2}{c}{$L=\kappa$} & & \multicolumn{2}{c}{$L=4\kappa$}\\
\includegraphics[trim=125 220 125 220,clip,width=.2\linewidth]{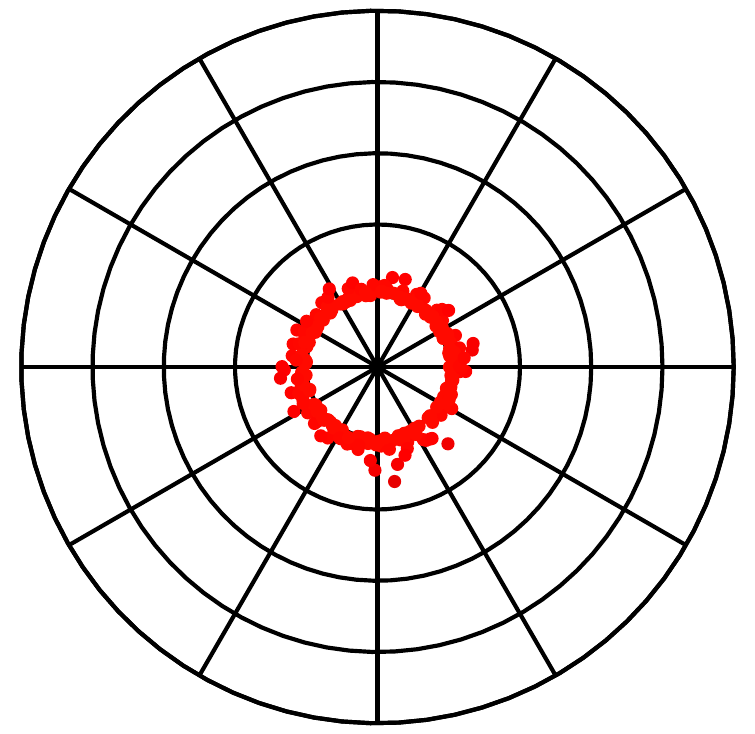} & \includegraphics[trim=125 220 125 220,clip,width=.2\linewidth]{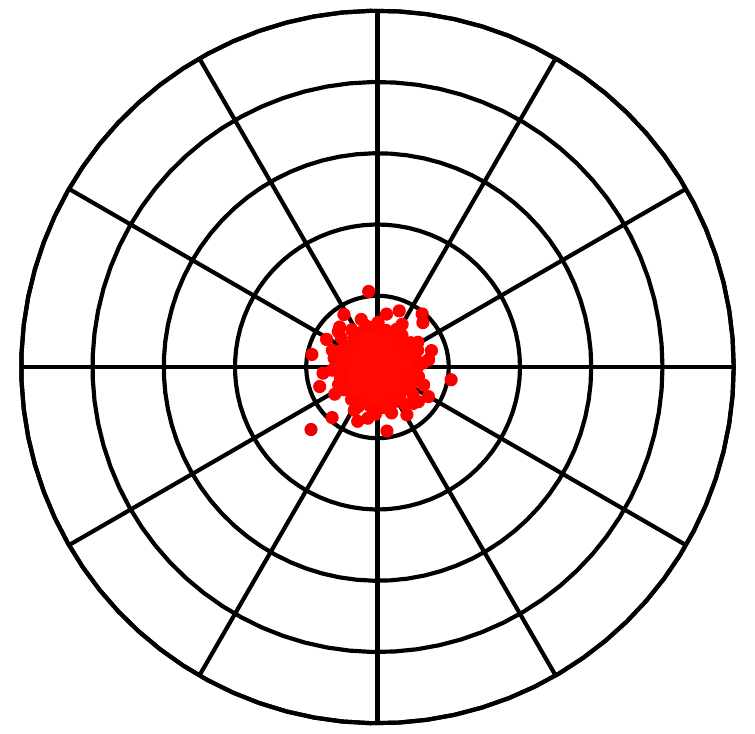} & \qquad \qquad & \includegraphics[trim=125 220 125 220,clip,width=.2\linewidth]{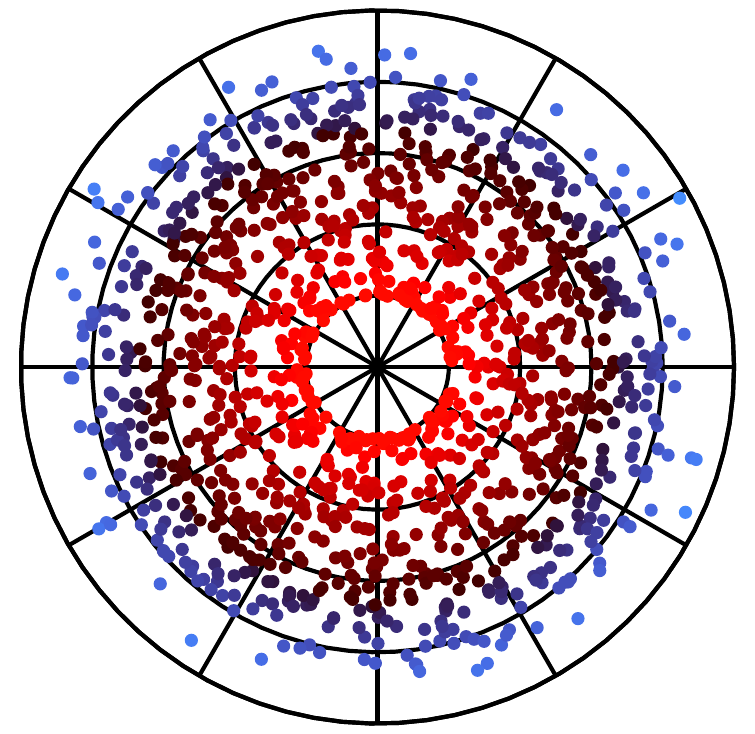} & 
\includegraphics[trim=125 220 125 220,clip,width=.2\linewidth]{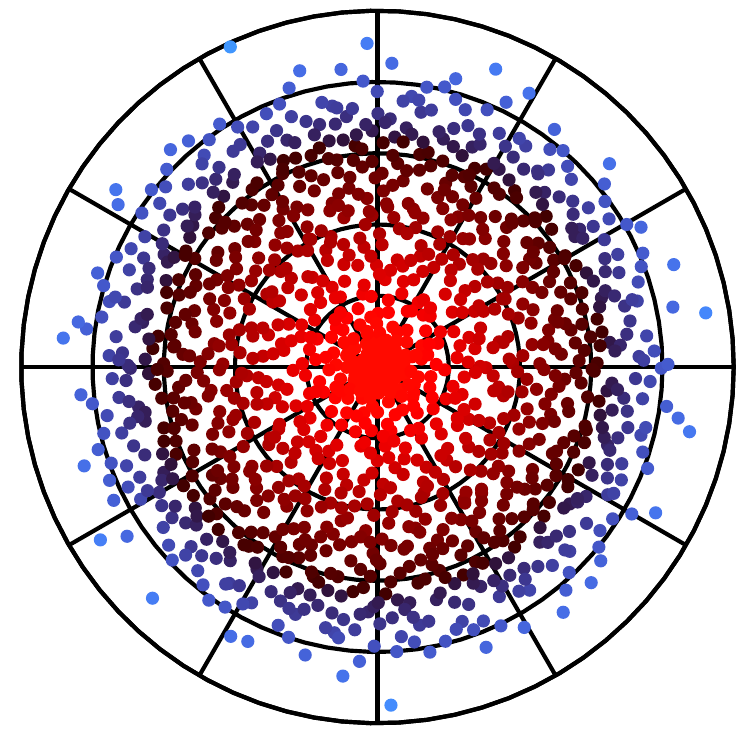}\\
\end{tabular}
\caption{$P=N(L)$ samples for $L$ equal to $\kappa$ (left) and $4\kappa$ (right).
For each value of $L$, on the left side we plot $\{(|\Re(\mathbf{d}(\mathbf{y}_p))|,\widehat{\theta}_{p,2})\}_{p=1}^P$, and on the right one $\{(|\Im(\mathbf{d}(\mathbf{y}_p))|,2\pi z_{p,\psi})\}_{p=1}^P$; see (\ref{clearer behavior}).
The points are colored according to the square root of $\mu_N$ in (\ref{rho density}). Wavenumber $\kappa=16$.}
\label{figure 5.2}
\end{figure}

\paragraph{Sampling strategies}
In analogy with \cite[sect.~8.1]{parolin-huybrechs-moiola}, here we briefly review some sampling methods, which differ by how we generate the distribution $\{\mathbf{z}_p\}_{p=1}^P$ in $[0,1]^n$, for $n \in \{2,4\}$:
\begin{itemize}
\item \emph{Deterministic sampling}: the samples are a Cartesian product of $n$ sets of equispaced points with equal number of points in each directions. 
\item \emph{Quasi-random sampling}: the samples correspond to quasi-random low-discrepancy sequences, such as Sobol sequences \cite{Bratley,Joe}, for instance.
\item \emph{Random sampling}: the samples are generated randomly according to the product of $n$ uniform distributions $\mathcal{U}_{[0,1]}$.
\end{itemize}
In the following numerical experiments, we use a sampling strategy that combines quasi-random Sobol sequences \cite{Bratley,Joe} and extremal point systems \cite{marzo-cerda,reimer,sloan-womersley1,sloan-womersley2}, according to (\ref{points in [0,1] 2}) and (\ref{points back in Y 2}).
Examples of EPW approximation sets constructed in this way are depicted in Figure~\ref{figure 5.2}.
As anticipated, for smaller values of $L$, which correspond to the regime where
PPWs provide a sufficient approximation, propagation vectors $\Re(\mathbf{d})$ group around
$\mathbb{S}^2$ and evanescent vectors $\Im(\mathbf{d})$ cluster near the origin.
When $L > \kappa$,
the target space $\mathcal B_L$ includes Fourier modes with finer oscillations along $\partial B_1$ and strong radial decay away from $\partial B_1$, whose approximation requires EPWs with comparable properties
so both $|\Re(\mathbf{d})|$ and
$|\Im(\mathbf{d})|$ increase.
This aligns with Figure~\ref{figure 5.1} results.

\paragraph{Approximation sets}
Two approximation sets can now be constructed, one consisting of sampling functionals in $\mathcal{A}$ and the other of EPWs in $\mathcal{B}$, namely
\begin{equation}
\Psi_{L,P}:=\,\,\Biggl\{\sqrt{\frac{\mu_N(\mathbf{y}_{p})}{P} }K_{\mathbf{y}_{p}}\Biggr\}_{p=1}^P \subset \mathcal{A} \qquad   \rightleftarrowss{\Tcont{Y}{EW}}{\left(\Tcont{Y}{EW}\right)^{-1}} \qquad \Phi_{L,P}:=\Biggl\{\sqrt{\frac{\mu_N(\mathbf{y}_{p})}{P} }\textup{EW}_{\mathbf{y}_{p}}\Biggr\}_{p=1}^P \subset \mathcal{B}.
\label{isomorphism approximation sets}
\end{equation}
Similarly to \cite[Conj.\ 7.1]{parolin-huybrechs-moiola}, we conjecture that
the evaluation functional sequence $\{\Psi_{L,P}\}_{L \geq 0, P \in \mathbb{N}}$ is a
stable discrete approximation for the space of Herglotz densities
$\mathcal{A}$, hence the \textup{EPW} sequence $\{\Phi_{L,P}\}_{L\geq 0,P \in
\mathbb{N}}$ is a stable discrete approximation for the space of Helmholtz solutions in
the ball
$\mathcal{B}$. This assertion is supported by the numerical results
in section~\ref{sec:numerical results}.

\paragraph{Cumulative density function approximation}
The ITS technique requires to invert the cumulative density function $\Upsilon_{N}$.
Although this can be easily done for every $(\boldsymbol{\theta},\psi)$, the numerical evaluation of the cumulative probability distribution (\ref{cumulative distribution}) is cumbersome 
to implement, costly to run and numerically unstable. In fact, due to (\ref{vector P}), (\ref{weight}), (\ref{1 lemma 5.3}), and (\ref{rho density}), we should compute:
\begin{align*}
\Upsilon_{N}(\zeta)&=\int_{0}^{\zeta}\int_{\Theta}\int_0^{2\pi}\frac{w(\eta)}{N \mu_N(\boldsymbol{\theta},\psi,\eta)}\,\textup{d}\psi\textup{d}\sigma(\boldsymbol{\theta})\textup{d}\eta=2\pi|\mathbb{S}^2|\frac{1}{N}\sum_{\ell=0}^L\alpha_{\ell}^2\int_{0}^{\zeta}\left|\mathbf{P}_{\ell}(\eta)\right|^2w(\eta)\,\textup{d}\eta\\
&=\frac{1}{N}\sum_{\ell=0}^L(2\ell+1)\frac{\sum_m\int_0^{\zeta}\left[\gamma_{\ell}^m P_{\ell}^m(\eta/2\kappa+1)\right]^2\eta^{1/2}e^{-\eta}\,\textup{d}\eta}{\sum_m\int_0^{\infty}\left[\gamma_{\ell}^m P_{\ell}^m(\eta/2\kappa+1)\right]^2\eta^{1/2}e^{-\eta}\,\textup{d}\eta}. \numberthis \label{upsilon}
\end{align*}
Our proposal is thus to rely on the following approximation:
\begin{equation}
\widetilde{\Upsilon}_N(\zeta):=1-\frac{1}{N}\sum_{\ell=0}^L(2\ell+1)\frac{Q\left(2\ell+3/2,2\kappa+\zeta \right)}{Q\left(2\ell+3/2,2\kappa\right)} \qquad \forall \zeta \in [0,+\infty),
\label{approx zeta cumulative}
\end{equation}
where $Q$ is the \emph{normalized upper incomplete Gamma function} \cite[eq.~(8.2.4)]{nist}.
The approximation $\widetilde\Upsilon_N$ is obtained from \eqref{upsilon} by reasoning similarly to \eqref{3 lemma 5.3}--\eqref{2 lemma 5.3}: approximating $P_\ell^m$ with a monomial and controlling $\eta^{1/2}$ with $(\eta+2\kappa)^{1/2}$; the details are expounded in \cite[sect.~6.3]{galante}.
Compared to (\ref{upsilon}), this concise explicit expression is better suited for numerical evaluation.
The function $\widetilde{\Upsilon}_N$ maintains the following essential properties:
  $0 \leq \widetilde{\Upsilon}_N(\zeta) \leq 1$,
  $\widetilde{\Upsilon}_N(0) = 0$ and
  $\lim_{\zeta \to \infty}\widetilde{\Upsilon}_N(\zeta) = 1$.
Some cumulative density functions $\widetilde{\Upsilon}_N$ are shown in the bottom row of Figure
\ref{figure 5.1}. When $\mathcal{A}_L$ only consists of elements
related to the propagative regime ($L\leq\kappa$), the cumulative distributions
$\widetilde{\Upsilon}_N$ resemble step functions, especially for large wavenumbers.
However, for $L > \kappa$, these functions become more complex. Thus, while it
is safe to choose only PPWs for $L \leq \kappa$, selecting EPWs becomes a
non-trivial task for $L > \kappa$.

\paragraph{Normalization coefficient approximation}
The normalization of the EPWs in $\Phi_{L,P}$ involves computing the $N\!$-term Christoffel function $\mu_N$, which, according to (\ref{rho density}), depends on both the normalization coefficients $\alpha_{\ell}$ in (\ref{a tilde definizione}) and $\mathbf{P}_{\ell}(\zeta)$ in (\ref{vector P}). While the latter can be computed through recurrence relations \cite[eqs.~(14.7.15) and (14.10.3)]{nist}, the former presents numerical challenges due to the integral in (\ref{1 lemma 5.3}). Once again, we can address this issue by relying on \cite[eq.~(6.19)]{galante} and employing the approximation:
\begin{equation}\label{eq:AlphaApprox}
\alpha_{\ell} \approx \frac{\kappa^{\ell}}{e^{\kappa}}\left[\frac{2\sqrt{\pi}}{\ell!}\Gamma\left(\ell+\frac{1}{2}\right)\Gamma\left(2\ell+\frac{3}{2},2\kappa \right)\right]^{-1/2} \qquad \forall \ell \geq 0,
\end{equation}
where we introduced the \emph{upper incomplete Gamma function} \cite[eq.~(8.2.2)]{nist}.
Alternatively, simpler normalization options are possible, such as using the $L^{\infty}$-norm on the unit ball.

\paragraph{Parameter tuning}
The construction of the sets $\Phi_{L,P}$ requires choosing just two parameters, $L$ and $P$:
\begin{itemize}
\item $L$ is the Fourier truncation level. As $L$ increases, the accuracy of the approximation of $u$ by $u_L$, or similarly of $v=\left(\Tcont{Y}{EW}\right)^{-1}u$ by $v_L$, improves.
\item $P$ is the EPW approximation space dimension.
When $L$ is fixed, increasing $P$ allows to enhance the accuracy of the approximation of  
$u_L$ (or $v_L$) by elements of $\Phi_{L,P}$ (or $\Psi_{L,P}$).
The empirical evidence detailed in section~\ref{sec:numerical results} validates this conjecture, showing experimentally that $P$ should scale linearly with $N(L)$, with a moderate proportionality constant.
\end{itemize}
Regarding the additional parameters discussed in section~\ref{subsec:regularized boundary sampling method}, namely the number $S$ of sampling points on $\partial B_1$ and the SVD regularization parameter $\epsilon$, in our numerical experiments we choose $S=\lceil \sqrt{2P} \rceil^2$ and $\epsilon=10^{-14}$ respectively, in accordance with \cite[sect.~6.1]{galante}.

\section{Numerical results}\label{sec:numerical results}

The numerical experiments presented in this section (see \cite[Ch.\ 7]{galante}
for more results) show the stability and accuracy achieved by the EPW sets
constructed in the previous section.
First, we consider the problem of the approximation of a spherical wave by either PPWs or EPWs,
confirming in particular the instability result of Lemma \ref{Lemma 4.4} and showing the radical improvement offered by EPWs.
Then, we explore the near-optimality of the EPW set by reconstructing
random-expansion solutions and analyzing the error convergence.
Further numerical results show that our recipe is very effective also for non-spherical domains.
For the sake of simplicity, we opted for the boundary sampling strategy described above to investigate numerically the EPW approximation properties.
However, alternative numerical methods (TDG, UWVF, \dots) can be employed together with EPWs.

\subsection{Plane wave stability}

Let us examine the approximation of spherical waves by the PPW and EPW approximation sets introduced in (\ref{plane waves approximation set}) and (\ref{isomorphism approximation sets}), respectively.
We will focus only on the case $m=0$, since the numerical results do not differ significantly varying the order $|m|\leq \ell$, as shown in \cite[Fig.\ 3.4 and Fig.\ 7.1]{galante}.
Once the approximation set is fixed, the same sampling matrix $A$, defined in (\ref{A matrix definition}), is used to approximate all the $b_{\ell}^0$ for $0\leq \ell \leq 5\kappa$.

\begin{figure}[t!]
\centering
\begin{subfigure}{0.46\linewidth}
\includegraphics[width=\linewidth]{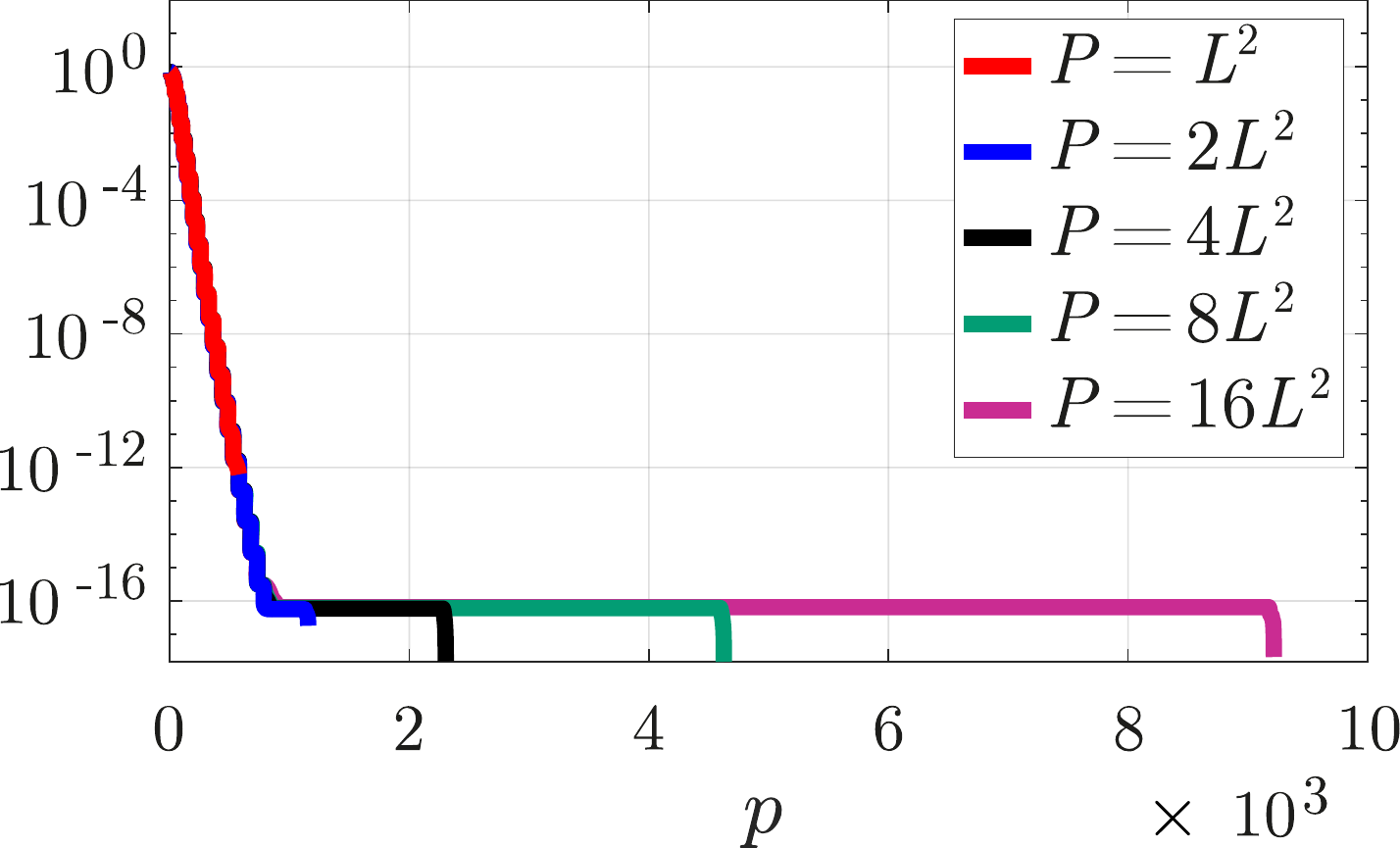}
\end{subfigure}
\hfill
\begin{subfigure}{0.46\linewidth}
\includegraphics[width=\linewidth]{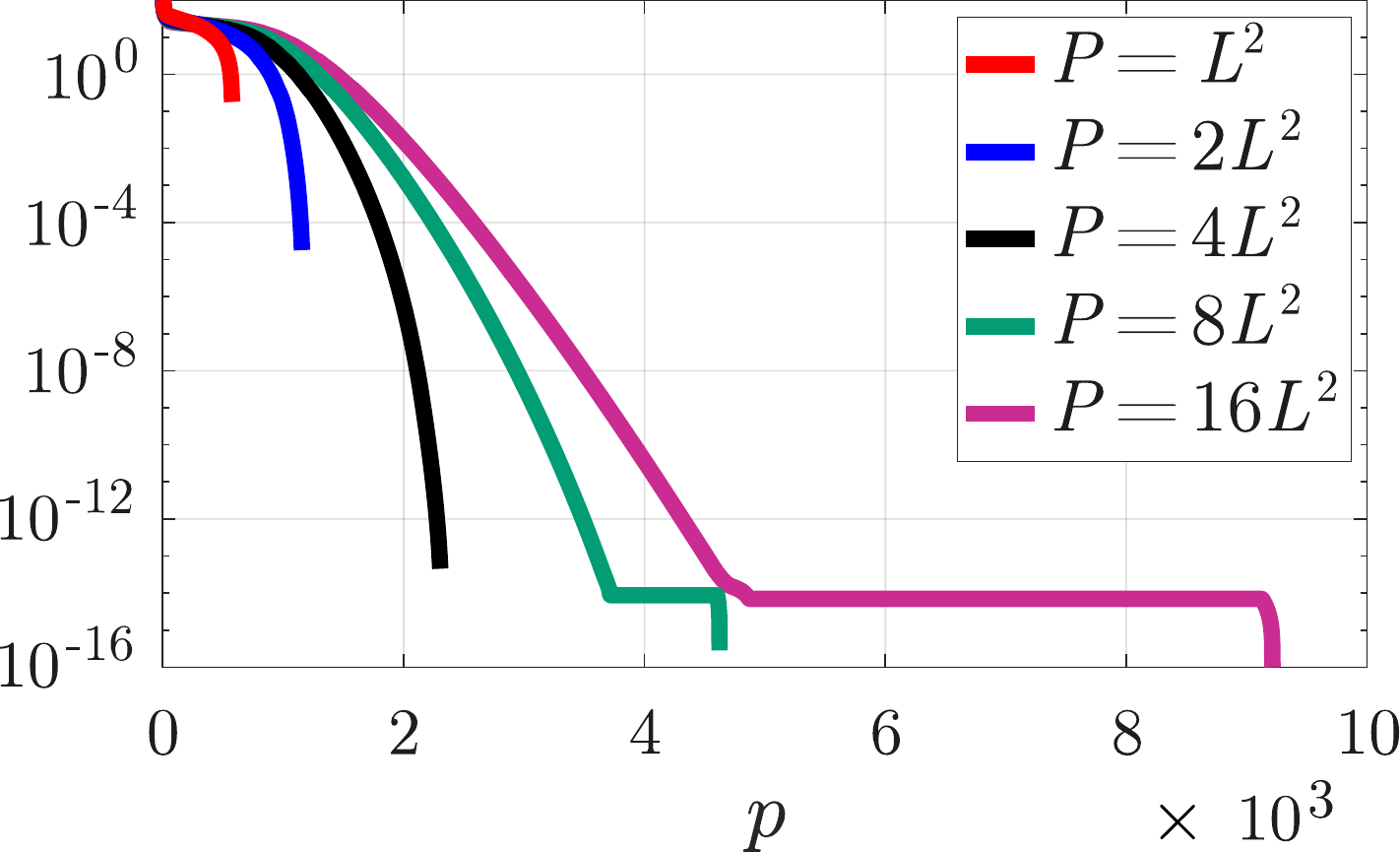}
\end{subfigure}
\caption{Singular values $\{\sigma_p\}_{p}$ of $A$ using PPWs (left) and EPWs (right). Wavenumber $\kappa=6$.
Raising the number $P$ of waves, 
the $\epsilon$-rank of the matrix increases for EPWs but not for PPWs.}
\label{figure 6.1}
\end{figure}

The matrix $A$ is known to be ill-conditioned: its condition number increases exponentially with the number of plane waves, a trend that can be inferred from Figure~\ref{figure 6.1}. This phenomenon is not unique to the sampling method and is observed in other experiments, see \cite[sect.~4.3]{hiptmair-moiola-perugia1}.

\begin{figure}[t!]
\centering
\begin{subfigure}{0.94\linewidth}
\includegraphics[width=\linewidth]{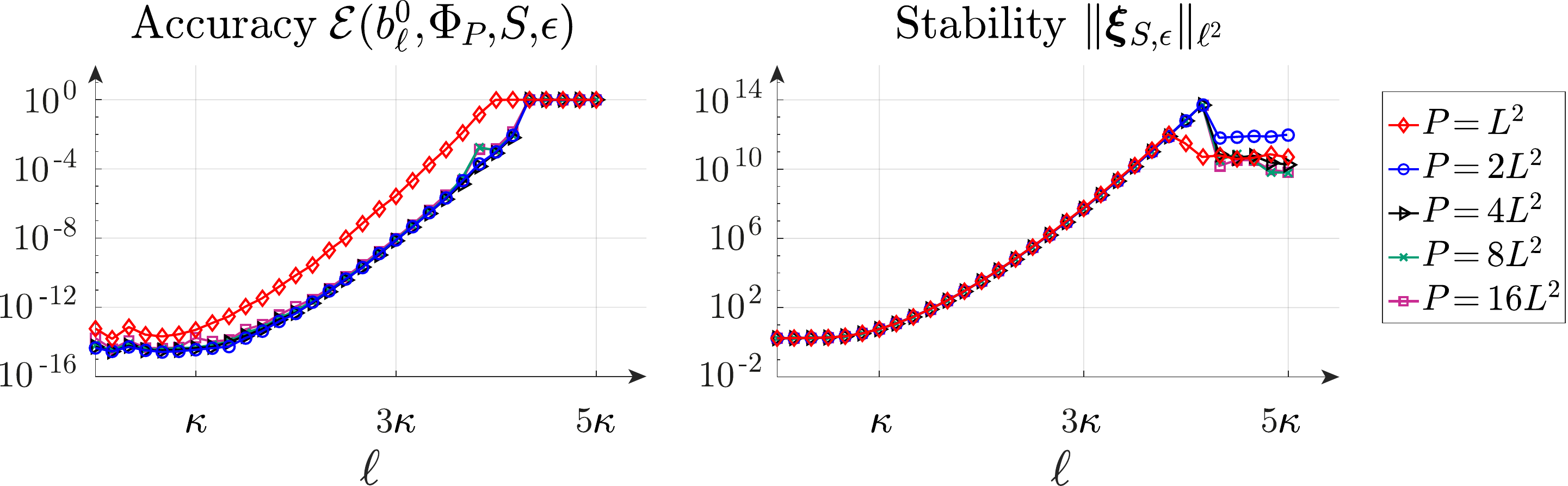}
\vspace{-1mm}
\end{subfigure}
\begin{subfigure}{0.94\linewidth}
\includegraphics[width=\linewidth]{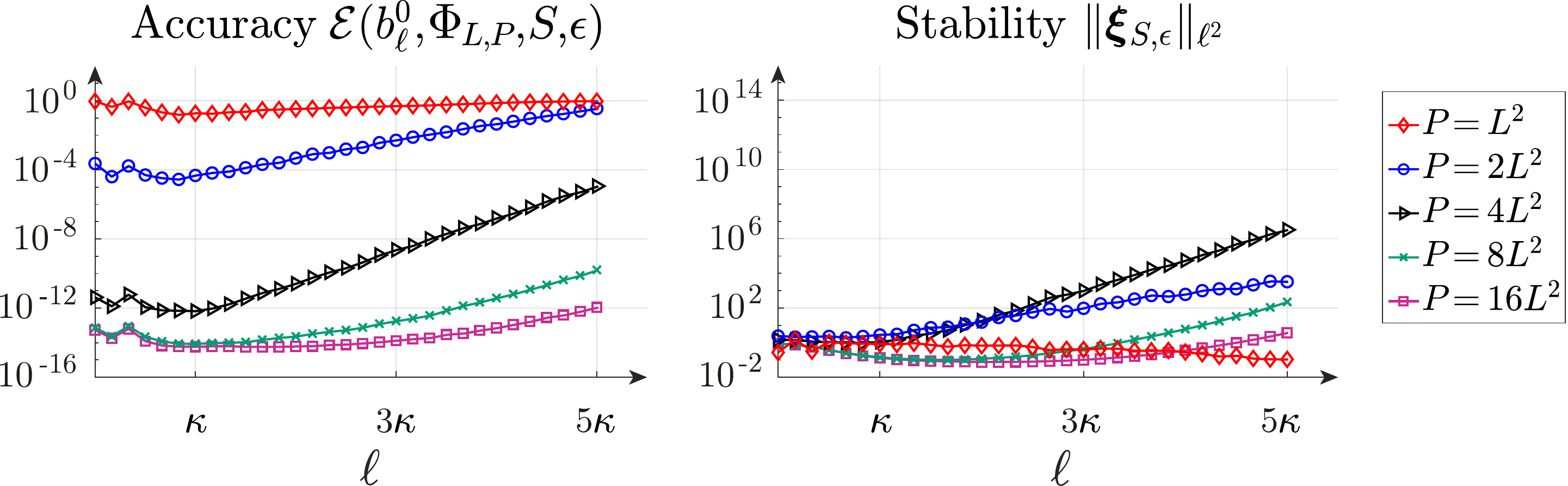}
\end{subfigure}
\caption{Accuracy $\mathcal{E}$ (\ref{relative residual}) (left) and stability $\|\boldsymbol{\xi}_{S,\epsilon}\|_{\ell^2}$ (right) of the approximation of spherical waves $b_{\ell}^0$ by both PPWs (top) and EPWs (bottom). Truncation at $L=4\kappa$ and wavenumber $\kappa=6$.}
\label{figure 6.2}
\end{figure}

In this setting, more useful than the condition number is 
the concept of $\epsilon$-rank, i.e.\ the number of singular values of $A$ larger than $\epsilon\sigma_{\max}$, which corresponds to the dimension of the numerically achievable approximation space:
\begin{align*}
\epsilon\text{-rank of }A\, : &=\# \{\sigma_p\ge \epsilon\sigma_{\max}\}
=\mathrm{dim}\,
  \big\{\mathbf{b} \in \mathbb{C}^{S} \;|\; 
    \exists \boldsymbol{\mu} \in \mathbb{C}^{|\Phi_{P}|},\ 
    \mathbf{b} = A\boldsymbol{\mu},\ 
    \|\boldsymbol{\mu}\|_{\ell^{2}} \leq
    \textstyle\frac1{\epsilon\sigma_{\max}}\|\mathbf{b}\|_{\ell^{2}}
  \big\}  \\
&=\mathrm{dim}\, \big\{
    u \in \operatorname{span} \Phi_{P} \;|\; 
    \exists \boldsymbol{\mu} \in \mathbb{C}^{|\Phi_{P}|},\ 
    u = \Tdisc{P}{PW}\boldsymbol{\mu},\ 
    \|\boldsymbol{\mu}\|_{\ell^{2}} \leq 
    \textstyle\frac1{\epsilon\sigma_{\max}} \|\mathbf{b}(u)\|_{\ell^{2}}
  \big\},
\end{align*}
where \(\mathbf{b}(u)\in \mathbb C^S\) is the boundary sampling vector of $u$, as in~\eqref{A matrix definition}.

The approximation results are shown in Figure~\ref{figure 6.2}. The left panels show the relative residual $\mathcal{E}$ in (\ref{relative residual}) as a measure of the approximation accuracy. On the right panels, the coefficient size $\|\boldsymbol{\xi}_{S,\epsilon}\|_{\ell^2}$ indicates the stability of the approximations.

\paragraph{Propagative plane waves}
Let us focus on the PPW approximation sets $\Phi_P$ in (\ref{plane waves approximation set}).

Figure~\ref{figure 6.1} shows that the $\epsilon$-\emph{rank} of the matrix $A$ does not increase as $P$ is raised: more PPWs do not lead to the stable approximation of more Helmholtz solutions.
\vspace{.5cm}

In Figure~\ref{figure 6.2} (top), three distinct regimes are observed:
\begin{itemize}
\item For the propagative modes, i.e.\ for spherical waves with mode number $\ell \leq \kappa$, the approximation is accurate ($\mathcal{E}<10^{-13}$) and the size of the coefficients is moderate ($\|\boldsymbol{\xi}_{S,\epsilon}\|_{\ell^2} < 10$).
\item For mode numbers $\ell$ larger than the wavenumber $\kappa$, the norm of the coefficient vector grows exponentially in $\ell$ and the accuracy decreases proportionally.
\item At a certain point (roughly between $\ell=4\kappa$ and $\ell=5\kappa$ in this numerical experiment), the exponential growth of the coefficients completely destroys the stability of the approximation and we are unable to approximate the target $b_{\ell}^0$ with any significant accuracy.
\end{itemize}
As in \cite[sect.~4.4]{parolin-huybrechs-moiola}, increasing $P$ does not enhance accuracy beyond a certain threshold. 
Despite the matrix $A$ being extremely ill-conditioned, accuracy for propagative modes $\ell \leq \kappa$ reaches machine precision. 
On the other hand, evanescent modes with larger mode numbers $\ell \geq 4\kappa$ maintain an error of $\mathcal{O}(1)$, thanks to the simple regularization technique outlined in section~\ref{subsec:regularized boundary sampling method}.
In line with Theorem \ref{Theorem 4.5}, any regularization technique can mitigate but not eliminate the inherent instability of Trefftz methods employing PPWs. Even with regularization, achieving accurate approximation of evanescent modes within a given floating-point precision remains unattainable.

\paragraph{Evanescent plane waves}
Now, let us consider the EPW approximation sets $\Phi_{L,P}$ in (\ref{isomorphism approximation sets}) instead. In Figure~\ref{figure 6.2} (bottom), we fix the truncation parameter at $L=4\kappa$. 
With enough waves, i.e.\ $P$ large enough, all modes $\ell \leq L=4\kappa$ are approximated to near machine precision. This encompasses both propagative modes $\ell \leq \kappa$, which were already well-approximated using only PPWs, and evanescent modes $\kappa < \ell \leq L=4\kappa$, for which purely PPWs provided poor or no approximation.
Moreover, higher-degree modes $L=4\kappa < \ell \leq 5\kappa$ are also accurately approximated.
The coefficient norms $\|\boldsymbol{\xi}_{S,\epsilon}\|_{\ell^2}$ in the approximate expansions are moderate, differing from the propagative case.
From Figure~\ref{figure 6.1}, one understands that if $P$ is large enough, the
condition number of the matrix $A$ is comparable for both PPWs and EPWs.
Improved accuracy for evanescent modes does not arise from better conditioning
but from a higher $\epsilon$-rank: from less than $10^3$ for PPWs to around $5
\times 10^3$ for EPWs in the case $P=16L^2$. Raising the truncation parameter
$L$ allows to increase the $\epsilon$-rank of $A$: more solutions can be approximated with bounded coefficients by the EPWs.

\subsection{Approximation of random-expansion solutions}\label{subsec:approximation of random-expansion solution}

\begin{figure}
\centering
\includegraphics[width=0.86\linewidth]{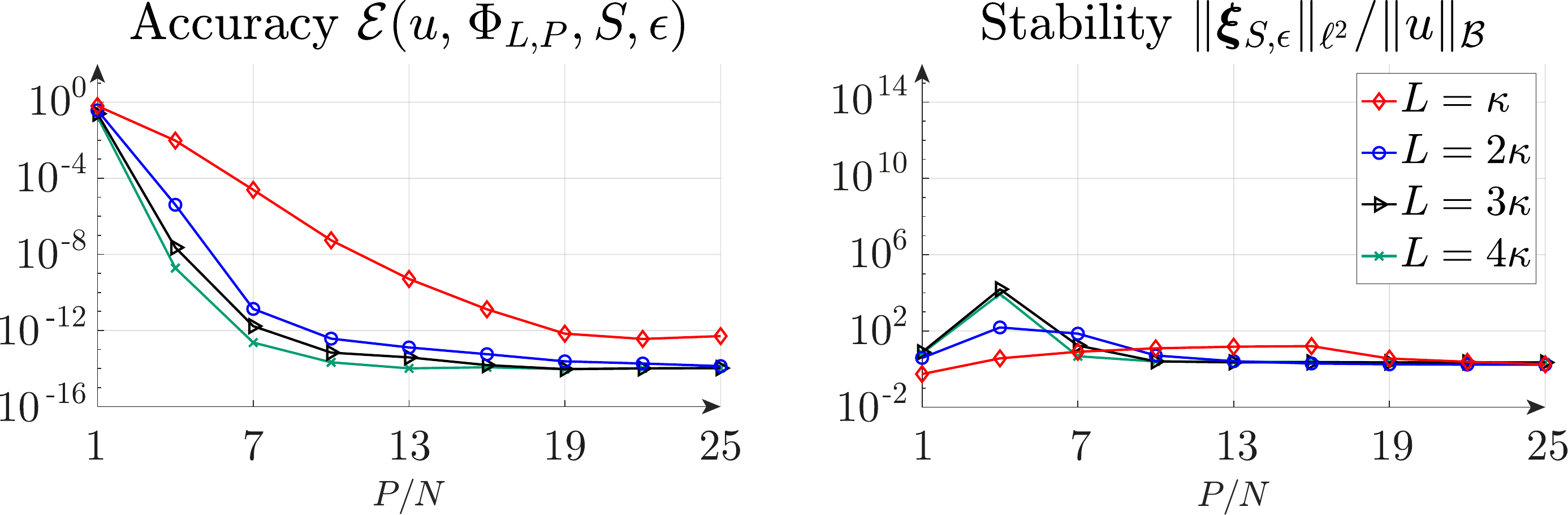}
\caption{Accuracy $\mathcal{E}$ (\ref{relative residual}) (left) and stability $\|\boldsymbol{\xi}_{S,\epsilon}\|_{\ell^2}/\|u\|_{\mathcal{B}}$ (right) of the approximation of solution $u$ in (\ref{random-expansion solution}) by $P$ EPWs. The horizontal axis represents the ratio $P/N(L)$, where $N(L)=(L+1)^2$ is the dimension of the space $\mathcal{B}_L$, to which $u$ belongs. Wavenumber $\kappa=6$.}
\label{figure 6.3}
\end{figure}

We test the numerical procedure presented in section~\ref{sec:numerical recipe} by reconstructing a solution of the form
\begin{equation}
  u:=\sum_{\ell=0}^{L}\sum_{m=-\ell}^{\ell}\widehat{u}_{\ell}^m[\max\{1,\ell-\kappa\}]^{-1}b_{\ell}^m \in \mathcal{B}_L,
\label{random-expansion solution}
\end{equation}
where the coefficients $\widehat{u}_{\ell}^m$ are independent, normally-distributed random numbers.
This is a challenging scenario, as the coefficients of any element in $\mathcal{B}$ decay in modulus as $o(\ell^{-1})$ for large $\ell$.

\begin{figure}[htb]
\centering
\begin{tabular}{cc}
\includegraphics[trim=120 120 120 120,clip,width=.28\textwidth]{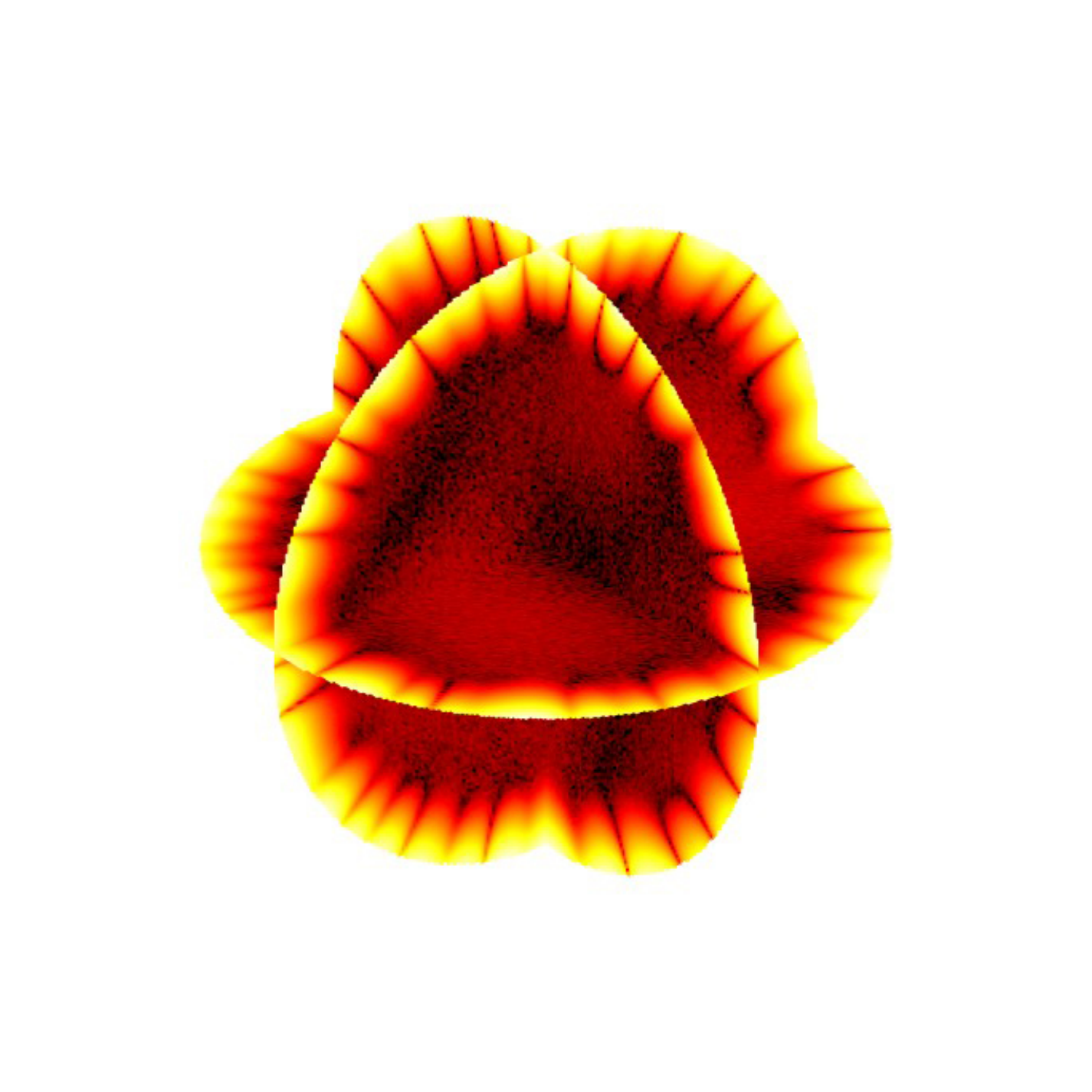} \qquad & \qquad
\includegraphics[trim=120 120 120 120,clip,width=.28\textwidth]{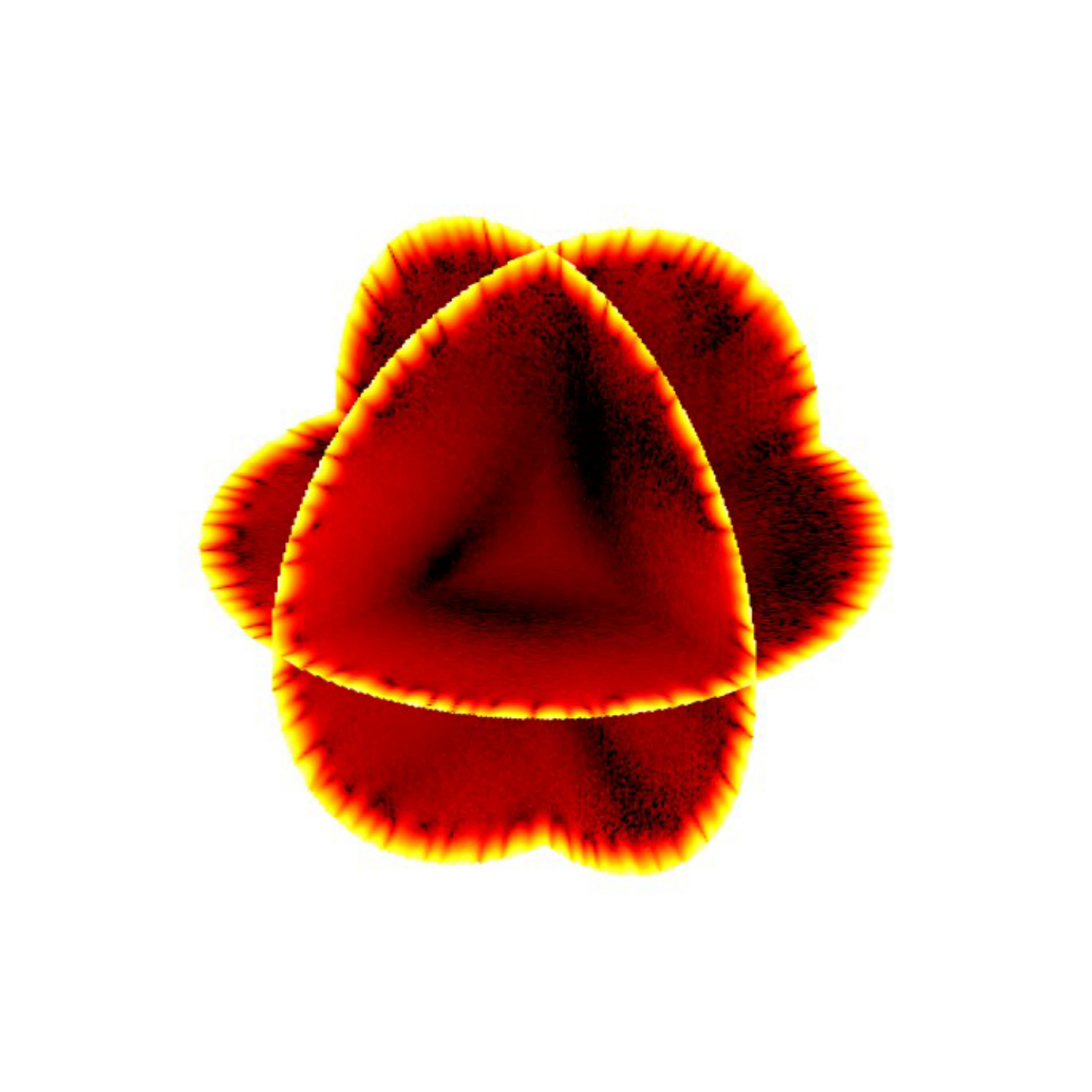}\\
\includegraphics[width=.35\textwidth]{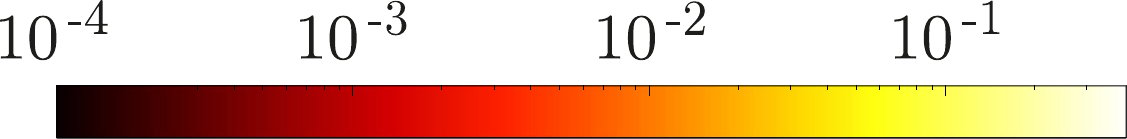} \qquad & \qquad
\includegraphics[width=.35\textwidth]{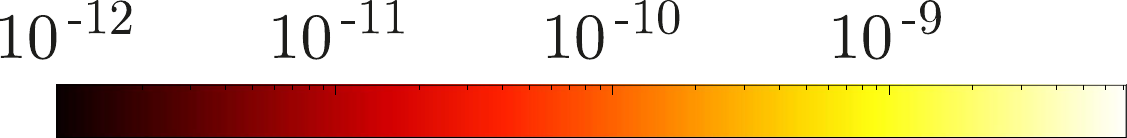} \\
\includegraphics[trim=120 120 120 120,clip,width=.28\textwidth]{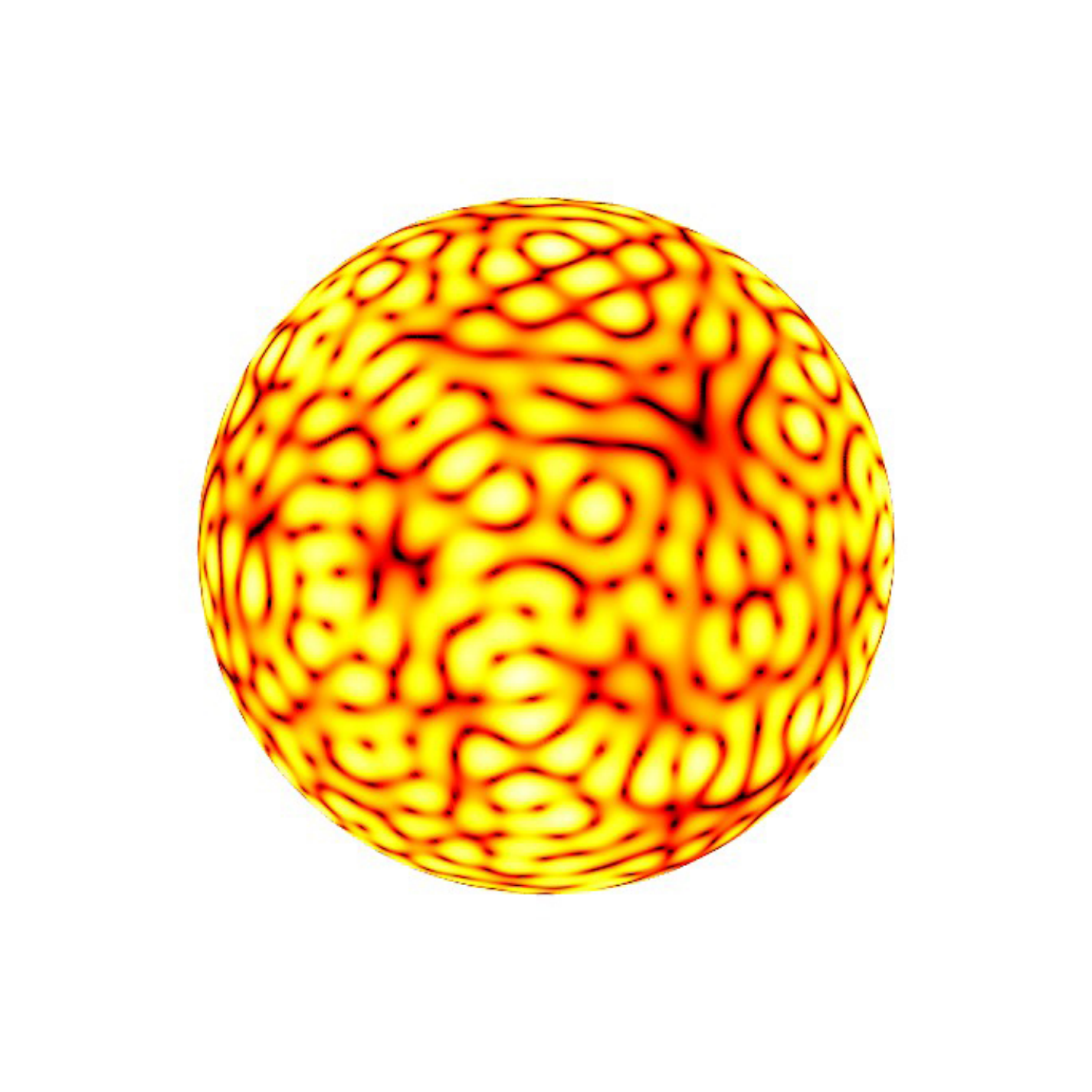} \qquad & \qquad
\includegraphics[trim=120 120 120 120,clip,width=.28\textwidth]{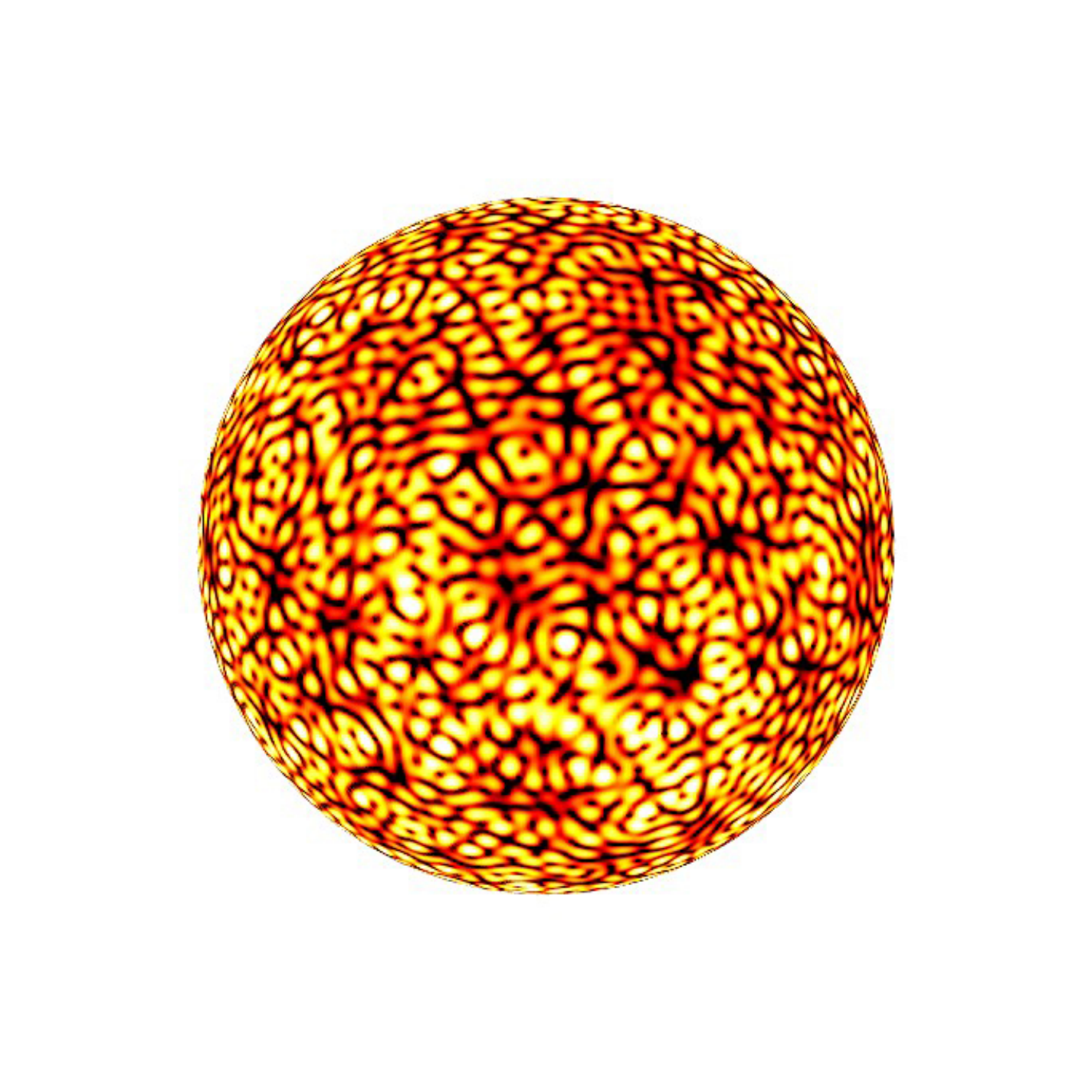}\\
\includegraphics[width=.35\textwidth]{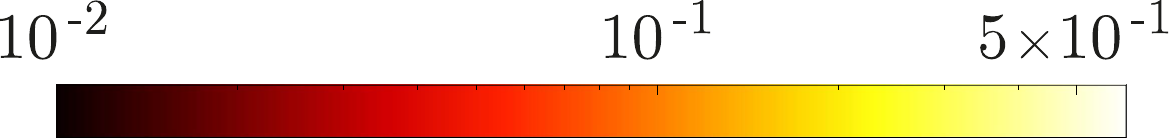} \qquad & \qquad
\includegraphics[width=.35\textwidth]{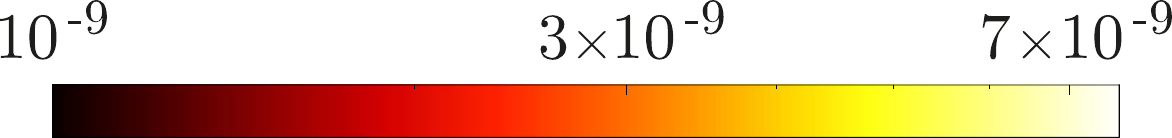}\\
\end{tabular}
\caption{Absolute errors of the approximation of a solution $u$ in (\ref{random-expansion solution}) with $\kappa=5$ and $L=5\kappa=25$, so that $u$ is defined by $N(L)=676$ random parameters.
The error corresponds to the use of  
$P=4(L+1)^2=2704$ plane waves, either PPWs in $\Phi_P$ from (\ref{plane waves approximation set}) (left) or EPWs in $\Phi_{L,P}$ from (\ref{isomorphism approximation sets}) (right).
Note the different ranges in the color scales.
The absolute errors are plotted both on $B_1 \cap \{\mathbf{x}=(x,y,z) : xyz=0\}$ (top) and on the unit sphere $\partial B_1$ (bottom).}
\label{figure 6.4}
\end{figure}

In Figure~\ref{figure 6.3} we display the relative residual $\mathcal{E}$ and the coefficient size $\|\boldsymbol{\xi}_{S,\epsilon}\|_{\ell^2}/\|u\|_{\mathcal{B}}$ with respect to the ratio $P/N(L)$, that is the approximation set dimension divided by the dimension of the space of the possible solutions in (\ref{random-expansion solution}).
The numerical results suggest that the size of the approximation set $P$ should vary linearly with respect to $N(L)$: when $L$ is large enough (e.g.\ $L \geq 2\kappa$), the decays are largely independent of $L$.
The $\Phi_{L,P}$ approximation sets (\ref{isomorphism approximation sets}) appear close to optimal, requiring only $\mathcal{O}(N)$ DOFs with a moderate proportionality constant to approximate $N$ spherical modes with reasonable accuracy. Here, for $L\ge2\kappa$, $P=10N$ suffices to obtain $\mathcal E\le10^{-12}$.

Figure~\ref{figure 6.4} shows the absolute errors resulting from approximating a solution of the form~\eqref{random-expansion solution}, with wavenumber $\kappa=5$ and truncation parameter $L=5\kappa=25$, by $P=4(L+1)^2=2704$ plane waves, whether they are PPWs or EPWs. For other plots of this kind see \cite[sect.~7.2]{galante}.

The error from PPWs is much larger than that from EPWs (around $8$ orders of magnitude in $L^{\infty}$-norm) and is mainly concentrated near the boundary.
This happens because EPWs can effectively capture the higher Fourier modes of Helmholtz solutions, which PPWs cannot achieve.

The number of DOFs per wavelength $\lambda=2\pi/\kappa$ employed in each direction can be estimated by $\lambda \sqrt[3]{3P/4\pi}$, which is approximately $11$ in Figure~\ref{figure 6.4}.
In low-order methods, a common rule of thumb is around $6 \sim 10$ DOFs per wavelength for $1 \sim 2$ digits of accuracy. 
Remarkably, thanks to the selected EPWs, merely a fraction above this count yields more than $8$ digits of accuracy.

\subsection{Other geometries}\label{ss:cube}

To conclude, we present some numerical results in a cubic domain, as well as in two more complex shapes -- a cow and a submarine -- to show that the approximation set we developed, based on the analysis of the unit ball $B_1$, performs well on other geometries as well.
All domains have been uniformly scaled so that the unit sphere circumscribes them.
Additional results involving tetrahedrons can be found in \cite[sect.~7.4]{galante}.

In all cases, the goal is to approximate the Helmholtz fundamental solution
\begin{equation}
\mathbf{x} \mapsto \frac{1}{4\pi}\frac{e^{i\kappa|\mathbf{x}-\mathbf{s}|}}{|\mathbf{x}-\mathbf{s}|} \qquad \forall \mathbf{x} \in \Omega, \qquad \text{where} \qquad \mathbf{s} \in \mathbb{R}^3 \setminus \overline{\Omega},
\label{fundamental solution}
\end{equation}
using the recipe of section \ref{subsec:regularized boundary sampling method} for the unit ball, i.e.\ sampling Dirichlet data points on $\partial \Omega$ and solving an oversampled system via regularized SVD.
Specifically, using evenly spaced sampling points on $\partial \Omega$ allows us to select uniform weights in (\ref{A matrix definition}).
The truncation parameter $L$ is computed from $P$ as $L:=\max\{\lceil \kappa\rceil,\lfloor\sqrt{P/10}\rfloor\}$, based on the numerical results of section~\ref{subsec:approximation of random-expansion solution}.
Moreover, the EPWs in (\ref{isomorphism approximation sets}) are normalized to have unit $L^{\infty}$-norm on $\partial \Omega$, this being 
the sole deviation from the sets used for spherical geometry.

\paragraph{Cube.}

\begin{figure}
\centering
\includegraphics[width=0.95
\linewidth]{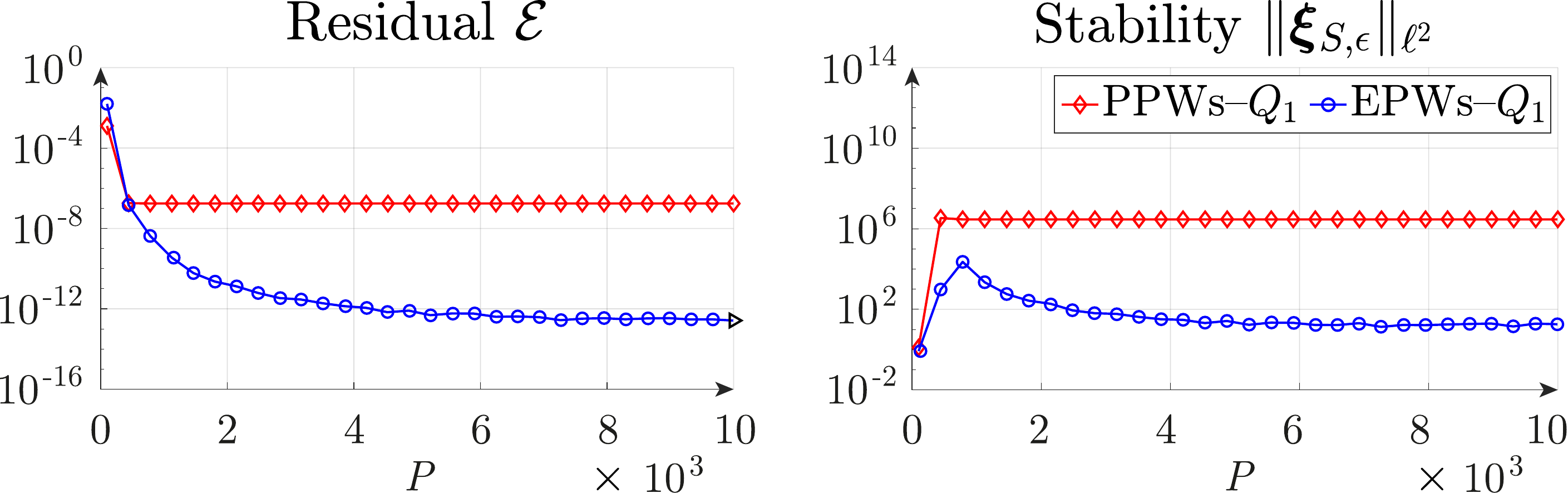}
\caption{Residual $\mathcal{E}$ (\ref{relative residual}) (left) and stability $\|\boldsymbol{\xi}_{S,\epsilon}\|_{\ell^2}$ (right) of the approximation of the fundamental solution (\ref{fundamental solution}) in the cube $Q_1$ with $\mathbf{s}=(1/\sqrt{3}+2\lambda/3,0,0)$, so that $\textup{dist}(\mathbf{s}, Q_1)=2\lambda/3$.
The PPW and EPW approximation sets $\Phi_P$ and $\Phi_{L,P}$ are compared.
We report the convergence for increasing size of the approximation set $P$.
Wavenumber $\kappa=5$.}
\label{figure 6.5}
\end{figure}

Let $Q_1$ denote the cube with edges aligned to the Cartesian axes and inscribed within the unit sphere.
In Figure~\ref{figure 6.5} we report the convergence of the plane wave approximations, either PPW or EPW, for increasing size of the approximation set $P$.
When PPWs are employed, the residual of the linear system initially reduces swiftly with increasing $P$, but eventually plateaus, well before reaching machine precision, due to the rapid growth of the coefficients. 
Conversely, when using EPW approximation sets, the residual converges to machine precision and the coefficient size remains reasonable.
In fact, by using EPWs, the truncation parameter $L$, and consequently the number of approximated modes, grows concurrently with $P$, providing an increasingly accurate approximation.
In contrast, PPWs are only able to stably approximate propagative modes.
Once this content is correctly captured, further increasing the discrete space dimension only brings instability, due to the impossibility of approximating high Fourier modes.

Figure~\ref{figure 6.6} shows the absolute errors in approximating a fundamental solution (\ref{fundamental solution}) by $P=2704$ plane waves, either PPWs or EPWs.

\begin{figure}[htb]
\centering
\begin{tabular}{cc}
\includegraphics[trim=135 260 170 175,clip,width=.28\textwidth]{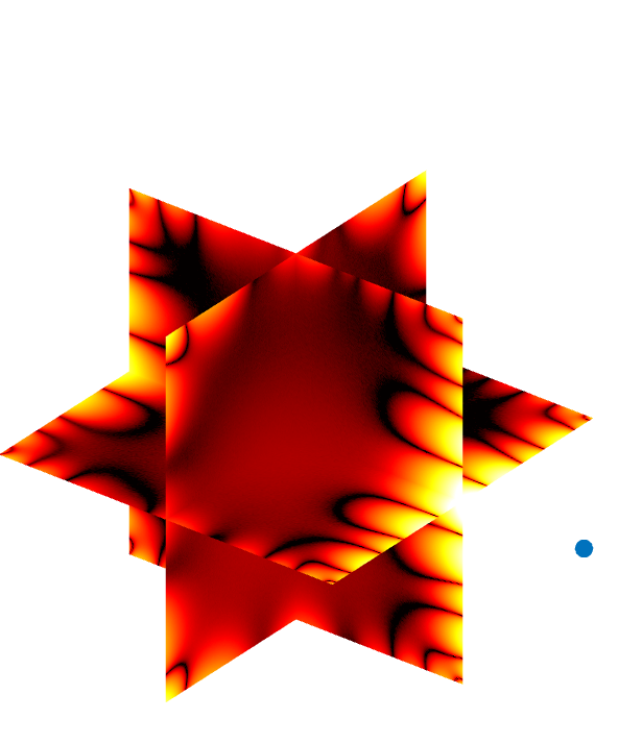} \qquad & \qquad
\includegraphics[trim=135 259 170 175,clip,width=.28\textwidth]{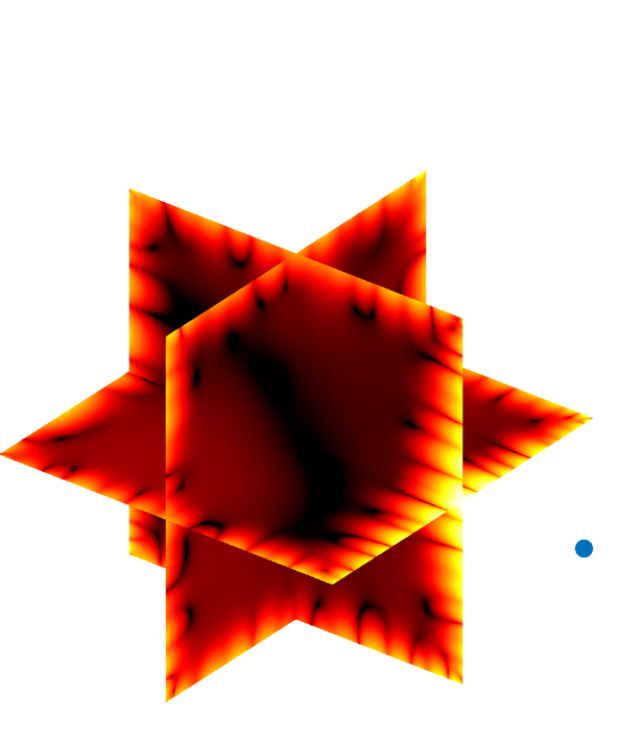}\\
\includegraphics[width=.35\textwidth]{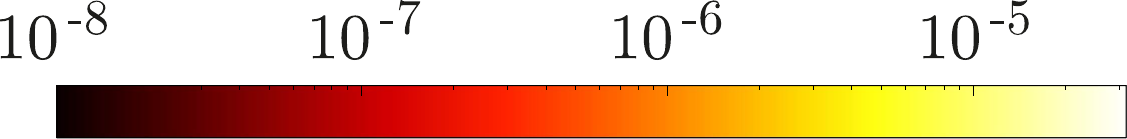} \qquad & \qquad
\includegraphics[width=.35\textwidth]{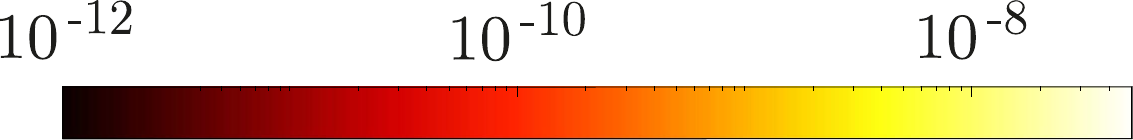} \\
\includegraphics[trim=135 245 165 175,clip,width=.27\textwidth]{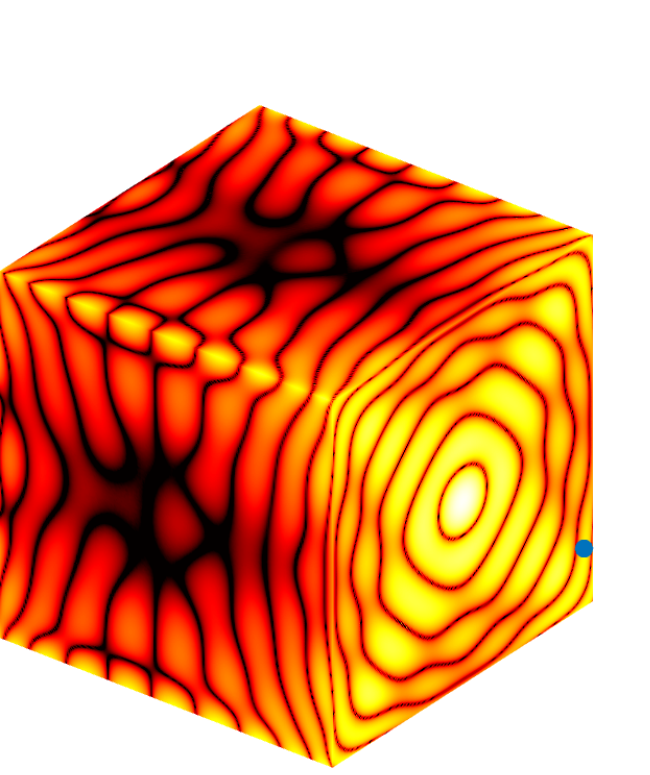} \qquad & \qquad
\includegraphics[trim=135 245 165 175,clip,width=.27\textwidth]{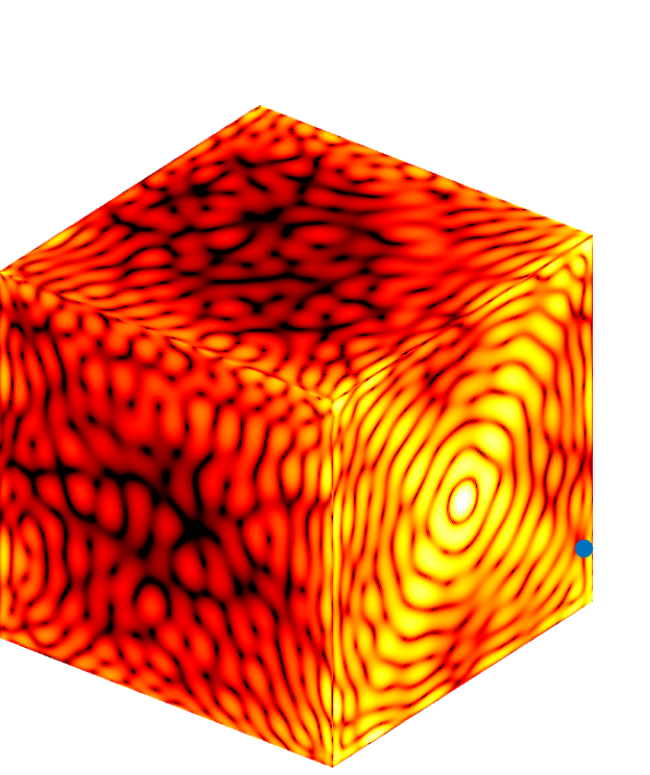}\\
\includegraphics[width=.35\textwidth]{images//bar_propagative_cube.pdf} \qquad & \qquad
\includegraphics[width=.35\textwidth]{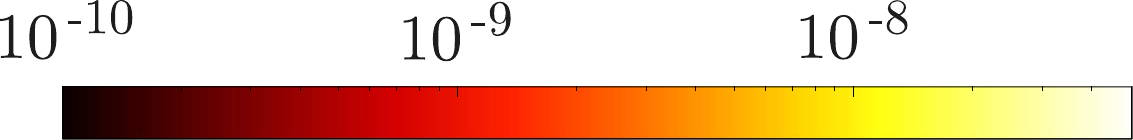} 
\end{tabular}
\caption{Absolute errors of the approximation of the fundamental solution (\ref{fundamental solution}) with $\kappa=5$ and $\mathbf{s} \in \mathbb{R}^3\setminus\overline{Q_1}$, marked by the blue dot in the plots, so that $\textup{dist}(\mathbf{s}, Q_1)=\lambda/3$.
The solution is approximated  
using $P=2704$ plane waves, either PPWs in $\Phi_P$ from (\ref{plane waves approximation set}) (left) or EPWs in $\Phi_{L,P}$ from (\ref{isomorphism approximation sets}) (right).
The absolute errors are plotted on both $Q_1 \cap \{\mathbf{x}=(x,y,z) : xyz=0\}$ (top) and the boundary $\partial Q_1$ (bottom).}
\label{figure 6.6}
\end{figure}

\begin{figure}
\centering
\begin{tabular}{cc}
\includegraphics[trim=135 240 165 230,clip,width=.27\textwidth]{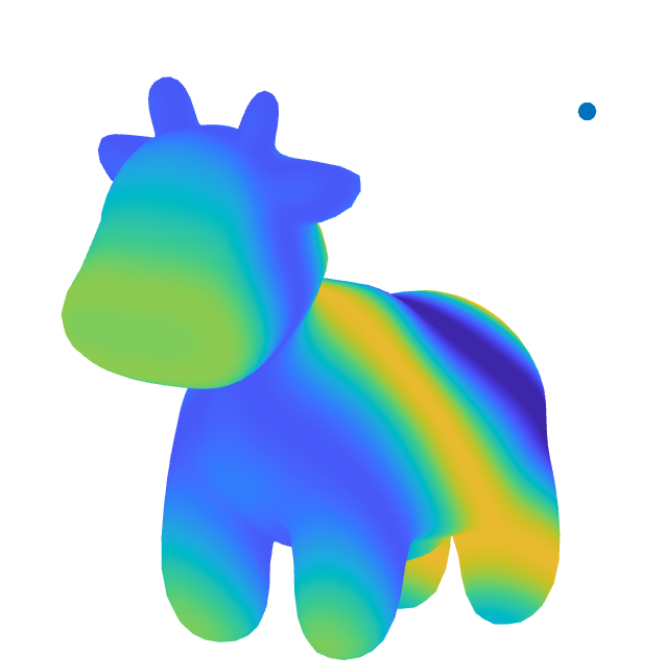} \qquad & \qquad
\includegraphics[trim=145 210 140 230,clip,width=.27\textwidth]{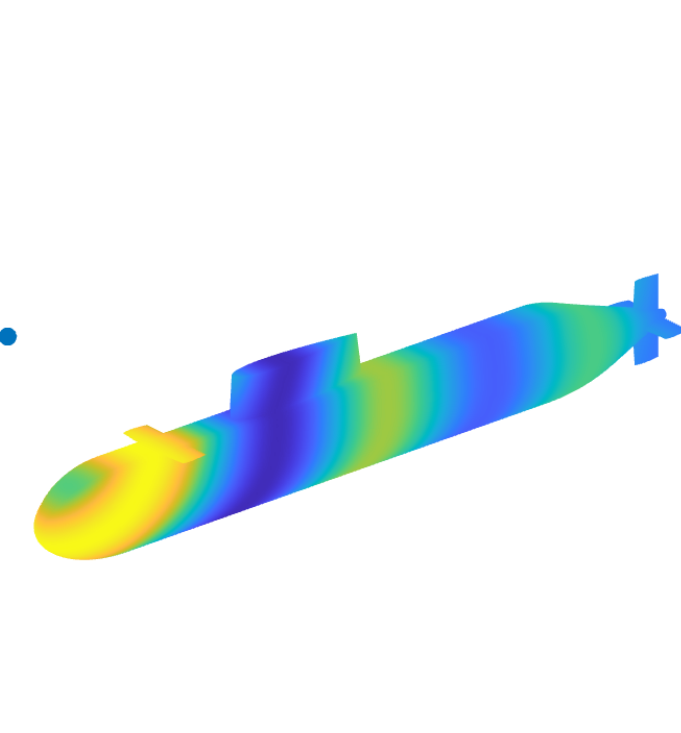}\\
\includegraphics[width=.35\textwidth]{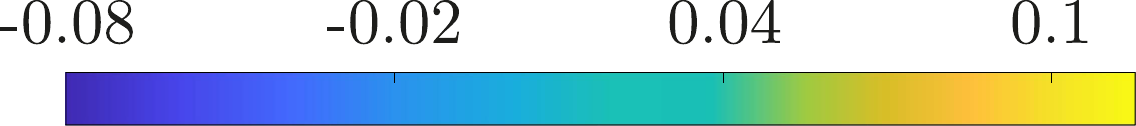} \qquad & \qquad
\includegraphics[width=.35\textwidth]{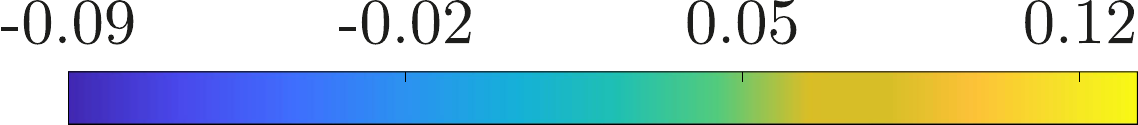} \\
\includegraphics[trim=135 240 165 230,clip,width=.27\textwidth]{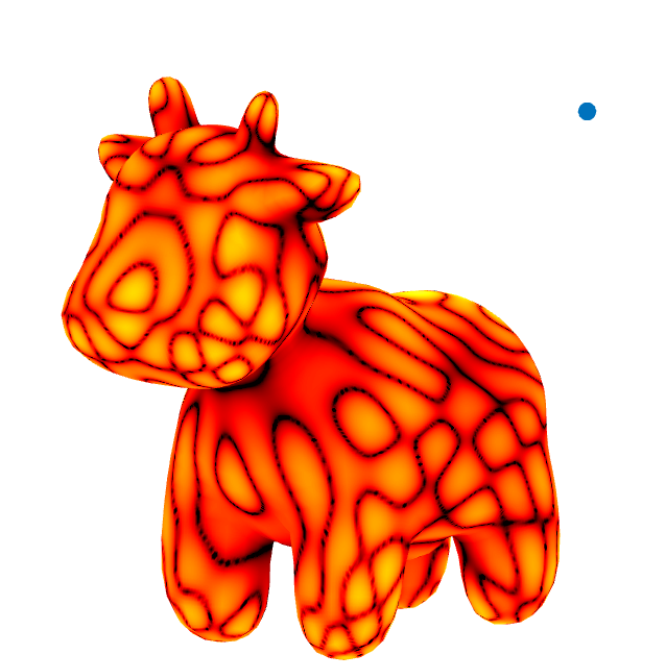} \qquad & \qquad
\includegraphics[trim=145 210 140 230,clip,width=.27\textwidth]{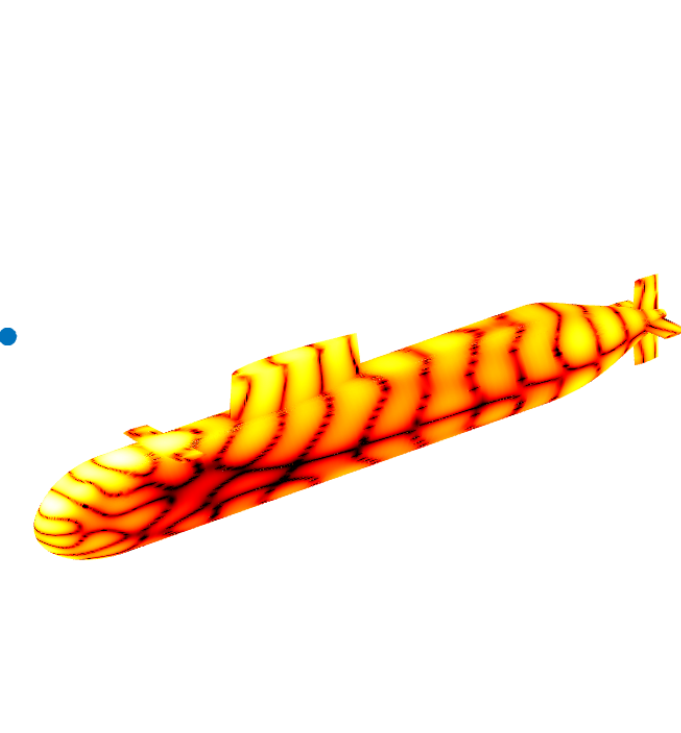}\\
\,\,\includegraphics[width=.35\textwidth]{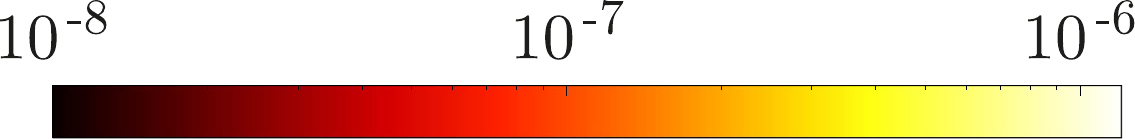} \qquad & \qquad \,\,
\includegraphics[width=.35\textwidth]{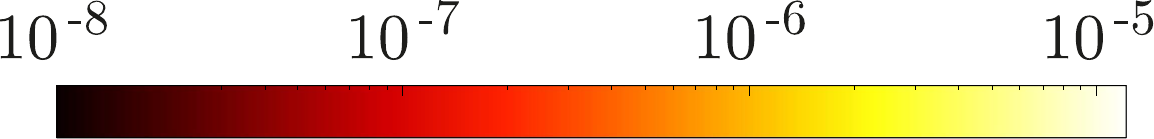} \\
\includegraphics[trim=135 240 165 230,clip,width=.27\textwidth]{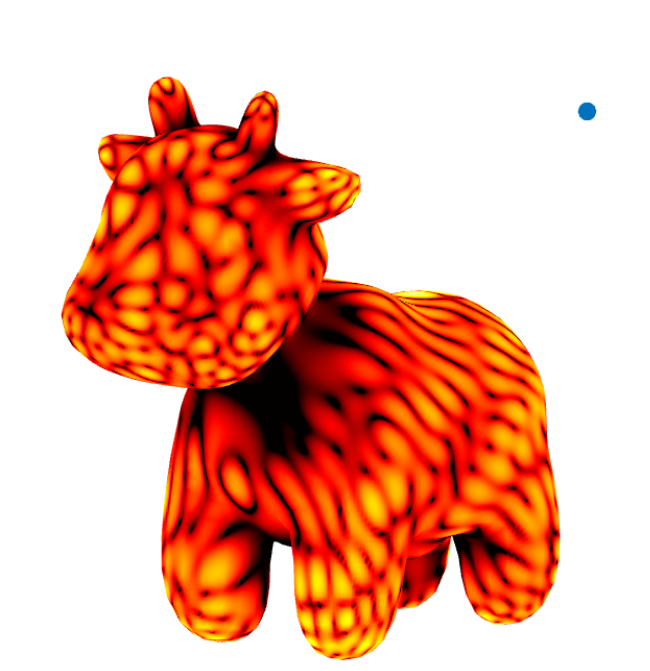} \qquad & \qquad
\includegraphics[trim=145 210 140 230,clip,width=.27\textwidth]{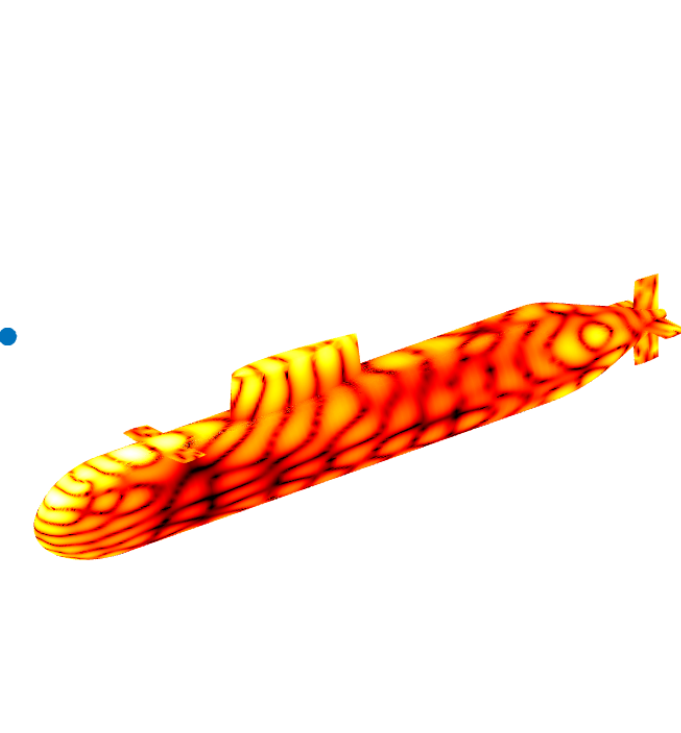}\\
\includegraphics[width=.35\textwidth]{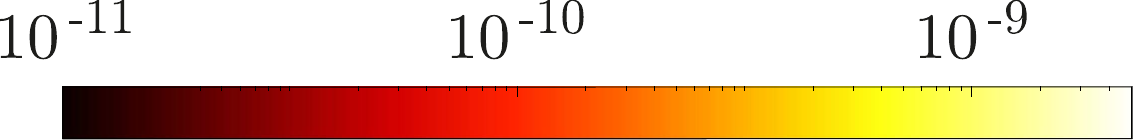} \qquad & \qquad
\includegraphics[width=.35\textwidth]{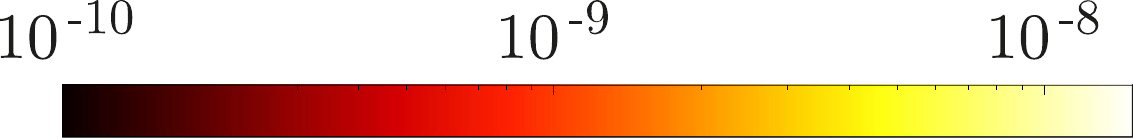}
\end{tabular}
\caption{Real part of the fundamental solution (\ref{fundamental solution}) with $\kappa=10$ and $\mathbf{s} \in \mathbb{R}^3 \setminus \overline{\Omega}$, marked by the blue dot in the plots, such that $\text{dist}(\mathbf{s}, \Omega) \approx \lambda$, along with approximation errors on $\partial \Omega$ using $P=13 \times 10^3$ plane waves, either PPWs in $\Phi_P$ from (\ref{plane waves approximation set}) or EPWs in $\Phi_{L,P}$ from (\ref{isomorphism approximation sets}). The left column shows results for the cow geometry, the right for the submarine. In each column, the first row shows the real part of the solution on $\partial \Omega$, the second the PPW approximation error, and the third the EPW approximation error.}
\label{figure 6.7}
\end{figure}

\paragraph{Cow.}
We now consider a domain $\Omega$ defined by a closed triangulated surface in the shape of a cow. The mesh, consisting of $2930$ triangular elements, was taken from \cite{Crane2013}. Since the mesh triangles are irregular and vary in size, particular care is needed in placing approximately equispaced sampling points on the surface.
To this end, we adopt the following strategy: each triangle is subdivided adaptively based on its diameter, so as to enforce a roughly uniform density of points over the entire surface. Specifically, a reference subdivision pattern is applied to each triangle after estimating how many points are required to match a target density, which depends on the global area of $\partial \Omega$ and the desired total number of sampling points $S$.

We then proceed exactly as in the cubic case: we evaluate the fundamental solution \eqref{fundamental solution} at the selected boundary points and compute an approximation using either EPWs or PPWs, solving the resulting oversampled system via a regularized SVD.
In this setting, with a fixed number of plane waves $P=13\times 10^3$, the runtime is nearly identical for PPW and EPW approximations, i.e.\ $395$s for PPWs and $409$s for EPWs, respectively. This indicates that constructing the EPW approximation set is as fast as constructing the PPW one. Moreover, over 92$\%$ of the total runtime is spent computing the SVD in both cases, confirming that the cost of building the approximation sets is negligible in comparison.

In Figure \ref{figure 6.7}, the left column shows the real part of the fundamental solution \eqref{fundamental solution} restricted to the boundary of the domain, along with the corresponding absolute error for both PPW and EPW approximations in the cow geometry.

\paragraph{Submarine.}
Finally, we consider a submarine-shaped domain, discretized via a triangular surface mesh with $44596$ elements, obtained from \cite{Venas2019}. The boundary $\partial \Omega$ exhibits fine-scale geometric features, such as tail fin, requiring even finer sampling to ensure accurate approximation.
As in the cow case, approximately equispaced boundary sampling is obtained by subdividing each triangle adaptively, according to its surface area and the desired global point density.
The approximation again employs \(P=13\times10^3\) plane waves, and the resulting runtimes -- $1037$s for PPWs and $1162$s for EPWs -- are comparable. This confirms that enriching the basis with evanescent waves has no appreciable impact on the overall cost. Moreover, in both cases, the SVD-based solver dominates the total runtime, reaffirming that the construction of the approximation sets requires only a small fraction of the overall computational effort.

In the right column of Figure \ref{figure 6.7}, the real part of the fundamental solution \eqref{fundamental solution}, as well as the absolute errors of the PPW and EPW approximations, are displayed on the boundary of the submarine geometry.

These results highlight the potential of the proposed EPW sampling algorithm for plane wave approximations and Trefftz schemes, particularly since it is not optimized for non-spherical geometries (except for the $L^{\infty}$ normalization at the boundary).
We are confident that our numerical recipe could be refined by defining rules tailored to the specific underlying geometries, thereby paving the way for even more effective approximation strategies.

\section{Conclusions}

This paper extends the analysis of plane wave approximation properties
from 2D to 3D.
As expected, also in 3D PPWs are not suited for stably approximating high-frequency
Fourier modes, whose integral representation as a continuous superposition of
PPWs features a density function that blows up with the mode number. 
This is reflected in large expansion coefficients associated with finite PPW
sets and constitutes the fundamental source of numerical instability in
standard plane-wave based Trefftz schemes.
Conversely, any Helmholtz solution in the unit ball can be exactly represented as a continuous superposition of EPWs with a unique bounded density function.
We propose a numerical strategy based on a sampling approach for the construction of finite EPW approximation sets. Given a fixed number of plane waves, the computational effort to construct both PPW and EPW approximation sets is similar. Nevertheless, EPW sets offer a clear advantage in terms of accuracy, yielding much more precise approximations.

The current form of the numerical recipe is tailored to a single (spherical) cell, and should be regarded as a preliminary step toward broader applications.
While numerical experiments are encouraging, further developments are needed to address more general geometries and to incorporate EPWs into full Trefftz Discontinuous Galerkin formulations.
Additional evidence supporting this strategy in 2D is provided in~\cite{Robert2024}, where the EPW-recipe~\cite{parolin-huybrechs-moiola} is successfully integrated into the Ultra Weak Variational Formulation (UWVF)~\cite{Cessenat1998}, showing its effectiveness within a complete Trefftz framework.

\section*{Acknowledgements}
AM acknowledges support from PRIN projects ``ASTICE'' and ``NA-FROM-PDEs'', GNCS--INDAM, and PNRR-M4C2-I1.4-NC-HPC-Spoke6, funded by the European Union - Next Generation EU.

\pagebreak
\addcontentsline{toc}{section}{References}
\printbibliography

\end{document}